\documentclass[10pt, a4paper, UKenglish, twoside]{article}
\usepackage[UKenglish]{babel}
\usepackage[vmargin=40mm,inner=25mm,outer=30mm,headheight=14pt]{geometry}
\usepackage{bbm}
\usepackage{appendix}
\usepackage[utf8]{inputenc}
\usepackage[T1]{fontenc}
\usepackage{fancyhdr}
\usepackage{amsmath} 
\usepackage[final]{showkeys}
\usepackage{amssymb}
\usepackage{paralist}
\usepackage{hyperref}
\usepackage{amsthm}
\usepackage{color}
\usepackage{mathtools}
\numberwithin{equation}{section}
\usepackage[pdftex]{graphics}
\usepackage{graphicx}
\usepackage{accents}
\usepackage{tikz}
\usetikzlibrary{decorations.pathmorphing}
\tikzset{wiggly/.style={decorate, decoration=snake}}
\usepackage{dsfont}

\theoremstyle{plain}
\usepackage[mathscr]{eucal}
\usepackage{bbold}
\newtheorem{lemma}{Lemma}[section]
\newtheorem{corollary}[lemma]{Corollary}
\newtheorem{theorem}[lemma]{Theorem}

\newtheorem{proposition}[lemma]{Proposition}
\theoremstyle{definition}
\newtheorem{define}[lemma]{Definition}
\newtheorem{remark}[lemma]{Remark}
\newtheorem{conjecture}[lemma]{Conjecture}

\newcommand{\discgamma}{{\bar{\gamma}}}
\newcommand{\discmu}{{\bar{\mu}}}
\newcommand{\discsigma}{\mathcal{F}}

\renewcommand{\t}{\boldsymbol{t}}
\newcommand{\Pcal}{\mathcal{P}}
\newcommand{\overcirc}{\accentset{\circ}}
\newcommand{\extgamma}{{\tilde{\gamma}}}
\newcommand{\extmu}{{\tilde{\mu}}}
\newcommand{\extsigma}{\mathcal{A}}
\newcommand{\gradsigma}{\mathcal{E}}
\newcommand{\Int}{\mathrm{int}}
\newcommand{\Ext}{\mathrm{ext}}
\renewcommand{\d}{\mathrm{d}}
\newcommand{\sd}{\mathrm{sd}}
\renewcommand{\k}{\kappa}

\newcommand{\R}{{\mathbb{R}}}
\renewcommand{\P}{{\mathbb{P}}}
\newcommand{\E}{{\mathbb{E}}}

\newcommand{\Q}{{\mathbb{Q}}}
\newcommand{\myF}{{\mathcal{F}}}
\newcommand{\Z}{{\mathbb{Z}}}

\newcommand{\mc}{\mathcal}

\newcommand{\bs}{\boldsymbol}

\newcommand{\p}{\varphi}
\newcommand{\crm}{\mathrm{c}}
\newcommand{\ST}{\mathrm{ST}}
\newcommand{\Edge}{\mathbf{E}}
\newcommand{\Plaq}{\mathbf{P}}
\newcommand{\Vertex}{\mathbf{V}}
\newcommand{\vertiii}[1]{{\left\vert\kern-0.25ex\left\vert\kern-0.25ex\left\vert
#1 
    \right\vert\kern-0.25ex\right\vert\kern-0.25ex\right\vert}}

\AtBeginDocument{%
   \def\MR#1{}
}

\def \thick {0.1cm}
\begin{document}

\title{Phase transitions for a class of gradient fields}
\author{Simon Buchholz\footnote{Institute for Applied Mathematics, University of Bonn,
Endenicher Allee 60, 53115 Bonn\newline E-mail: buchholz@iam.uni-bonn.de}}
\maketitle
\begin{abstract}
We consider gradient fields on $\Z^d$ for potentials $V$ that can be expressed as
\begin{align*}
e^{-V(x)}=pe^{-\frac{qx^2}{2}}+(1-p)e^{-\frac{x^2}{2}}.
\end{align*}
This representation allows us
to associate a random conductance type model
to the gradient fields with zero tilt.
 We investigate this random conductance model and prove correlation inequalities, duality properties, and uniqueness of the Gibbs measure in certain regimes. 
 Moreover, we show that there is a close relation between Gibbs measures of the random conductance model and gradient Gibbs measures with zero tilt for the potential $V$. 
Based on these results we can give a new proof for the non-uniqueness of gradient Gibbs measures without using reflection positivity. We also show uniqueness of ergodic zero tilt gradient Gibbs measures
for almost all values of $p$ and $q$ and, in dimension $d\geq 4$, for $q$ close to one or for $p(1-p)$ sufficiently small.

{\footnotesize{\bf 2010 Mathematics Subject Classification.} 82B05, 82B26, 82B20\newline
}
\end{abstract}

\section{Introduction}\label{sec:introduction}
Gradient fields are a statistical mechanics model that can be used to model phase separation 
or, in the case of vector valued fields, solid materials. 
Formally they can be defined as a random field $(\p_x)_{x\in \Z^d}\in \R^{\Z^d}$ with distribution
\begin{align}\label{eq:formal_grad_int_def}
\frac{\exp\left(-\sum_{x\sim y} V(\p(x)-\p(y))\right)}{Z}\prod_{x\in\Z^d} \d\p(x).
\end{align}
Here $\d\p(x)$ denotes the Lebesgue measure, $V:\R\to\R$ a measurable symmetric potential, and $\sim$ indicates the neighbourhood relation for $\Z^d$.
We can give a meaning to the formal expression \eqref{eq:formal_grad_int_def}
using the DLR-formalism. The DLR-formalism defines
equilibrium distributions usually called Gibbs measure for this type of models
as measures $\mu$ on $\R^{\Z^d}$ such that
the conditional probability of the restriction to any finite set is as above.
In the setting of gradient interface models no Gibbs measure exists in dimension $d\leq 2$.
Therefore one often considers gradient Gibbs measures \cite{MR1463032,MR0094682}.
This means that attention is restricted to the $\sigma$-algebra generated by the gradient fields
\begin{align}
\eta_{xy}=\p(y)-\p(x) \quad \text{for $x\sim y$.}
\end{align} 
Then infinite volume measures exist if $V(s)$ grows sufficiently fast (linearly is sufficient) as $s\to\pm\infty$. 
Gradient Gibbs measures are also useful to model tilted surfaces. 
For a translation invariant gradient Gibbs measure $\mu$ the tilt vector $u\in \R^d$ is defined by
\begin{align}
\E_\mu(\nabla \p(x))=u
\end{align}
 where $\nabla\p(x)\in\R^d$ denotes the discrete derivative, i.e., the vector with entries $\nabla_i\p(x)=\p(x+e_i)-\p(x)$ with $e_i$ denoting the $i$-th standard unit vector. 
 If the gradient Gibbs measure is ergodic the tilt corresponds to the asymptotic average inclination of almost every realisation of the gradient field.
  
 Gradient interface models have been studied frequently in the past years.
 In particular the discrete Gaussian free field with $V(s)=s^2$ where the fields are Gaussian caught considerable attention.
 Many of the results obtained in this case were generalized to the class of strictly convex potentials satisfying $c_1\leq V''(s)\leq c_2$ for some $0<c_1<c_2$ and all $s\in \R$.
 Let us only mention two results for convex potentials and refer to the literature in particular the reviews 
 \cite{MR2228384,MR2251117} for all further results and references. Funaki and Spohn showed in \cite{MR1463032} that
 for every tilt vector $u$ there exists a unique translation invariant gradient Gibbs measure. 
 Moreover, the scaling limit of the model is a massless Gaussian field
  as shown by Naddaf and Spencer \cite{MR1461951} for zero tilt and generalised to arbitrary tilt by Giacomin, Olla, and Spohn \cite{MR1872740}.
In contrast for non-convex potentials far less is known because all the techniques seem to rely on convexity in an essential way. For potentials 
of the form $V=U+g$ where $U$ is strictly convex and $g''\in L^q$ for some $q\geq 1$ 
with sufficiently small norm the problem can be led back to the convex theory
by integrating out some degrees of freedom. This way many results from the convex case can be proved in particular uniqueness and existence of the Gibbs measure for every tilt and that
the scaling limit is Gaussian \cite{MR2470934, MR2976565, MR3913274}. This corresponds to a high temperature result. 
For low temperatures which correspond to non-convexities far away from the minimium of $V$ it was shown that the surface tension is strictly convex and the scaling limit is Gaussian \cite{adams2016strict, hilger2016scaling}.

For intermediate temperatures that correspond to very non-convex potentials no robust techniques are known. 
All results to date are restricted to the special
class of potentials introduced by Biskup and Kotecky in \cite{MR2322690} that can be represented as
\begin{align}\label{eq:pot_general}
e^{-V(x)}=\int_{\R_+} e^{-\frac{\kappa x^2}{2}}\,\rho(\d \kappa)
\end{align}
where $\rho$ is a non-negative Borel measure on the positive real line.
Biskup and Kotecky mostly considered the simplest nontrivial case, denoting the Dirac measure at
$x\in \R$ by $\delta_x$,
\begin{align}\label{eq:def_rho}
\rho=p\delta_q+(1-p)\delta_1
\end{align}
where $p\in [0,1]$ and $q\geq1$.
They show that in dimension $d=2$ and for $q>1$ sufficiently large there exist two ergodic zero-tilt gradient Gibbs measures. Later,
Biskup and Spohn showed in \cite{MR2778801} that nevertheless the scaling limit of every zero-tilt gradient Gibbs measure is Gaussian  if the measure $\rho$ is compactly supported in $(0,\infty)$. 
In \cite{MR3982951} their result was recently extended by Ye to potentials 
of the form $V(s)=(1+s^2)^\alpha$ with $0<\alpha<\frac12$. Those potentials  can be expressed as in 
\eqref{eq:pot_general} but $\rho$ has unbounded support so that the results from \cite{MR2778801} do not directly apply.

The main reason to study this class of potentials is that such potentials are much more tractable  because the variable  $\kappa$ can be considered as an additional degree of freedom  using the representation \eqref{eq:pot_general}.
This leads to extended gradient Gibbs measures
which are given by the joint law of $(\eta_{e},\kappa_{e})_{e\in\Edge(\Z^d)}$.  
These extended gradient Gibbs measures can be  represented as a mixture of non-homogeneous 
Gaussian fields with bond potential $\kappa_{e}\eta^2/2$ for every edge $e\in \Edge(\Z^d)$
and $\kappa_{e}\in \R_+$. This implies that for a given $\kappa$ the distribution of the random field is Gaussian with covariance given by the inverse of the operator
$\Delta_\kappa$ where
\begin{align}
\Delta_\kappa f(x)=\sum_{y\sim x} \kappa_{\{x,y\}}(f(x)-f(y)).
\end{align} 
In all the works mentioned before this structure is frequently used, e.g. 
in \cite{MR2778801} it is proved that the resulting $\kappa$-marginal of the extended gradient Gibbs measure is ergodic so that well known homogenization results for random walks in ergodic environments can be applied.

The main purpose of this
 note is to investigate the properties of the $\kappa$-marginal of extended gradient Gibbs measures in a bit more detail.  
The starting point is the observation  that the $\kappa$-marginal of an extended gradient Gibbs measures with zero tilt  is itself a Gibbs measure for a certain specification. This specification arises as the infinite volume of an infinite range  random conductance model defined on finite graphs. 
On the other hand, we show that starting from a Gibbs measure for the random conductance model we can construct a zero tilt gradient Gibbs measure thus showing a one to one relation between the two notions of  Gibbs measures. In particular,
 we can lift results about the random conductance model to results about gradient Gibbs measures.
Note that one major drawback is the restriction to zero tilt that applies here and to all earlier results for this model. 
Let us mention that massive
$\R$-valued random fields  have been earlier connected to discrete percolation models to analyse the existence of phase transitions  \cite{MR1766352}. For gradient models the setting is slightly different because we consider a random conductance model on the bonds with long ranged correlations while for massive models one typically considers some type of site percolation with quickly decaying correlations.

  The main motivation for our analysis is that it provides a first step to the completion of the phase diagram for this potential and zero tilt and a better understanding of the two coexisting Gibbs states. Moreover, the random conductance model appears to be interesting in its own right.
We could define the random conductance model and prove several of the results for arbitrary $\rho$ but we mostly restrict our analysis to the simplest case where $\rho$ is as in \eqref{eq:def_rho} and the potential is of the form
\begin{align}\label{eq:pot_special}
e^{-V_{p,q}(x)}=pe^{-\frac{qx^2}{2}}+(1-p)e^{-\frac{x^2}{2}}.
\end{align} 
 We prove several results about the random conductance model in particular correlation inequalities
 (that extend to arbitrary $\rho$). One helpful observation is that the random conductance model is closely related to determinantal processes  because its definition involves a determinant weight. This simplifies several of the proofs because all correlation inequalities can be immediately led back to similar results for the weighted spanning tree. 
Using the correlation inequalities it is possible to show uniqueness of its Gibbs measure in certain regimes.

It was already observed in \cite{MR2322690} that the gradient interface model with potential $V_{p,q}$ 
exhibits a duality property when defined on the torus. 
Moreover, there is a self dual point $p_{\sd}=p_{\sd}(q)\in (0,1)$ where the model agrees with its own dual.
The self dual point satisfies the equation
\begin{align}\label{eq:self_dual_equation}
\left(\frac{p_{\sd}}{1-p_{\sd}}\right)^4=q.
\end{align}
In \cite{MR2322690} it is shown that the location of the phase transition in $d=2$ must be the self dual point.

 We extend the duality to the random conductance model and arbitrary  planar graphs.
 Using the fact that $\Z^2$ as a graph is self-dual we can use the duality to prove non-uniqueness of 
 the Gibbs measure therefore reproving the result from \cite{MR2322690} without the use of reflection positivity.
 Many of our techniques and results for the random conductance model originated in the study of the random cluster model and we conjecture further similarities. 

This 
paper
 is structured as follows. In Section \ref{sec:model} we give a precise definition of gradient Gibbs measures and state our main results. Then, in Section \ref{sec:model_intro} we introduce and motivate
the random conductance model and its relation to extended gradient Gibbs measures. We prove properties of the random conductance model in Sections \ref{sec:percolation} and \ref{sec:further_prop}.
Finally, in Section \ref{sec:phase_transition} we use the duality of the model to reprove the phase transition result. 
Two technical proofs and some results about regularity properties of discrete elliptic equations are delegated to appendices.

\section{Model and main results}\label{sec:model}

\paragraph{Specifications.} Let us briefly recall the definition of a specification because the concept will be needed in full generality for the random conductance model (see Section \ref{sec:percolation}). We consider a countable set $S$ (mostly $\Z^d$ or the edges of $\Z^d$)
and a measurable state space $(F,\myF)$ (mostly either $|F|=2$ or $(F,\myF)=
(\R,\mathcal{B}(\R))$). Random fields are probability measures on $(F^S,\myF^S)$
where $\myF^S$ denotes the product $\sigma$-algebra. The set of probability measures on a measurable space $(X,\mathcal{X})$ will be denoted by $\mathcal{P}(X,\mathcal{X})$.
For any $\Lambda\subset S$ we denote by $\pi_\Lambda:F^S\to F^\Lambda$ the canonical projection. We often consider the $\sigma$-algebra $\myF_\Lambda=\pi_\Lambda^{-1}(\myF^\Lambda)$ of events depending on the set $\Lambda$.
Recall that a probability kernel $\gamma$ from $(X,\mathcal{B})$ to $(X,\mathcal{X})$, where
$\mathcal{B}\subset \mathcal{X}$ is a sub-$\sigma$-algebra,
is called proper if $\gamma(B,\cdot)=\mathbb{1}_B$ for $B\in \mathcal{B}$.
 
 \begin{define}\label{def:spec}
 	A specification is a family of proper probability kernels
 	$\gamma_\Lambda$ from $\myF_{\Lambda^\crm}$ to $\myF_{S}$ indexed by finite subsets $\Lambda\subset S$ such that $\gamma_{\Lambda_1}\gamma_{\Lambda_2} =\gamma_{\Lambda_1}$ if $\Lambda_2\subset \Lambda_1$. We define the set of random fields admitted to $\gamma$ by
 	\begin{align}\label{eq:def_admitted}
 	\mathcal{G}(\gamma)=\{\mu\in \mathcal{P}(F^S,\myF_S): \mu(A|\myF_{\Lambda^c})(\cdot)=\gamma_{\Lambda}(A|\cdot) \text{ $\mu$-a.s. for all
 	$A\in \myF_S$ and $\Lambda\subset S$ finite}\}.
 	\end{align}
 \end{define}
\begin{remark}\label{remark:Gibbs_specification}
There is a well known equivalent definition of Gibbs measures.
A cofinal set $I$ is a subset of subsets of $S$ with the property that for 
any finite set $\Lambda_0\subset S$ there is $\Lambda\in I$ such that $\Lambda_0\subset \Lambda$.
Then $\mu\in \mathcal{G}(\gamma)$ if and only if  $\mu\gamma_{\Lambda}=\mu$ for $\Lambda\in I$ where $I$ is a cofinal subset of subsets of $S$. 
See Remark 1.24 in \cite{MR2807681} for a proof.
\end{remark}

\paragraph{Gradient Gibbs measures.}
We introduce the relevant notation and the definition of Gibbs and gradient Gibbs measures to state our results. For a broader discussion see \cite{MR2807681, MR2251117}.
In this paragraph we consider real valued random fields indexed by a lattice $\Lambda \subset \Z^d$. We will denote the set of nearest neighbour bonds of $\Z^d$ by $\Edge(\Z^d)$. More generally, we will write 
$\Edge(G)$ and $\Vertex(G)$ for the edges and vertices of a graph $G$.
To consider gradient fields it is useful to choose on orientation of the edges.
We orient the edges $e=\{x,y\}\in \Edge(\Z^d)$
from $x$ to $y$ iff $x\leq y$ (coordinate-wise), i.e., we can view the graph $(\Z^d,\Edge(\Z^d))$ as a directed graph but
mostly we work with the undirected graph.

To any random field $\p:\Z^d\to \R$ we associate the 
gradient field $\eta=\nabla \p\in \R^{\Edge(\Z^d)}$ given by $\eta_{e}=\p_y-\p_x$ if $\{x,y\}\in \Edge(\Z^d)$
are nearest neighbours and $x\leq y$. We formally write $\eta_{x,y}=\eta_e=\p_y-\p_x$
and $\eta_{y,x}=-\eta_e=\p_x-\p_y$.
The gradient field $\eta$ 
 satisfies the plaquette condition
\begin{align}\label{eq:plaquette}
\eta_{x_1,x_2}+\eta_{x_2,x_3}+\eta_{x_3,x_4}+\eta_{x_4,x_1}=0
\end{align}
for every plaquette, i.e., nearest neighbours $x_1,x_2,x_3,x_4,x_1$.
Vice versa, given a field $\eta\in\R^{\Edge(\Z^d)}$ that satisfies the plaquette condition 
there is a up to constant shifts a unique field $\p$ such that $\eta=\nabla \p$ (the antisymmetry
of the gradient field is contained in our definition).
We will refer to those fields as gradient fields and denote them by
$\R^{\Edge(\Z^d)}_g$.
To simplify the notation we write $\p_\Lambda$ for $\Lambda\subset \Z^d$ and $\eta_{E}$
for $E\subset \Edge(\Z^d)$
for the the restriction of fields and gradient fields.
We usually identify a subset   $\Lambda\subset \Z^d$ with the graph generated by it and
as before we write  $\Edge(\Lambda)$ for the bonds with both endpoints in
$\Lambda$. 

For a subgraph $H\subset G$ we write $\partial H$ for the (inner) boundary of $H$ consisting of all points
$x\in \Vertex(H)$ such that there is an edge $e=\{x,y\}\in \Edge(G)\setminus \Edge(H)$.
In the case of a graph generated by $\Lambda\subset G$ we have $x\in \partial\Lambda$ if
there is $y\in \Lambda^{\crm}$ such that $\{x,y\}\in \Edge(G)$.
We define $\accentset{\circ} \Lambda=\Lambda\setminus \partial\Lambda$. 
For a finite subset $\Lambda\subset \Z^d$ we denote by $\d\p_\Lambda$ the Lebesgue measure on $\R^\Lambda$. 
We define for $\omega\in \R^{\Edge(\Z^d)}_{g} $  and $\Lambda$ finite and simply connected (i.e., $\Lambda^\crm$ connected) the following a priori measure on gradient configurations
\begin{align}\label{eq:apriori_measure}
\nu_{\Lambda}^{\omega_{\Edge(\Lambda)^{\crm}}}(\d \eta)
=\nabla_\ast\left(\prod_{x\in \overcirc\Lambda^{\crm}} \delta_{\tilde\p(x)}(\cdot)\;\d\p_{\overcirc\Lambda}\right)
\end{align} 
where $\tilde\p$ is a configuration such that $\nabla\tilde\p=\omega$ and $\nabla_\ast$ the push-forward of this measure along the gradient map $\nabla:\R^{\Z^d}\to \R^{\Edge(\Z^d)}_g$.
The shift invariance of the Lebesgue measure implies that this definition is independent of the choice of $\tilde\p$ and it only depends on the restriction $\omega_{\Edge(\Lambda)^{\crm}}$ since 
$\Lambda^\crm$ is connected.
For a potential $V:\R\to\R$ satisfying some growth condition we define the specification $\gamma_\Lambda$
\begin{align}\label{eq:def_Gibbs_specification}
\gamma_\Lambda(\d \eta,\omega_{\Edge(\Lambda)^{\crm}})=
\frac{\exp\left(-\sum_{e\in \Edge(\Lambda)}V(\eta_e)\right)}{Z_\Lambda(\omega_{\Edge(\Lambda)^{\crm}})}\nu_\Lambda^{\omega_{\Edge(\Lambda)^{\crm}}}(\d \eta)
\end{align}
where the constant $Z_\Lambda(\omega_{\Edge(\Lambda)^{\crm}})$ ensures the normalization of the measure.
We introduce the notation $\gradsigma_{E}=\pi_E^{-1}(\mc{B}(\R)^E)$ for $E\subset \Edge(\Z^d)$ for the $\sigma$-algebra 
of events depending only on $E$.
Measures that are admitted to the specification $\gamma$, i.e., measures $\mu$ that satisfy
for simply connected $\Lambda\subset\Z^d$
\begin{align}
\mu(A\mid \gradsigma_{\Edge(\Lambda)^{\crm}})(\cdot)=\gamma_\Lambda(A,\cdot)\qquad \text{$\mu$ a.s.}
\end{align}
 will be called gradient Gibbs measures for the potential $V$.

For $a\in \Z^d$ we consider the shift $\tau_a:\R^{\Edge(\Z^d)}\to \R^{\Edge(\Z^d)}$ that is defined by 
\begin{align}\label{eq:shift}
(\tau_a\eta)_{x,y}=\eta_{x+a,y+a}.
\end{align} 
  A measure is translation invariant if $\mu (\tau_a^{-1}(A))=\mu(A)$ for all $a$ and $A\in \mathcal{B}(\R)^{\Edge(\Z^d)}$.
An event is translation invariant if $\tau_a(A)=A$ for all $a\in \Z^d$.
A gradient measure is ergodic if $\mu(A)\in \{0,1\}$ for all translation invariant $A$.
  
\paragraph{Main results.}
Our first main result is the following almost always uniqueness result for the gradient Gibbs measures for potentials as in \eqref{eq:pot_special}.
\begin{theorem}\label{th:main_unique}
For every $q$ and $d\geq 2$ 
there is an at most countable set $N(q,d)\subset [0,1]$ such that
for any $p\in [0,1]\setminus N(q,d)$  there is a unique shift invariant ergodic gradient Gibbs measure
$\mu$ with zero tilt 
 for the potential $V_{p,q}$. 
\end{theorem}
This theorem is proved in Section \ref{sec:further_prop} below the proof of Theorem \ref{th:countable_non_uniqueness}.
Moreover, we reprove the non-uniqueness result originally shown in \cite{MR2322690} for this type of potential.
\begin{theorem}\label{th:main_non_unique}
There is $q_0\geq 1$ such that for $d=2$, $q\geq q_0$, and $p=p_{\sd}(q)$ the solution 
of \eqref{eq:self_dual_equation}, there are at least two shift invariant gradient Gibbs measures with 0 tilt.
\end{theorem}
The proof of this theorem is given at the end of Section \ref{sec:phase_transition}.
Moreover we prove uniqueness for 'high temperatures' and dimension $d\geq 4$. This corresponds to the regime where the Dobrushin condition holds. 
\begin{theorem}\label{th:Dobrushin}
Let $d\geq 4$. For any $q\geq 1$ there
exists $p_0=p_0(q,d)>0$ such that for  all
$p\in [0,p_0)\cup (1-p_0,1]$ there is a unique shift invariant ergodic gradient Gibbs measure with zero tilt for the potential $V_{p,q}$.
Moreover, there exists  
$q_0=q_0(d)>1$ such that for any $q\in [1,q_0]$ and any $p\in [0,1]$  there is a unique shift invariant ergodic gradient Gibbs measure with zero tilt for the potential $V_{p,q}$. 
\end{theorem}
The proof of this Theorem is given in Section \ref{sec:further_prop} below the proof of Theorem \ref{th:Dobrushin_discrete}.

The main tool in the proofs of these theorems is the fact that the structure  of the potentials 
$V$ in \eqref{eq:pot_general} allows us to consider $\kappa$ as a further degree of freedom and we consider the joint distribution of the gradient field $\eta$ and $\kappa$.
We show that the law of the $\kappa$-marginal can be related to a random conductance model.
The analysis of this model then translates back into the theorems stated before. We will make those statements precise in the next section.
Let us end this section with some remarks.
\begin{remark}
\begin{enumerate}
\item For spin systems with finite state space and bounded interactions there are general results that show that phase transitions, i.e., non-uniqueness of the Gibbs measure are rare, see, e.g., \cite{MR2807681}. 
Theorem \ref{th:main_unique} establishes a similar result for a specific class of potentials
for a unbounded spin space. 
As discussed in more detail at the end of Section \ref{sec:further_prop} we expect that for every $q\geq 1$ the Gibbs measure is unique for all $p\in [0,1]$ except possibly for $p=p_c$ for some critical value $p_c=p_c(q)$.  Hence, Theorem \ref{th:main_unique} is far from optimal 
but we hope that the results provided in this paper prove useful to establish stronger results.
\item Let us compare the results to earlier results in the literature.
For $p/(1-p)<1/q$ the potential $V_{p,q}$ is strictly convex so that uniqueness of the Gibbs measure is well known and holds for every tilt. 
The two step integration used by Cotar and Deuschel extends the uniqueness
result to the regime $p/(1-p)<C/\sqrt{q}$ (see Section 3.2 in \cite{MR2976565}).
In particular the case $p\in [0,p_0)$ in Theorem \ref{th:Dobrushin} is included in earlier results.
However, the potential becomes very non-convex (has a very negative second derivative at some points) for $p$ close to $1$ and the  uniqueness result for $p\in (1-p_0,1]$ and $d\geq 4$ appears to be new. 
In this regime the only known result seems to be convexity of the surface tension 
as a function of the tilt
which was shown in  \cite{adams2016strict} (see in particular Proposition 2.4 there). Their results
apply to  $p$ very close to one, $q-1$ very small, and $d\leq 3$.
\item  The restriction to dimension $d\geq 4$ arises from the fact that the Green's function
for inhomogeneous elliptic operators in divergence form decays slower than in the homogeneous case. 
\end{enumerate}
\end{remark}
\section{Extended gradient Gibbs measures and random conductance model}\label{sec:model_intro}
\paragraph{Extended gradient Gibbs measure.}
In this work we restrict to potentials of the form introduced in \eqref{eq:pot_general}. 
  As already discussed in more detail in \cite{MR2322690} and \cite{MR2778801} it is possible to use the special structure of 
  $V$ to raise  $\kappa$ to a degree of freedom. 
  Let $\mu$ be a gradient Gibbs measure for $V$.
  For a finite set $E\subset \Edge(\Z^d)$ and Borel sets $\mathbf{A}\subset \R^E$ and $\mathbf{B}\subset \R_+^E$
  we define the extended gradient Gibbs measure
  \begin{align}\label{eq:def_extended_Gibbs}
  \tilde{\mu}((\eta_b,\kappa_b)_{b\in E}\in \mathbf{A}\times \mathbf{B})
  =\int_{\mathbf{B}} \rho_{E}(\d\kappa) \mathbb{E}_\mu\left(\mathbb{1}_{\mathbf{A}}\prod_{e\in E} e^{-\frac12 \kappa_e\eta_e^2+V(\eta_e)}\right).
  \end{align}
  It can be checked that this is a consistent family of measures and thus we can extend $\extmu$ to a measure on  $(\R\times \R_+)^{\Edge(\Z^d)}$.
 It was  explained in \cite{MR2322690} that  $\tilde{\mu}$ is itself a Gibbs measure for the specification
  $\tilde \gamma_\Lambda$ defined by
  \begin{align}\label{eq:specification_gamma_tilde}
  \tilde{\gamma}_\Lambda((\d \bar\eta,\d \bar\kappa), (\eta,\kappa))
  =\frac{\exp\left(-\frac12 \sum_{e\in \Edge(\Lambda)}\bar\kappa_e\bar\eta_e^2\right)}{Z_\Lambda(\eta_{\Edge(\Lambda)^{\crm}})} \nu_\Lambda^{\eta_{\Edge(\Lambda)^{\crm}}}(\d \bar\eta)\prod_{e\in \Edge(\Lambda)}\rho(\d\bar\kappa_e)
  \prod_{e\in \Edge(\Lambda)^{\crm}}\delta_{\kappa_e}(\d\bar\kappa_e).
  \end{align}
Note that the distribution $(\d\bar\eta,\d\bar\kappa)_{\Edge(\Lambda)}$  actually only depends on $\eta_{\Edge(\Lambda)^{\crm}}$ and is independent of $\kappa$.
Let us add one remark concerning the notation. In this work we essentially consider three strongly related viewpoints of one model. The first viewpoint are gradient Gibbs measures that are measures on 
$\R^{\Edge(\Z^d)}_g$. Thy will be denoted by $\mu$ and the corresponding specification is denoted by $\gamma$. Then there are extended gradient Gibbs measures for a specification $\extgamma$. They are measures on 
$\R_g^{\Edge(\Z^d)}\times \R_+^{\Edge(\Z^d)}$ and will be denoted by $\extmu$.
The $\eta$-marginal of $\extmu$ is a gradient Gibbs measure $\mu$.
Finally there is also the $\kappa$-marginal  of $\extmu$ which is a measure on $\R_+^{\Edge(\Z^d)}$
and will be denoted by $\discmu$. An important result here is that
$\discmu$ is a Gibbs measure for a specification $\discgamma$ if $\rho$ is a measure as in 
\eqref{eq:def_rho}. In this case $\discmu$ is a measure on the discrete space $\{1,q\}^{\Edge(\Z^d)}$.
We expect that this result can be extended to far more general measures $\rho$ but we do not pursue this matter here.
 To keep the notation consistent we denote objects with single spin space $\R$, e.g., gradient Gibbs measures without symbol modifier, objects with single spin space $\{1,q\}$, e.g., the $\kappa$-marginal with a bar, 
and objects with single spin space $\{1,q\}\times \R$, e.g., extended Gibbs with a tilde.
Let us also fix a notation for the corresponding relevant $\sigma$-algebras. 
We write as before $\gradsigma_E$ for the $\sigma$-algebra on $\R^{\Edge(\Z^d)}$ generated by 
$(\eta_e)_{e\in E}$ and we define $\gradsigma=\gradsigma_{\Edge(\Z^d)}$. For the
$\kappa$-marginal we similarly consider the $\sigma$-algebra
$\discsigma_E$ on $\{1,q\}^{\Edge(\Z^d)}$ generated by $(\kappa_e)_{e\in E}$ and we write again $\discsigma=\discsigma_{\Edge(\Z^d)}$.
For the extended space $\R^{\Edge(\Z^d)}\times \{1,q\}^{\Edge(\Z^d)}$ we use the product
$\sigma$-algebra $\extsigma_E=\pi_1^{-1}(\gradsigma_E)\otimes \pi_2^{-1}(\discsigma_E)$.

It was already remarked in \cite{MR2322690} that this setting resembles the situation for the Potts model that can be coupled to the random cluster model via the Edwards-Sokal coupling measure.

\paragraph{The random conductance model.}
As explained before our strategy is to analyse the $\kappa$-marginal of extended gradient Gibbs measures
and  then use the results to deduce properties of the gradient Gibbs measures for $V_{p,q}$. 
The key observation is that the $\kappa$-marginal of extended gradient Gibbs measures is given by the infinite volume limit of a strongly coupled random conductance model.
To motivate the definition of the random conductance model we consider the $\kappa$-marginal of the extended specification $\tilde{\gamma}$ defined in \eqref{eq:specification_gamma_tilde}.
For zero boundary value $\bar{0}\in \R_g^{\Edge(\Z^d)}$ with $\bar{0}_e=0$ and $\lambda\in \{1,q\}^{\Edge(\Z^d)}$ we obtain 
\begin{align}
\tilde\gamma_{\Lambda}\big(\kappa_{\Edge(\Lambda)}=\lambda_{\Edge(\Lambda)} \, ,\, \bar{0}\big)
=\frac{1}{Z} \int \prod_{e\in \Edge(\Lambda)} p^{\mathbb{1}_{\lambda_e=q}}(1-p)^{\mathbb{1}_{\lambda_e=1}}
e^{-\frac12\lambda_e^2\omega_e^2} \,\nu^{\bar{0}_{\Edge(\Lambda)^{\crm}}}_{\Lambda}(\d\omega).
\end{align}
We write $\Lambda^w=\bar{\Lambda}/\partial \Lambda$ for the graph where the entire boundary is collapsed
to a single point (this is called wired boundary conditions and we will discuss this below in more detail).
We denote the lattice Laplacian with conductances $\lambda$ and zero boundary condition outside of $\overcirc\Lambda$ by $\tilde{\Delta}_{\lambda}^{{\Lambda}^w}$, i.e., $\tilde{\Delta}_{\lambda}^{{\Lambda}^w}$ acts on functions $f:\overcirc{\Lambda}\to \R$ by  $\tilde{\Delta}_{\lambda}^{{\Lambda}^w}f(x)
=\sum_{y\sim x} \lambda_{\{x,y\}}(f(x)-f(y))$ where we set  $f(y)=0$ for $y\notin \overcirc\Lambda$. The definition \eqref{eq:specification_gamma_tilde} and an integration by parts 
followed by Gaussian calculus imply then
\begin{align}
\begin{split}\label{eq:motivation_rc_model}
\tilde\gamma_{\Lambda}\big(\kappa_{\Edge(\Lambda)}=\lambda_{\Edge(\Lambda)}\, ,\, \bar{0}\big)
&=\frac{1}{Z} p^{|\{e\in \Edge(\Lambda)\,:\, \lambda_e=q\}|}(1-p)^{|\{e\in \Edge(\Lambda)\,:\, \lambda_e=1\}|}\int e^{-\frac{1}{2}(\p,\tilde{\Delta}_{\lambda}^{\Lambda^w}\p)}
\,\d\p_{\overcirc{\Lambda}}
\\
&=\frac{1}{Z}\frac{p^{|\{e\in \Edge(\Lambda)\,:\, \lambda_e=q\}|}(1-p)^{|\{e\in \Edge(\Lambda)\,:\, \lambda_e=1\}|}}{\sqrt{\det (2\pi)^{-1} \tilde{\Delta}_{\lambda}^{\Lambda^w}}}.
\end{split}
\end{align}
It simplifies the presentation to introduce the random conductance model of interest in a slightly
more general setting.
We consider a finite and connected graph $G=(V,E)$. The combinatorial graph Laplacian $\Delta_c$ associated to set of conductances $c:E\to \R_+$ is defined by
\begin{align}\label{eq:def_weight_laplace}
 \Delta_c f(x)=\sum_{y\sim x} c_{\{x,y\}} (f(x)-f(y)) 
\end{align}
for any function $f:V\to \R$.
Note that we defined the graph Laplacian as a non-negative operator which is convenient
for our purposes and common in the context of graph theory.
In the following we view  the Laplacian $\Delta_c$ as a linear map on the space $H_0=\{f:V\to \R: \sum_{x\in V} f(x)=0\}$ of functions with vanishing average.
We define  $\det \Delta_c$ as the determinant of this linear map.
By the maximum principle the Laplacian is injective on $H_0$, hence $\det \Delta_c>0$.
Sometimes we clarify the underlying graph by writing $\Delta_c^G$.
\begin{remark}
In the general setting it is more natural to let the Laplacian act on $H_0$ instead of fixing a point to 0 as in the definition of $\tilde{\Delta}_{\lambda}^{{\Lambda}^w}$ above where this corresponds to Dirichlet boundary conditions. 
It would also be possible to fix a point $x\in \Vertex(G)$ and consider $\tilde{\Delta}_c^G$
acting on functions $f:\Vertex(G)\setminus \{x_0\}\to \R$ 
defined by $(\tilde\Delta_c^Gf)(x)=\sum_{y\sim x} c_{\{x,y\}} f(x)-f(y)$ for $x\in \Vertex(G)\setminus \{x_0\}$ where we set $f(x_0)=0$.
It is easy to see using, e.g., Gaussian calculus and a change of measure that the determinant of
$\tilde\Delta_c^G$ is independent of $x_0$ and 
\begin{align}\label{eq:relation_Laplacians}
|G|\det \tilde{\Delta}_c^G=\det \Delta_c^G.
\end{align}
\end{remark}
Motivated by \eqref{eq:motivation_rc_model} we fix a real number $q\geq 1$  and consider the following probability
measure on $\{1,q\}^E$  
\begin{align}\label{eq:def_perc_model}
\P^{G,p}(\kappa)=\frac{1}{Z}\frac{p^{|\{e\in E: \kappa_e=q\}|}(1-p)^{|\{e\in E:\kappa_e=1\}|}}{\sqrt{\det \Delta_\kappa}}
\end{align}
where $Z=Z^{G,p}$ denotes a normalisation constant such that $\P^{G,p}$ is a probability measure.
In the following we will often drop $G$ and $p$ from the notation  and we will always suppress $q$.
 We restrict our attention to $q\geq 1$ because by scaling the model with conductances $\{1,q\}$ has the same distribution as a model with conductances $\{\alpha,\alpha q\}$ for $\alpha>0$  so that we can set the smaller conductance to 1.
Let us state a remark concerning the relation to the random cluster model.
\begin{remark}\label{rem:sim_random-cluster}
\begin{enumerate}
\item 
We chose the notation such that the similarity to the random cluster model is apparent. Both models have an a priori distribution given by independent Bernoulli distribution with parameter  $p$  on the bonds that is then correlated by a complicated infinite range interaction depending on $q$. They reduce to Bernoulli percolation
for $q=1$.
At the end of Section \ref{sec:further_prop} we state a couple of conjectures about the behaviour of this model that show that we expect similarities with the random cluster model in many more aspects.
\item While there are several close  similarities to the random cluster model there is also one important difference  that seems to 
pose additional difficulties in the analysis of this model.
The conditional distribution in a finite set depends on the entire configuration of the conductances outside the finite set (not just a partition of the boundary as in the random cluster model).
In particular the often used argument that the conditional distribution of a random 
cluster model in a set given that all boundary edges are closed is the free boundary random cluster distribution has no analogue in our setting.
\item We refer to the model as a random conductance model since we will (not very surprisingly) use tools from the theory of electrical networks. Note that in the definition of the potential $V$ the parameters correspond to  different (random) stiffness of the bonds.
\end{enumerate}
\end{remark}

\section{Basic properties of the random conductance model}\label{sec:percolation}
\paragraph{Preliminaries.}
As before we consider a connected graph $G=(V,E)$.
To simplify the notation we introduce for $E'\subset E$ and $\kappa \in \{1,q\}^E$ the notation
\begin{align}\label{eq:def_h}
h(\kappa,E')&= {|\{e\in E': \kappa_e=q\}|}\\ \label{eq:def_s}
s(\kappa,E')&= {|\{e\in E': \kappa_e=1\}|}
\end{align}
 for the number of hard and soft edges respectively and we define
 $h(\kappa)=h(\kappa,E)$ and $s(\kappa)=s(\kappa,E)$.
 Let us introduce the weight of  a
subset of edges $\t\subset E$ by defining
\begin{align}\label{eq:def_weight}
w(\kappa,\t)=\prod_{e\in \t} \kappa_e.
\end{align} 
We will denote the set of all spanning trees of a graph by $\ST(G)$. We  identify spanning trees with their edge sets.
 In the following, we will frequently use the Kirchhoff formula 
 \begin{align}\label{eq:Kirchhoff}
	\det \Delta_c= |G| \sum_{\t \in \ST(G)} w(c,\t). 
 \end{align} 
 for 
 the determinant of a weighted graph Laplacian (cf. \cite{MR1813436} for a proof). Let us remark that the Kirchhoff formula is frequently
 used in statistical mechanics and has also been used in the context of gradient interface models for some potentials as in \eqref{eq:pot_general} in 
 \cite{MR2905798}.
\begin{remark}\label{rem:multiedges}
Note that equation \eqref{eq:Kirchhoff} remains true for graphs with multi-edges and loops. Indeed, loops
have no contribution on both sides and multi-edges can be replaced by a single edge with the sum of the conductances as conductance.
\end{remark}

\paragraph{Correlation inequalities}
 We will now show correlation inequalities for the measures $\P=\P^{G,p}$. 
We start by recalling several of the well known correlation inequalities.
To state our results we introduce some notation.
Let $E$ be a finite or countable infinite set. Let $\Omega=\{1,q\}^E$
and $\mathcal{F}$ the $\sigma$-algebra generated by cylinder events. 
We consider the usual partial order on $\Omega$ given by $\omega^1\leq \omega^2$ iff $\omega^1_e\leq \omega^2_e$ for all $e\in E$. A function $X:\Omega\to \R$ is increasing if
$X(\omega_1)\leq X(\omega_2)$  for $\omega_1\leq \omega_2$ and decreasing if $-X$ is increasing.
An event $A\subset \Omega$ is increasing if its indicator function is increasing.
We write $\discmu_1\succsim \bar \mu_2$ if $\discmu_1$ stochastically dominates $\discmu_2$ which is by Strassen's Theorem equivalent to the existence of a coupling $(\omega_1,\omega_2)$ such that $\omega^1\sim \discmu_1$ and $\omega^2\sim\bar \mu_2$ and $\omega^1\geq \omega^2$ (see \cite{MR177430}). We introduce the minimum $\omega^1 \wedge \omega^2$ and the maximum $\omega^1\vee \omega^2$ of two configurations given by
$(\omega^1 \wedge \omega^2)_e=\min(\omega^1_e, \omega^2_e)$
and 
$(\omega^1 \vee \omega^2)_e=\max(\omega^1_e, \omega^2_e)$ for any $e\in E$.
We call a measure $\discmu$ on $\Omega$ strictly positive if $\discmu(\omega)>0$ for all $\omega\in \Omega$. 
Finally we introduce for $f,g\in E$ and $\omega\in\Omega$ the notation $\omega_{fg}^{\pm \pm}\in \Omega$ for the configuration given by
$(\omega_{fg}^{\pm\pm})_e=\omega_e$ for $e\notin \{f,g\}$ 
and $(\omega_{fg}^{\pm \ast})_f= 1+(q-1)_\pm$, $(\omega_{fg}^{\ast \pm})_g= 1+(q-1)_\pm$. We define $\omega_f^{\pm}$ similarly. We sometimes drop the edges $f$, $g$ from the notation.
We write $\discmu(\omega)=\discmu(\{\omega\})$ for $\omega\in\Omega$ and $\discmu(X)=\int_{\Omega}X\,\d\discmu$ for $X:\Omega\to \R$.
\begin{theorem}[Holley inequality]\label{th:stoch_dom}
Let $\Omega=\{1,q\}^E$ be finite and $\discmu_1$, $\discmu_2$ strictly positive measures on $\Omega$  that  satisfy the Holley inequality
\begin{align}\label{eq:HolleyTheorem}
\discmu_2(\omega_1\vee \omega_2)\discmu_1(\omega_1\wedge \omega_2)\geq \discmu_1(\omega_1)\discmu_2(\omega_2) \quad\text{for $\omega_1,\omega_2\in \Omega$}.
\end{align}
Then $\discmu_1\precsim \discmu_2$.
\end{theorem}
\begin{proof}
The original proof appeared in \cite{MR0341552}, a simpler proof can be found , e.g., in \cite[Theorem 2.1]{MR2243761}. 
\end{proof}
A strictly positive measure is called \emph{strongly positively associated}
if it satisfies the FKG lattice condition
\begin{align}\label{eq:lattice_FKG}
\discmu(\omega_1\vee\omega_2) \discmu(\omega_1\wedge \omega_2)\geq \discmu(\omega_1)\discmu(\omega_2) \quad\text{ for $\omega_1,\omega_2\in \Omega$}.
\end{align}
\begin{theorem}\label{th:FKG}
A strongly positively associated measure $\discmu$ satisfies the FKG inequality, i.e., for increasing functions $X,Y:\Omega\to \R$ 
\begin{align}\label{eq:FKG}
\discmu(XY)\geq \discmu(X)\discmu(Y).
\end{align}
\end{theorem}
\begin{proof}
A proof can be found in  \cite[Theorem 2.16]{MR2243761}.
\end{proof}
The next theorem provides a simple way to verify the assumptions
of Theorem \ref{th:stoch_dom} and Theorem \ref{th:FKG}.
Basically it states that it is sufficient to check the conditions when varying at most two edges.
\begin{theorem}\label{th:stoch_dom_crit}
Let $\Omega=\{1,q\}^E$ be finite and $\discmu_1$, $\discmu_2$ strictly positive measures on $\Omega$.
Then $\discmu_1$ and $\discmu_2$ satisfy \eqref{eq:HolleyTheorem}
iff the following two inequalities hold
\begin{align}
\discmu_2(\omega_f^+)\discmu_1(\omega_f^-)&\geq 
\discmu_1(\omega_f^+)\discmu_2(\omega_f^-),\quad \text{for $\omega\in \Omega$, $f\in E$,} \label{eq:stoch_dom_crit1}\\
\discmu_2(\omega_{fg}^{++})\bar{\mu}_1(\omega_{fg}^{--})
&\geq \discmu_1(\omega_{fg}^{+-})\discmu_2(\omega_{fg}^{-+}),\quad \text{for $\omega\in \Omega$, $f,g\in E$}. \label{eq:stoch_dom_crit2}.
\end{align}
In particular, \eqref{eq:stoch_dom_crit1} and \eqref{eq:stoch_dom_crit2}
together imply $\discmu_1\precsim\discmu_2$.
\end{theorem}
\begin{proof}
See  \cite[Theorem 2.3]{MR2243761}.
\end{proof}
We state one simple corollary of the previous results.
\begin{corollary}\label{co:stoch_dom}
Let $\discmu_1$, $\discmu_2$ be strictly positive measures on $\Omega=\{1,q\}^E$ such that 
one of the measure $\discmu_1$, $\discmu_2$ is strongly positively associated.
 Then
\begin{align}\label{eq:stoch_dom_crit1_repeated}
\discmu_2(\omega_f^+)\discmu_1(\omega_f^-)&\geq 
\discmu_1(\omega_f^+)\discmu_2(\omega_f^-),\quad \text{for $\omega\in \Omega$, $f\in E$}
\end{align}
implies $\discmu_1\precsim \discmu_2$.
\end{corollary}
\begin{proof}
Assuming that $\discmu_1$ is strongly positively associated we find using first the assumption
\eqref{eq:stoch_dom_crit1_repeated} and then \eqref{eq:lattice_FKG}
\begin{align}
\discmu_2(\omega_{fg}^{++})\discmu_1(\omega_{fg}^{--})
\geq \frac{\discmu_1(\omega_{fg}^{++})\discmu_2(\omega_{fg}^{-+})}{\discmu_1(\omega_{fg}^{-+})}\discmu_1(\omega_{fg}^{--})\geq \discmu_2(\omega_{fg}^{-+})\discmu_1(\omega_{fg}^{+-}).
\end{align}
Now Theorem \ref{th:stoch_dom_crit} implies the claim. The proof if $\discmu_2$ is strictly positively associated is similar.
\end{proof}

It is convenient to derive 
the following correlation results for the measures $\P^{G,p}$  from corresponding results for the weighted spanning tree measure. The weighted spanning tree measure on a connected weighted graph $(G,\kappa)$ is a measure
on $\ST(G)$ with distribution
\begin{align}\label{eq:weighted_spanning}
\mathbb{Q}_\kappa^G(\t)=\frac{w(\kappa,\t)}{\sum_{\t'\in \ST(G)}w(\kappa,\t')}.
\end{align} 
This model has been  studied extensively, see \cite{MR1825141} for a survey. An important special case is the uniform spanning tree corresponding to constant conductances $\kappa$ that assigns equal probability to every spanning tree.

 The following lemma provides the basic estimate to check the condition \eqref{eq:stoch_dom_crit2} for the measures $\P^{G,p}$.
Recall the notation $\kappa^{\pm\pm}_{fg}$ introduced before Theorem \ref{th:stoch_dom}
and also the shorthand $\kappa^{\pm\pm}$. 
 \begin{lemma}\label{le:lattice_2_edges}
For a finite graph $G$ and $\kappa\in \{1,q\}^E$ as above  
 \begin{align}\label{eq:lattice_2_edges}
 \det \Delta_{\kappa^{++}} \, \det \Delta_{\kappa^{--}}
 \leq\det \Delta_{\kappa^{+-}} \, \det \Delta_{\kappa^{-+}}.
 \end{align}
 \end{lemma}
 \begin{remark}
 The proof in fact extends to any $\kappa \in \R_+^E$ and
 $(\kappa^{\pm\pm}_{fg})_f=c_f^{\pm}$, $(\kappa^{\pm\pm}_{fg})_g=c_g^{\pm}$
 with $c_f^-\leq c_f^+$ and $c_g^-\leq c_g^+$.
 \end{remark}
 \begin{proof}
 The lemma can be derived from the fact that the weighted  spanning tree has negative correlations. 
It is well known (see, e.g., \cite{MR1825141}) that for all positive weights $\kappa$ on a finite graph $G$ 
the measure $\mathbb{Q}_\kappa^G$ has negative edge correlations
 \begin{align}\label{eq:neg_correlations_ST}
 \mathbb{Q}^G_\kappa(e\in \t|f\in \t)\leq \mathbb{Q}^G_\kappa(e\in \t).
 \end{align}
Simple algebraic manipulations show that this is equivalent to
\begin{align}\label{eq:consequence_neg_correlation_UST}
\Q_\kappa^G(e \in \t, f\in \t) \Q^G_\kappa( e\notin \t, f\notin \t)
\leq \Q^G_\kappa(e\in \t, f\notin \t)\Q^G_\kappa(e\notin \t,f\in \t).
\end{align}
We introduce the following sums
\begin{align}
\begin{alignedat}{2}
A_{fg}&=\sum_{\t\in \ST(G), \, f,g\in \t} w(\kappa,\t),\qquad
&A_{f}&=\sum_{\t\in \ST(G), \, f\in \t, \, g\notin \t} w(\kappa,\t),\\
A_{g}&=\sum_{\t\in \ST(G), \, g\in \t,\, f\notin \t} w(\kappa,\t),\qquad
&A&=\sum_{\t\in \ST(G), \, f,g\notin \t} w(\kappa,\t).
\end{alignedat}
\end{align}
With this notation multiplication 
by $(A_{fg}+A_f+A_g+A)^2$ shows that  \eqref{eq:consequence_neg_correlation_UST} is equivalent to
\begin{align}\label{eq:neg_corr_in_A}
A_{fg} A\leq A_fA_g.
\end{align}
It remains to show that the  statement in the lemma can be deduced from 
\eqref{eq:neg_corr_in_A} (actually the statements are  equivalent).
Clearly we can assume $\kappa=\kappa^{--}$, i.e., $\kappa_f=\kappa_g=1$. 
Using the Kirchhoff formula \eqref{eq:Kirchhoff}
we find the following expression 
\begin{align}\label{eq:Delta_A}
|G|^{-1}\det \Delta_{\kappa^{\pm\pm}}
=\!\sum_{\t\in \ST(G)} w(\kappa^{\pm\pm},\t)
=(\kappa^{\pm\pm})_f(\kappa^{\pm\pm})_g A_{fg}+(\kappa^{\pm\pm})_f A_f+
(\kappa^{\pm\pm})_g A_g+A.
\end{align}
Hence we obtain
\begin{align}
\begin{split}\label{eq:Delta_A2}
|G|^{-2}\det \Delta_{\kappa^{+-}} \, \det \Delta_{\kappa^{-+}} &=
\left( q A_{fg} + q A_f + A_g + A\right)
\left( q A_{fg} +  A_f + q A_g + A\right),
 \\
|G|^{-2}\det \Delta_{\kappa^{++}} \, \det \Delta_{\kappa^{--}} &=\left( q^2 A_{fg} + q A_f + q A_g + A\right)
\left(  A_{fg} + A_f +  A_g + A\right).
\end{split}
\end{align}
Subtracting those two identities we find that only the cross-terms between $A_f,A_g$ and between $A_{fg}, A$ do not cancel and we get
\begin{align}
\begin{split}\label{eq:Delta_A3}
|G|^{-2}\left(\det \Delta_{\kappa^{+-}} \, \det \Delta_{\kappa^{-+}}
-\det \Delta_{\kappa^{++}} \, \det \Delta_{\kappa^{--}}\right)
&=(q^2+1-2q)(A_fA_g-A_{fg}A)
\\
&=(q-1)^2(A_fA_g-A_{fg}A).
\end{split}\end{align}
We can conclude using \eqref{eq:neg_corr_in_A}. 
 \end{proof}
The previous lemma directly implies that the measures $\P^{G,p}$ are strongly positively associated. 
 \begin{corollary}\label{co:Holley1}
The measure $\P^{G,p}$ satisfies the FKG lattice condition for any $\kappa_1,\kappa_2\in \{1,q\}^E$
\begin{align}\label{eq:HolleyCondition}
\P^{G,p}(\kappa_1 \wedge \kappa_2)\P^{G,p}(\kappa_1 \vee \kappa_2)\geq \P^{G,p}(\kappa_1)\P^{G,p}(\kappa_2)
\end{align}
and the FKG inequality
\begin{align}\label{eq:FKG_measure}
\mathbb{E}^{G,p}(XY)\geq \mathbb{E}^{G,p}(X)\, \mathbb{E}^{G,p}(Y)
\end{align}
for any increasing functions $X,Y:\{1,q\}^E\to \R$.
\end{corollary}
\begin{proof}
Lemma \ref{le:lattice_2_edges} 
and the trivial observation that $h(\kappa^{++})+h(\kappa^{--})=h(\kappa^{+-})+h(\kappa^{-+})$
imply for any $\kappa\in \{1,q\}^E$ and $f,g\in E$  the lattice inequality
\begin{align}\label{eq:Lattice_Condition_final_measure}
\P^{G,p}(\kappa^{++})\P^{G,p}(\kappa^{--})\geq \P^{G,p}(\kappa^{+-})
\P^{G,p}(\kappa^{-+}).
\end{align}
Then Theorem \ref{th:stoch_dom_crit} applied to $\discmu_1=\discmu_2=\P^{G,p}$ implies that the FKG lattice condition  \eqref{eq:HolleyCondition}
holds and therefore by Theorem \ref{th:FKG} also the FKG-inequality \eqref{eq:FKG_measure}.

\end{proof}

Let us first state a trivial consequence of this corollary.
\begin{lemma}\label{le:mon_p}
The measures $\P^{G,p}$ and $\P^{G,p'}$ satisfy for $p\leq p'$ 
\begin{align}\label{eq:dom_p}
\P^{G,p'}\succsim \P^{G,p}.
\end{align}
\end{lemma}
\begin{proof}
Using Corollary \ref{co:Holley1} and Corollary \ref{co:stoch_dom} we only need to check
whether \eqref{eq:stoch_dom_crit1_repeated} holds for $\discmu_1=\P^{G,p}$ and $\discmu_2=\P^{G,p'}$. This is clearly the case if $p\leq p'$.
\end{proof}

The next step is to show correlation inequalities with respect to the size of the graph. More specifically we show statements for subgraphs and contracted graphs. This will later easily imply 
the existence of infinite volume limits.
 Moreover, we can bound infinite volume states by finite  volume measures in the sense of stochastic domination.
Let $F\subset E$ be a set of edges. We define the contracted graph
$G/F$ by identifying for every edge $f\in F$ the endpoints of $f$.
Similarly for a set $W\subset V$ of vertices we define the contracted graph $G/W$ by identifying all vertices in $W$.
The resulting graphs may have multi-edges.
We also consider connected subgraphs $G'=(V',E')$ of $G$.
Recall the notation $\kappa^{\pm}=\kappa_f^\pm$ for $f\in E$.
We use the notation $\Delta_\kappa^{G'}$ for the graph Laplacian on $G'$ where we restrict the conductances $\kappa$ 
to $E'$ and we denote by $\Delta_\kappa^{G/F}$ the graph Laplacian on $G/F$. 
The following lemma relates the determinants of the different graph Laplacians.
\begin{lemma}\label{le:det_incl}
With the notation introduced above we have for $\kappa\in \{1,q\}^E$
\begin{align}\label{eq:det_inequality_diff_sizes}
\frac{\det \Delta^{G'}_{\kappa^+}}{\det \Delta^{G'}_{\kappa^-}}\geq 
\frac{\det \Delta^G_{\kappa^+}}{\det \Delta^G_{\kappa^-}}\geq
\frac{\det \Delta^{G/F}_{\kappa^+}}{\det \Delta^{G/F}_{\kappa^-}}.
\end{align}
\end{lemma}
\begin{remark} The lemma again extends to $\kappa\in \R_+^E$ and $\kappa^\pm_f$
with $(\kappa_f^+)_f=c_+>c_-=(\kappa_f^-)_f$. 
\end{remark}
\begin{proof}
The proof is similar to the proof of Lemma \ref{le:lattice_2_edges}.
We  derive the statement from a property of the weighted spanning tree model.
For graphs as above and $e\in E'$ the estimate 
\begin{align}\label{eq:proba_inequality_diff_sizes}
\Q_\kappa^{G'}(e\in \t)\geq \Q_\kappa^{G}(e\in \t)\geq \Q_\kappa^{G/F}(e\in \t)
\end{align}
holds  (see Corollary 4.3 in \cite{MR1825141} for a proof).
We can rewrite (assuming again $\kappa_f=1$, i.e., $\kappa=\kappa^-$)
\begin{align}
\begin{split}\label{eq:quotient_det_restricted}
\frac{\det \Delta^{G}_{\kappa^+}}{\det \Delta^{G}_{\kappa^-}}
=\frac{\sum_{\t\in \ST(G), f\notin \t} w(\kappa,\t)+q
\sum_{\t\in \ST(G), f\in \t} w(\kappa,\t)}{
\sum_{\t\in \ST(G), f\notin \t} w(\kappa,\t)+
\sum_{\t\in \ST(G), f\in \t} w(\kappa,\t)}.
\end{split}
\end{align}
Note that
\begin{align}\label{eq:quotient_det_restricted2}
\frac{\sum_{\t\in \ST(G), f\in \t} w(\kappa,\t)}{\sum_{\t\in \ST(G), f\notin \t} w(\kappa,\t)}
=\frac{\Q_\kappa^{G}(f\in \t)}{\Q^{G}_\kappa(f\notin \t)}
\end{align}
and therefore (using $\kappa=\kappa^-$)
\begin{align}\label{eq:quotient_det}
\frac{\det \Delta^{G}_{\kappa^+}}{\det \Delta^{G}_{\kappa^-}}
=\frac{1+q \frac{\Q^{G}_{\kappa^{-}}(f\in \t)}{\Q^{G}_{\kappa^-}(f\notin \t)}}{1+ \frac{\Q_{\kappa^-}^{G}(f\in \t)}{\Q^{G}_{\kappa^-}(f\notin \t)}}
=1+(q-1)\Q^{G}_{\kappa^-}(f\in \t).
\end{align}
Similar statements hold for the graphs $G/F$ and $G'$.
Hence \eqref{eq:proba_inequality_diff_sizes} implies 
\eqref{eq:det_inequality_diff_sizes}.
\end{proof}
Let us remark that the probability $\Q^G_\kappa(f\in \t)$ can also be expressed as a 
current in a certain electrical network. In order to avoid unnecessary notation at this point we
kept the weighted spanning tree measure and we will only exploit this connection when necessary below.

Again, the previous estimates implies correlation inequalities for the measures $\P^{G,p}$.
In the following we consider a fixed value of $p$ but different graphs so that we drop only $p$ from the notation but we keep the graph $G$.
We  introduce the distribution under boundary conditions for a connected subgraph $G'=(V',E')$ of $G$.
For  $\lambda\in \{1,q\}^E$ we define the measure $\P^{G,E',\lambda}$ on $\{1,q\}^{E'}$ by 
\begin{align}\label{eq:def_boundary_condition}
\P^{G,E',\lambda}(\kappa)=\frac{1}{Z} \frac{p^{h(\kappa)}(1-p)^{s(\kappa)}}{\sqrt{\det \Delta^G_{(\lambda,\kappa)}}}
\end{align}
where $(\lambda,\kappa)\in \{1,q\}^E$ denotes the conductances given by $\kappa$ on $E'$
and by $\lambda$ on $E\setminus E'$. 
This definition implies that we have the following domain Markov property for $\omega\in \{1,q\}^{E'}$
\begin{align}\label{eq:DMP}
\P^{G}(\kappa{_{E'}} = \omega \mid \kappa{_{E\setminus E'}}=\lambda_{E\setminus E'})
=\P^{G,E',\lambda}(\omega).
\end{align}
Since the measure $\P^G$ is strongly positively associated,
 \eqref{eq:DMP}
and Theorem 2.24  in \cite{MR2243761}
implies that the measure
$\P^{G,E',\lambda}$ is strongly positively associated.
We now state the consequences of Lemma \ref{le:det_incl} on stochastic ordering.
\begin{corollary}\label{co:boundary_conditions}
For a finite graph $G=(V,E)$, a connected subgraph $G'=(V',E')$, an edge subset $F\subset E$, and
configurations $\lambda_1, \lambda_2\in \{1,q\}^E$ such that $\lambda_1\leq \lambda_2$ the following 
holds
\begin{align}\label{eq:stoch_dom_incl}
\P^{G'}\precsim \P^{G,E',\lambda_1},\qquad  \P^{G,E',\lambda_1}\precsim \P^{G,E',\lambda_2},\qquad
P^{G,E\setminus F, \lambda_2}\precsim  \P^{G/F}.
\end{align}
More generally, we have for $\lambda\in \{1,q\}^E$ and $E''\subset E'$ or $E''\cap F=\emptyset$ respectively
\begin{align}\label{eq:stoch_dom_incl2}
\P^{G',E'',\lambda_{E'}}\precsim \P^{G,E'',\lambda}, \qquad \P^{G,E'',\lambda}\precsim \P^{G/F,E'',\lambda_{E\setminus F}}.
\end{align}
\end{corollary}
\begin{proof}
From Lemma \ref{le:det_incl} we obtain for $f\in E'$ and any $\kappa\in \{1,q\}^{E'}$
\begin{align}\label{eq:lattice_incl1}
\frac{\P^{G,E',\lambda}(\kappa^+)}{\P^{G,E',\lambda}(\kappa^-)}
&\geq 
\frac{\P^{G'}(\kappa^+)}{
\P^{G'}(\kappa^-)}.
\end{align}
Similarly, Lemma \ref{le:det_incl} implies for $f\in E\setminus F$ and $\kappa \in \{1,q\}^{E\setminus F}$
\begin{align}\label{eq:lattice_incl2}
\frac{\P^{G/F}(\kappa^+)}{ 
\P^{G/F}(\kappa^-)}
&\geq
\frac{\P^{G,E\setminus F,\lambda}(\kappa^+)}{\P^{G,E\setminus F,\lambda}(\kappa^-)}.
\end{align}
Then the the strong positive association of $\P^G$  and Corollary \ref{co:stoch_dom} imply
the first and the last stochastic ordering claimed in \eqref{eq:stoch_dom_incl}.
The  stochastic domination result in the middle of \eqref{eq:stoch_dom_incl} follows from \eqref{eq:DMP} and a general result for 
strictly positive associated measures (see \cite[Theorem 2.24]{MR2243761}).
The proof of \eqref{eq:stoch_dom_incl2} is similar.
\end{proof}

\paragraph{Infinite volume measures.} The definition of the measure $\P$ shows that it is a finite volume Gibbs measure
for the energy $E(\kappa)=\ln(\det \Delta_\kappa)/2$ and a homogeneous Bernoulli a priori measure.
We would like to define infinite volume limits for the measures $\P^G$ and
define a notion of Gibbs measures in infinite volume.
This requires some additional definitions.
Recall the definition of the $\sigma$-algebras $\discsigma_E$ for $E\subset\Edge(\Z^d)$ and note that there is a similar definition for general graphs which will be used in the following.
An event $A\subset \discsigma$ is called local if it measurable with respect 
to $\discsigma_E$ for some finite set $E$, i.e., $A$ depends only on finitely many edges. 
Similarly we define a local function as a function that is measurable with respect to
$\discsigma_E$ for a finite set $E$.
We say that a sequence of measures $\mu_n$ on $\{1,q\}^{\Edge(\Z^d)}$ 
converges in the topology of local convergence to a measure $\mu$ if $\mu_n(A)\to \mu(A)$ for all local events $A$. For a background on the choice of topologies in the context of Gibbs measures we refer to \cite{MR2807681}.
 The construction of the infinite volume states proceeds similarly to 
the construction for the random cluster model by
defining a specification and introducing the notion of free and wired boundary conditions. For simplicity
we restrict the analysis to $\Z^d$ but the generalisation to more general graphs is straightforward. 
First, we define infinite volume limits of the finite volume distributions with wired and free boundary conditions.
Let us denote by $\Lambda_n=[-n,n]\cap \Z^d$ the
ball with radius $n$ in the maximum norm around the origin and we denote by $E_n=\Edge(\Lambda_n)$
the edges in $\Lambda_n$.
We introduce the shorthand   $\Lambda_n^w=\Lambda_n/\partial\Lambda_n$ for the box with wired boundary conditions.
We define 
\begin{align}\label{eq:def_Gibbs_Lambda_n}
\discmu^{0}_{n, p}=\P^{\Lambda_n,p},\quad \qquad  \discmu^{1}_{n, p}=\P^{\Lambda_n^w, p}
\end{align}
for the measure $\P$ on $\Lambda_n$ with free and wired  boundary conditions respectively.
From Corollary \ref{co:boundary_conditions} and equation \eqref{eq:DMP}  we conclude that
for any increasing event $A$ depending only on edges in $E_n$
\begin{align}\label{eq:compare_bc_n_n+1}
\discmu_{n+1}^0(A)=\P^{\Lambda_{n+1}}(A)=\P^{\Lambda_{n+1}}( \P^{\Lambda_{n+1},E_n, \kappa}(A))\geq  \P^{\Lambda_n}(A)=\discmu_n^0(A).
\end{align}
We conclude that for any increasing event $A$ depending only on finitely many edges
the limits $\lim_{n\to \infty} \discmu_{n,p}^0(A)$ and similarly $\lim_{n\to \infty} \discmu_{n,p}^1(A)$ exist.
Using standard arguments we can write every local event $A$  as a union
and difference of increasing local events and we conclude that $\lim_{n\to\infty} \discmu_{n,p}^{0}(A)$ and $\lim_{n\to\infty} \discmu_{n,p}^{1}(A)$ exist. It is well known (see \cite{MR1700749}) that this implies convergence of $\discmu_{n,p}^{0}$ and $\discmu_{n,p}^1$ to a measure on $\{1,q\}^{\Edge(\Z^d)}$ in the topology of local convergence.
We denote the infinite volume measures by $\discmu_p^{0}$ and $\discmu_p^1$.
\begin{lemma}\label{le:inf_volume_measures}
The measure $\discmu_p^0$ and $\discmu_p^1$ satisfy the FKG-inequality and for $0\leq p\leq p'\leq 1$ the relations
\begin{align}
\discmu_p^0\precsim \discmu_p^1, \quad \discmu_p^0\precsim \discmu_{p'}^0, \quad \discmu_p^1\precsim \discmu_{p'}^1.
\end{align}
Moreover they are invariant under symmetries of the lattice and  ergodic with respect to translations.
\end{lemma}
\begin{proof}
This is a consequence of Corollary \ref{co:Holley1} and Corollary \ref{co:boundary_conditions}
and a limiting argument. See the proof of Theorem 4.17 and Corollary 4.23 in \cite{MR2243761} for a detailed proof for the random cluster model which also applies to the model considered here.
Ergodicity is proved by showing that the measures are even mixing.
\end{proof}

\paragraph{Infinite volume specifications.}
We now introduce the concept of  infinite volume Gibbs measures for this model.
We first consider the case of a finite connected graph $G$.
For $E\subset \Edge(G)$ we consider the finite volume specifications $\discgamma^G_{E}:\discsigma\times \{1,q\}^{\Edge(G)}\to \R$ 
\begin{align}\label{eq:def_discgamma}
\discgamma^G_{E}(A,\lambda)
=\frac{1}{Z_\lambda}\sum_{\kappa\in A} \mathbb{1}_{\kappa{_{E^{\crm}}}=\lambda{_{E^{\crm}}}}
\frac{p^{h(\kappa)}(1-p)^{s(\kappa)}}{\sqrt{\det \Delta_\kappa}}
\end{align} 
where the normalisation $Z_\lambda$ ensures that $\discgamma^G_{E}(\cdot,\lambda)$ is a probability measure. 
A simple calculation shows that $\discgamma^G$ is indeed a specification, i.e., $\discgamma^G_{E}$ are proper probability kernels that satisfy for $E\subset E'$
\begin{align}
\discgamma_{E'}^G\discgamma_{E}^G=\discgamma_{E'}^G.
\end{align}
Since $\discgamma_{E}(\cdot, \lambda)$ is concentrated on a finite set
it is helpful to use the notation $\discgamma_{E}(\kappa,\lambda)=\discgamma_{E}(\{\kappa\},\lambda)$.
The measure $\P^{G}$ is a finite volume Gibbs measure, i.e., it satisfies
\begin{align}\label{eq:Gibbs_finite_volume}
\P^G\discgamma^{G}_{E}=\P^G
\end{align}
or put differently for $\kappa,\lambda\in \{1,q\}^{E}$
\begin{align}\label{eq:discgamma_and_bc}
\discgamma_{E}^G(\kappa,\lambda)=\P^{G,E,\lambda}(\kappa_{E}) \mathbb{1}_{\kappa_{\Edge(G)\setminus E}=\lambda_{\Edge(G)\setminus E}}.
\end{align}

We would like to call $\mu$ a Gibbs measure on $\{1,q\}^{\Z^d}$ for the random conductance model if
\begin{align}\label{eq:defGibbsMorally}
\discmu\discgamma^{\Z^d}_E=\discmu
\end{align}
holds for all $E\subset \Edge(\Z^d)$ finite.
However, $\discgamma^{G}_E$ is a priori only well defined for finite graphs so that
we use  an approximation procedure for infinite graphs.
Let $G$ be an connected infinite graph.
We are a bit sloppy with the notation and do not distinguish between $\discgamma^H_{E}$ 
for a subgraph $H$ of $G$
and its proper extension to $\discsigma\times \{1,q\}^{\Edge(G)}$, i.e., we define for $\kappa,\lambda\in \{1,q\}^{\Edge(G)}$
\begin{align}
\discgamma^{H}_{E}(\kappa,\lambda)=\mathbb{1}_{\kappa_{E^{\crm}}=\lambda_{E^{\crm}}} \discgamma^H_{E}(\kappa_{\Edge(H)},\lambda_{\Edge(H)}).
\end{align}
We denote for $f\in \Edge(G)$ and $\kappa\in \{1,q\}^{\Edge(G)}$ by $\kappa^+$ and $\kappa^-$ as before the configurations 
such that $\kappa_e^+=\kappa^-_e$
for $e\neq f$ and $\kappa_f^-=1$, $\kappa^+_f=q$.

In the following we assume $p\in (0,1)$. For $p\in \{0,1\}$ the measures $\P^{G,p}$ agree
with the Dirac measure on the constant 1 or constant $q$ configuration.
Since we assume that $E$ is finite the specification $\discgamma^H_{E}$ is uniquely characterized
by the fact that it is proper and it satisfies for $\kappa,\lambda\in \{1,q\}^{\Edge(H)}$ such that
$\kappa_{E^{\crm}}=\lambda_{E^{\crm}}$
\begin{align}\label{eq:characterization_specification}
\frac{\discgamma^{H}_{E'}(\kappa^-,\lambda)}{
\discgamma^{H}_{E'}(\kappa^+,\lambda)}
=
\frac{1-p}{p} \sqrt{\frac{\det \Delta^{H}_{\kappa^+} }{ \det \Delta^{H}_{\kappa^-}}}
=\frac{1-p}{p} \sqrt{
1+(q-1)\mathbb{Q}^{H}_{\kappa^-}(f\in \t)}
\end{align}
where we used \eqref{eq:quotient_det} in the second step.
We show that we can give meaning to this expression in infinite volume. For this we sketch the definition of spanning trees in infinite volume but we refer to the literature for details (see \cite{MR1825141}).
A monotone exhaustion of an infinite graph $G$ is a sequence of subgraphs $G_n$ such that
$G_n\subset G_{n+1}$ and $G=\bigcup_{n\geq 1} G_n$.
It can be shown that for any finite sets $E_1\subset E_2\subset \Edge(G)$ the limit 
$\lim_{n\to \infty} \Q^{G_n}_\kappa(\t\cap E_2=E_1)$ exists. In fact this is a consequence
of \eqref{eq:proba_inequality_diff_sizes} and the arguments we used for $\discmu_n^0$ above.
Hence it is possible to define a measure $\Q_\kappa^{G,0}$ on $2^{\Edge(G)}$, the power set of $\Edge(G)$ which will be called the weighted free spanning forest on $G$ (as the name suggest the measure is supported on forests but not necessarily on trees, i.e., on connected subsets of edges).
Similarly, we can define the wired spanning forest $\Q_\kappa^{G,1}$ replacing the subgraphs $G_n$ by the contracted graphs $G_n/\partial G_n$.
By definition those measures satisfy 
\begin{align}
\label{eq:convergence_edge_UST}
\lim_{n\to \infty}\Q_\kappa^{G_n}(f\in \t)&=\Q^{G,0}_\kappa(f\in \t)\\
\label{eq:convergence_edge_UST2}
\lim_{n\to \infty}\Q_\kappa^{G_n/\partial G_n}(f\in \t)&=\Q^{G,1}_\kappa(f\in \t)
\end{align}
for any $f\in E$.
Then it is possible to define two families of proper probability kernels $\discgamma^{G,0}_{E}$ 
and $\discgamma^{G,1}_{E}$ for $E\subset(\Edge(G))$ finite by the property that for $f\in E$
and $\kappa,\lambda\in \{1,q\}^{\Edge(G)}$ such that $\kappa_{E^{\crm}}=\lambda_{E^{\crm}}$
\begin{align}
\frac{\discgamma^{G,0}_{E}(\kappa^-,\lambda)}{
\discgamma^{G,0}_{E}(\kappa^+,\lambda)}
&=\frac{1-p}{p} \sqrt{
1+(q-1)\mathbb{Q}^{G,0}_{\kappa^-}(f\in \t)}
\\
\frac{\discgamma^{G,1}_{E}(\kappa^-,\lambda)}{
\discgamma^{G,1}_{E}(\kappa^+,\lambda)}
&=\frac{1-p}{p} \sqrt{
1+(q-1)\mathbb{Q}^{G,1}_{\kappa^-}(f\in \t)}.
\end{align}
From and \eqref{eq:characterization_specification} and \eqref{eq:characterization_specification} we conclude that $\discgamma^{G,0}$ and $\discgamma^{G,1}$ are well defined.
Moreover we obtain that this family of  probability kernels satisfy
for $\lambda,\kappa\in \{1,q\}^{\Edge(G)}$
\begin{align}\label{eq:conv_spec_1}
\discgamma^{G,0}_{E}(\kappa,\lambda)&=
\lim_{n\to \infty}
\discgamma^{G_n}_{E}(\kappa,\lambda),
\\
\label{eq:conv_spec_2}
\discgamma^{G,1}_{E}(\kappa,\lambda)&=
\lim_{n\to \infty}
\discgamma^{G_n/\partial G_n}_{E}(\kappa,\lambda).
\end{align}
Note that the concatenation for $\discgamma^{G,0}$ for $E',E\subset\Edge(\Z^d)$  is given by
\begin{align}
\discgamma^{G,0}_{E}\discgamma^{G,0}_{E'}(\kappa,\lambda)
=\sum_{\sigma: \sigma_{E^\crm}=\lambda_{E^\crm}} \discgamma^{G,0}_{E}(\sigma,\lambda)\discgamma^{G,0}_{E'}(\kappa,\sigma),
\end{align}
in particular it only involves a finite sum in the case of a finite spin space.
We conclude using \eqref{eq:conv_spec_1} and \eqref{eq:conv_spec_2} 
that $\discgamma^{G,0}_E$ and $\discgamma^{G,1}_E$ define two specifications on $G$.

Suppose the wired and the free uniform spanning forest on $G$ agree. This implies that also the weighted 
wired and free spanning forest $\Q^{G,0}_\kappa$ and $\Q^{G,1}_\kappa$ on $G$ agree if the conductances $\kappa_e$ are contained in a compact subset of $(0,\infty)$ (see Theorem 7.3 and Theorem 7.7 in \cite{MR1825141}). 
Thus  $\discgamma^{G,1}_{E}=\discgamma^{G,0}_{E}$ in this case.
In particular we obtain that $\discgamma^{\Z^d,0}_{E}=\discgamma^{\Z^d,1}_{E}$ 
because the free and the wired uniform spanning forest on $\Z^d$ agree (Corollary 6.3 in \cite{MR1825141}).
In the following we will denote this specification by $\discgamma_{E}$.
To ensure consistency with the earlier definition of $\extgamma$ we define
for a connected subset $\Lambda\subset \Z^d$ that
$\discgamma_{\Lambda}=\discgamma_{\Edge(\Lambda)}$.
We can now give a formal definition of Gibbs measures for the random conductance model.
\begin{define}\label{def:kappa_Gibbs}
A measure $\discmu\in \mathcal{P}(\{1,q\}^{\Edge(\Z^d)})$ is a Gibbs measure if it
is admitted to the specification $\discgamma_E$.
\end{define} 
As one would expect the infinite volume measures $\discmu_p^0$ and $\discmu_p^1$ are Gibbs measures.
\begin{lemma}
The measures $\discmu_p^0$ and $\discmu_p^1$ are Gibbs measures as defined in Definition \ref{def:kappa_Gibbs}. Moreover any Gibbs measure $\discmu$ satisfies 
$\discmu_p^0\precsim \discmu\precsim \discmu_p^1$.
\end{lemma}
\begin{proof}
By equation \eqref{eq:Gibbs_finite_volume} we have for $E\subset E_n$
\begin{align}\label{eq:discmu_Gibbs}
\discmu_{n}^0\discgamma^{\Lambda_n}_{E}=\discmu_{n}^0.
\end{align}
We show that  both sides converge in the topology of local convergence as $n\to \infty$.
Let $A$ be an increasing event depending on a finite number of edges. We have seen in \eqref{eq:compare_bc_n_n+1} that 
$\mu_n^0(A)$ is an increasing sequence and converges by definition to $\mu^0(A)$.
We derive the convergence of the left hand side of equation \eqref{eq:discmu_Gibbs}
from the following three observations.
First, we conclude from \eqref{eq:stoch_dom_incl} and \eqref{eq:discgamma_and_bc} 
that
$\discgamma^{\Lambda_n}_E(A,\cdot)$ is an increasing function.
Second, using  \eqref{eq:stoch_dom_incl2} and \eqref{eq:discgamma_and_bc} 
we obtain  $\discgamma^{\Lambda_{n+1}}_E(A,\kappa)\geq \discgamma^{\Lambda_n}_E(A,\kappa_{E_n})$
 for all $\kappa\in \{1,q\}^{E_{n+1}}$.
The third observation is that \eqref{eq:compare_bc_n_n+1} can  also be applied to an increasing function
instead of an increasing event. These three  facts imply
\begin{align}
\discmu_n^0(\discgamma^{\Lambda_n}_E(A,\cdot))\leq \discmu^0(\discgamma_E^{\Lambda_n}(A,\cdot))
\leq \discmu^0(\discgamma_E(A,\cdot)).
\end{align}
On the other hand, we obtain for any $m\in \mathbb{N}$
\begin{align}
\lim_{n\to \infty} \discmu_n^0(\discgamma^{\Lambda_n}_E(A,\cdot))
\geq \lim_{n\to \infty}\discmu_n^0(\discgamma^{\Lambda_m}_E(A,\cdot))=\discmu^0(\discgamma^{\Lambda_m}_E(A,\cdot)).
\end{align}
Sending $m\to \infty $ we get 
\begin{align}
\lim_{n\to \infty} (\discmu_n^0\discgamma^{\Lambda_n}_E)(A)\geq  (\discmu^0\discgamma_E)(A).
\end{align}
Hence, we have shown that 
\begin{align}\label{eq:limit_Gibbs_Spec}
\discmu^0\discgamma_E(A)=\lim_{n\to \infty} \discmu_{n}^0\discgamma^{\Lambda_n}_{E}(A)=\lim_{n\to \infty}\discmu_{n}^0(A)=\discmu^0(A)
\end{align}
holds for any increasing and local event $A$. Using standard arguments \eqref{eq:limit_Gibbs_Spec} holds for all local events.
Therefore  $\mu^0$ is  a Gibbs measure.
The proof for $\mu^1$ is similar based on the  identity 
\begin{align}
\discmu_n^1\discgamma_E^{\Lambda_n^w} =\discmu_n^1.
\end{align} 
Finally, a limiting argument and the comparison of boundary conditions show that $\discmu_p^0\precsim \discmu\precsim \discmu_p^1$ for any Gibbs measure $\mu$ (see \cite[Proposition 4.10]{MR2243761}).
\end{proof}

Let us briefly introduce the class of quasilocal specifications which is a natural and useful condition for a specification. For an extensive discussion we refer to the literature \cite{MR2807681}. A quasilocal function on a general state space is
a bounded function $X:F^S\to \R$ that can be approximated arbitrarily well by local functions, i.e., 
\begin{align}
\inf_{Y\text{ local}} \sup_{\omega\in F^S} |X(\omega)-Y(\omega)|=0. 
\end{align}
A specification $\gamma$ is called quasilocal if $\gamma_{\Lambda} X$ is a quasilocal function for every local function $X$.
We will show that the specification $\discgamma_{E}$ is quasilocal. This will be a direct consequence of the following result that shows uniform convergence of
$\discgamma^{\Lambda_n^w}_E$ to $\discgamma_E$. This convergence will be of independent use later.
\begin{lemma}\label{le:uniform_specification}
The specifications $\discgamma_{E_n}$ and $\discgamma_{E_n}^{\Lambda_N^w}$ satisfy 
\begin{align}\label{eq:uniform_specification}
\limsup_{N\to \infty} \sup_{\kappa, \lambda\in \{1,q\}^{\Edge(\Z^d)}} |\discgamma_{E_n}(\kappa,\lambda)-
\discgamma_{E_n}^{\Lambda_N^w}(\kappa,\lambda)|=0.
\end{align}
\end{lemma}
\begin{proof}
First, we claim that it is sufficient to show that
\begin{align}\label{eq:uniform_specification2}
\limsup_{N\to \infty} \sup_{\kappa\in \{1,q\}^{\Edge(\Z^d)}}\sup_{f\in E_n}
|\mathbb{Q}_\kappa^{\Z^d}(f\in \t)-\mathbb{Q}_\kappa^{\Lambda_N^w}(f\in \t)|=0.
\end{align}
Indeed, using \eqref{eq:uniform_specification2} in \eqref{eq:characterization_specification} we obtain 
\begin{align}\label{eq:uniform_specification3}
\limsup_{N\to \infty} \sup_{\substack{\kappa,\lambda\in \{1,q\}^{\Edge(\Z^d)}\\
\kappa_{E_n^{\crm}}=\lambda_{E_n^{\crm}}}}\sup_{f\in E_n}
\frac{\discgamma_{E_n}(\kappa^-_f,\lambda)}{
\discgamma_{E_n}(\kappa^+_f,\lambda)}/
\frac{\discgamma^{\Lambda_N^w}_{E_n}(\kappa^-_f,\lambda)}{
\discgamma^{\Lambda_N^w}_{E_n}(\kappa^+_f,\lambda)}=1.
\end{align}
Since $E_n$ is finite this implies the claim.

It remains to prove \eqref{eq:uniform_specification2}. This is a consequence of the transfer current theorem (see Theorem 4.1 in \cite{MR1825141}) that states
in the special case of the occupation property that for $f=\{x,y\}\in \Edge(G)$
\begin{align}\label{eq:uniform_specification4}
\mathbb{Q}_\kappa^{G}(f\in \t)=I_f(f)=\kappa_f(\delta_x-\delta_y)(\Delta^G_\kappa)^{-1}(\delta_x-\delta_y)
\end{align}
where the expression $I_f(f)$ denotes the current through the edge $f$ when 1 unit of current is induced respectively removed at the two ends of $f$. In the last step we used
that $I_f(f)$  can be calculated by applying the inverse Laplacian to the sources to obtain the potential which can be used to calculate the current through $f$. 
Now \eqref{eq:uniform_specification2} follows from the  display \eqref{eq:uniform_specification4} and Lemma \ref{le:uniform_Laplace_inverse}.
\end{proof}

\begin{corollary}\label{co:quasilocal}
The specification $\discgamma_E$ is quasilocal.
\end{corollary}
\begin{proof}
Let $X$ be a local functions. We need to show that $\discgamma_EX$ is quasilocal.
Lemma~\ref{le:uniform_specification} implies that the local functions $\discgamma^{\Lambda_N^w}_EX$ satisfy
\begin{align}
\lim_{N\to \infty}\sup_{\kappa\in \{1,q\}^{\Edge(\Z^d)}} |\discgamma_EX(\kappa)-\discgamma^{\Lambda_N^w}_EX(\kappa)|
\leq\lim_{N\to\infty} \sup_{\kappa\in \{1,q\}^{\Edge(\Z^d)}} \sum_{
\substack{\lambda\in \{1,q\}^{\Edge(\Z^d)} \\
\lambda_{E^\crm}=\kappa_{E^\crm}}} \left|\left(\discgamma_E(\lambda,\kappa)-\discgamma_E^{\Lambda_N^w}\right) X(\lambda)\right|=0.
\end{align}
\end{proof}
\paragraph{Relation to extended gradient Gibbs measures}
In this paragraph we state the results that relate the random conductance model to 
extended gradient Gibbs measure. This is finally the justification to consider this model. The proofs of the results in this paragraph are deferred to Section \ref{sec:proofs_of_propositions}.
The first Proposition establishes that the $\kappa$-marginal of extended gradient Gibbs measures are Gibbs states for the random conductance model.



\begin{proposition}\label{prop:discmu_Gibbs}
Let $\tilde{\mu}$ be an  extended gradient Gibbs measure associated to a translation invariant and ergodic gradient Gibbs measure $\mu$ with zero tilt. Then the
$\kappa$-marginal $\bar{\mu}$ of $\tilde{\mu}$ is a Gibbs measure in the sense of definition \ref{def:kappa_Gibbs}.
\end{proposition}

The second main result in this paragraph is a 
reverse of Proposition \ref{prop:discmu_Gibbs}, namely that it is possible to obtain an extended Gibbs measure with zero tilt for the potential $V_{p,q}$, given a Gibbs measure  $\discmu$
for the random conductance model with parameters $p,q$.
\begin{proposition}\label{prop:discmu_to_extmu}
Let $\discmu$ be a Gibbs measure  in the sense of Definition \ref{def:kappa_Gibbs} for
parameters $p$ and $q$ and $\kappa\sim \discmu$.
Let  $\p^\kappa$ be the random field that for given $\kappa$ is a Gaussian field with 
zero average, $\p^\kappa(0)=0$, and covariance $(\Delta_\kappa)^{-1}$, i.e.,
$\p^\kappa$ satisfies for $f:\Z^d\to \R$ with finite support and $\sum_{x} f(x)=0$ 
\begin{align}\label{eq:Var_extmu}
\mathrm{Var} \left( (f,\p^\kappa)_{\Z^d}\right)=(f,(\Delta_\kappa)^{-1}f).
\end{align}
Let $\extmu$ be the joint law of $(\kappa,\nabla\p^\kappa)$.
Then $\extmu$ is
an extended Gibbs measure for the potential $V_{p,q}$ 
with zero tilt, in particular its $\eta$-marginal is a gradient Gibbs measure with zero tilt.
\end{proposition}
As a last result in this direction we state a very useful result from  \cite{MR2778801} that characterizes the law of $\p$ given $\kappa$ for extended gradient Gibbs measures if $\p$ is distributed according to a gradient Gibbs measure. 
\begin{proposition}\label{prop:phi_given_kappa}
Let $\mu$ be a translation invariant, ergodic gradient Gibbs measure with zero tilt 
and $\extmu$ the corresponding extended gradient Gibbs measure.
Then the conditional law of $\p$
given $\kappa$ is $\extmu$-almost surely Gaussian.
It is determined by its expectation
\begin{align}\label{eq:Ephi_given_kappa}
\E\big(\p_x\mid\discsigma\big)(\kappa)=0
\end{align} 
and the covariance given by $(\Delta_\kappa)^{-1}$, i.e., for $f:\Z^d\to \R$ with finite support and $\sum_{x} f(x)=0$ 
\begin{align}\label{eq:Varphi_given_kappa}
\mathrm{Var}_{\extmu} \left( (f,\p)_{\Z^d}\mid\discsigma\right)(\kappa)=(f,(\Delta_\kappa)^{-1}f).
\end{align}
\end{proposition}
\begin{proof}
This is Lemma 3.4 in \cite{MR2778801}.
\end{proof}

In particular those results establish the following. Assume that $\mu$ is an ergodic zero tilt gradient Gibbs measure. Let $\discmu$ be the 
$\kappa$-marginal of the corresponding extended gradient Gibbs measure $\extmu$ 
(which by Proposition \ref{prop:discmu_Gibbs} is Gibbs for the random conductance model). 
We can  use Proposition \ref{prop:discmu_to_extmu} to construct an extended gradient Gibbs measure $\extmu'$. Using the definition of $\extmu'$ 
in Proposition~\ref{prop:discmu_to_extmu} and Proposition~\ref{prop:phi_given_kappa} 
we conclude that we get back the extended gradient Gibbs measure we started from, i.e., $\extmu=\extmu'$.

\section{Further properties of the random conductance model}\label{sec:further_prop}
In this section we state and prove more results about the random conductance model 
considered in this work and use the results from the previous section to derive corresponding results for the
associated gradient interface model. We end this section with some  conjectures and open questions.
We start by proving  $\discmu_p^0=\discmu_p^1$ for $d\geq 2$ and almost all values of $p$ which
will in particular implies uniqueness of the Gibbs measure for those $p$.
\begin{theorem}\label{th:countable_non_uniqueness}
For every $q\geq 1$ there are at most countably many  $p\in [0,1]$ such that $\discmu_p^1\neq \discmu_p^0$.
\end{theorem}
\begin{proof}
It is a standard consequence of the invariance under lattice symmetries and $\discmu_{p}^0\precsim \discmu_{p}^1$ that $\discmu_p^1=\discmu_p^0$ is equivalent to $\discmu_p^1(\kappa_e=q)=\discmu_p^0(\kappa_e=q)$ for one and therefore any $e\in \Edge(\Z^d)$ (see, e.g, Proposition 4.6 in  \cite{MR2243761}).
Lemma \ref{le:monotonicity_in_p} below implies for $e\in \Edge(\Z^d)$ 
\begin{align}\label{eq:discmu_monotone}
\discmu_p^0(\kappa_e=q)\leq\discmu_p^1(\kappa_e=q)\leq \discmu_{p'}^0(\kappa_e=q)
\end{align}
for any $p'>p$. In particular, we can conclude that $\discmu_p^0=\discmu_p^1$ holds for all points of continuity of the map $p\mapsto \discmu^0_p(\kappa_e=q)$.
Since this map is increasing by Lemma \ref{le:mon_p} it has only countably many points of discontinuity.
\end{proof}

We are now in the position to prove Theorem \ref{th:main_unique}. 
\begin{proof}[Proof of Theorem \ref{th:main_unique}]
We note that a translation invariant zero tilt Gibbs measure exists for any $p$ and $q$, e.g., 
as a limit of torus Gibbs states (see the proof of Theorem 2.2 in \cite{MR2322690}).
It remains to show uniqueness.
Consider $p$ such that $\bar{\mu}^1_p=\bar{\mu}^0_p$ which is true for all but a countable number of $p\in [0,1]$ by Theorem \ref{th:countable_non_uniqueness} above.
Let $\mu_1$ and $\mu_2$ be ergodic  zero tilt gradient Gibbs measures for $V=V_{p,q}$. 
By Proposition \ref{prop:discmu_Gibbs} the corresponding $\kappa$-marginals $\discmu_1$ and $\discmu_2$ of the extended Gibbs measures
$\extmu_1$ and $\extmu_2$ are Gibbs measures in the sense of Definition \ref{def:kappa_Gibbs} and therefore equal. 
Using Proposition~\ref{prop:phi_given_kappa} we conclude  that since $\mu_1$ and $\mu_2$ are ergodic zero tilt gradient Gibbs measures  their laws are determined by $\discmu_1$ and $\discmu_2$, hence $\mu_1=\mu_2$. 
\end{proof}
\begin{remark}\label{rem:conv_pressure}
Similar arguments for this model appeared already in the proof of Theorem 2.4 in \cite{MR2322690} where they use the convexity of the pressure to show that the number of $q$-bonds on the torus is concentrated around its expectation in the thermodynamic limit. However, this is not sufficient to conclude uniqueness.
\end{remark}
The key ingredient in the  proof of Theorem \ref{th:countable_non_uniqueness}
is the following lemma that 
compares $\discmu_p^1(\kappa_e=q)$ with $\discmu_{p'}^0(\kappa_e=q)$ for $p<p'$. Intuitively the reason for this result is that 
a change of $p$ is a bulk effect of order $|\Lambda|$  while the effect of the boundary conditions 
is of order $|\partial\Lambda|$.
\begin{lemma}\label{le:monotonicity_in_p}
For any $p<p'$ we have
\begin{align}\label{eq:monotonicity_in_p}
 \mu_{p'}^0(\kappa_e=q)\geq \mu_{p}^1(\kappa_e=q).
\end{align}
\end{lemma}
\begin{proof}
The proof follows the proof of Theorem 1.12 in 
\cite{duminil2017lectures} where a similar result for the random cluster model is shown. The only difference is that the comparison between free and wired boundary conditions is slightly less direct.
We define $a= \discmu_{p'}^0(\kappa_e=q)$ and $b= \discmu_{p}^1(\kappa_e=q)$.
Comparison between boundary condition implies $\discmu_{n,p'}^0(\kappa_e=q)\leq \discmu_{p'}^0(\kappa_e=q)=a$ for any
$e\in E_n$. Recall that $h(\kappa)=| \{e\in \Edge(G):\kappa_e=q\} |$ denotes the number of 
$q$-bonds and $s(\kappa)$ similarly the number of 1-bonds. The definition of $a$ and $b$ implies   for $0<\varepsilon<1-a$
\begin{align}\label{eq:monoton_p_1}
\discmu_{n, p'}^0\Big(h(\kappa)\Big)\leq a|E_n| \;\Rightarrow
\discmu_{n, p'}^0\Big(h(\kappa)\leq (a+\varepsilon)|E_n|\Big)\geq \varepsilon.
\end{align}
Similarly for $0<\varepsilon<b$
\begin{align}\label{eq:monoton_p_2}
\discmu_{n, p}^1\Big(h(\kappa)\Big)\geq b|E_n| \;\Rightarrow
\discmu_{n, p}^1\Big(h(\kappa)\geq (b-\varepsilon)|E_n|\Big)\geq \varepsilon.
\end{align}
Our goal is to show that $b-\varepsilon\leq a+\varepsilon$.
We denote by $\Delta^0$ and $\Delta^1$ the graph Laplacian on $\Lambda_n$ with free and wired boundary conditions respectively. 
To compare the boundary conditions we denote by $T_1=\ST(\Lambda_n^w)$ the set of wired spanning trees on $\Lambda_n$ and
by $T_0=\ST(\Lambda_n)$ the set of spanning trees on 
$ \Lambda_n$ with free boundary conditions.
There is a map  $\Phi:T_0 \to T_1$ such that $\Phi(\t){\restriction_{\Lambda_{n-1}}}=\t{\restriction_{\Lambda_{n-1}}}$.
Indeed, removing all edges in $E_n\subset E_{n-1}$ from $\t$ we obtain an acyclic 
subtree of $\Lambda_n^w$, hence we can find a tree $\Phi(\t)$ such that $\t{\restriction_{\Lambda_{n-1}}}\subset \Phi(\t)\subset \t$. 
The observation $|\t\setminus \Phi(\t)|=|\partial \Lambda_n|-1$ implies that $w(\kappa,\t)\leq
w(\kappa,\Phi(\t))q^{|\partial\Lambda_n|-1}$.
Since $\Phi$ does not change the edges in $E_{n-1}$ each tree
$\t\in T_1$ has at most $2^{|E_n\setminus E_{n-1}|}$ preimages.
We obtain that
\begin{align}\label{eq:compare_Det}
\begin{split}
|\Lambda_n|^{-1}\det \Delta_\kappa^0=\!\sum_{\t\in T_0}w(\kappa,\t)
\leq \sum_{\t\in T_0} w({\kappa},\Phi(\t)) q^{|\partial\Lambda_n|-1}
&\leq 2^{|E_n\setminus E_{n-1}|}q^{|\partial\Lambda_n|}\sum_{\t\in T_1} w({\kappa},\t)
\\
&=2^{|E_n\setminus E_{n-1}|}q^{|\partial\Lambda_n|}|\Lambda_n^w|^{-1} \det \Delta_\kappa^1.
\end{split}
\end{align}
Similarly, there is an injective mapping $\Psi:T_1\to T_0$ such that
$\t\subset \Psi(\t)$.
Indeed, we fix a tree $\t_b$ in the graph $(\Lambda_n\setminus \Lambda_{n-1}, \Edge(\Lambda_n\setminus \Lambda_{n-1})$ and define $\Psi(\t)=\t\cup \t_b\in T_0$.
We get
\begin{align}\begin{split}\label{eq:compare_Det2}
|\Lambda_n^w|^{-1}\det \Delta_\kappa^1=\sum_{\t\in T_1}w(\kappa,\t)
\leq \sum_{\t\in T_1} w({\kappa},\Psi(\t))
&\leq  \sum_{\t\in T_0} w({\kappa},\t)
= |\Lambda_n|^{-1}\det \Delta_\kappa^0.
\end{split}
\end{align}
Inserting the bound $|E_n\setminus E_{n-1}|\leq 2d|\partial \Lambda_n|$ 
 we infer from the definition
\eqref{eq:def_perc_model} for
any $\kappa\in \{1,q\}^{E_n}$
\begin{align}\label{eq:monoton_p_3}
\left({2^{2d} q}\right)^{-|\partial\Lambda_n|/2}\discmu_{n,p}^0(\kappa)\leq \discmu_{n,p}^1(\kappa)
\leq \left({2^{2d} q}\right)^{|\partial\Lambda_n|/2} \discmu_{n,p}^0(\kappa).
\end{align}
We define the constant $\alpha=p'(1-p)/(p(1-p'))>1$. Simple manipulation show  that for any function $X:\{1,q\}^{E_n}\to \R$ 
\begin{align}\label{eq:monoton_p_4}
\discmu_{\Lambda_n,p'}^0(X)=\frac{\discmu_{\Lambda_n,p}^0(\alpha^{h(\kappa)}X)}{\discmu_{\Lambda_n,p}^0(\alpha^{h(\kappa)})}.
\end{align}
Therefore we obtain
\begin{align}
\begin{split}\label{eq:monoton_p_5}
\discmu_{\Lambda_n, p'}^0\Big(h(\kappa)\leq (a+\varepsilon)|E_n|\Big)
&=\frac{\discmu_{\Lambda_n,p}^0\Big(\alpha^{h(\kappa)}
\mathds{1}_{h(\kappa)\leq (a+\varepsilon)|E_n|}\Big)}{\mu_{\Lambda_n,p}^0\Big(\alpha^{h(\kappa)}\Big)}
\\
&\leq 
\frac{\discmu_{\Lambda_n,p}^0\Big(\alpha^{h(\kappa)}
\mathds{1}_{h(\kappa)\leq (a+\varepsilon)|E_n|}\Big)}{\discmu_{\Lambda_n,p}^0\Big(\alpha^{h(\kappa)}\mathds{1}_{h(\kappa)\geq (b-\varepsilon)|E_n|}\Big)}
\\
&\leq 
\frac{\alpha^{(a+\varepsilon)|E_n|}}{
\left(2^{2d} q\right)^{-|\partial\Lambda_n|/2}
\alpha^{(b-\varepsilon)|E_n|} \discmu_{\Lambda_n, p}^1\Big(h(\kappa)\geq (b-\varepsilon)|E_n|\Big)}.
\end{split}
\end{align}
From \eqref{eq:monoton_p_1} and \eqref{eq:monoton_p_2} we conclude
\begin{align}\label{eq:monotone_p_6}
\varepsilon^2\leq \left({2^{2d} q}\right)^{|\partial\Lambda_n|/2}\alpha^{(a-b+2\varepsilon)|E_n|}
\end{align}
which implies $a-b+2\varepsilon\geq 0$ as $n\to \infty$ since $\alpha>1$ and $|E_n|/|\partial \Lambda_n|\to \infty$. The lemma follows as $\varepsilon\to 0$. 
\end{proof}
The next result is a non-uniqueness result for the random conductance model.
\begin{theorem}\label{th:randomconductance_non_unique}
In dimension $d=2$ and for $q>1$ sufficiently large there are two distinct Gibbs measures $\discmu_{p_{\sd}}^1\neq \discmu_{p_{\sd}}^0$ at the self-dual point defined by equation 
\eqref{eq:self_dual_equation}.
\end{theorem}
The proof uses duality of the random conductance model and can be found in Section \ref{sec:phase_transition}.
This result easily implies Theorem \ref{th:main_non_unique}.
\begin{proof}[Proof of Theorem \ref{th:main_non_unique}]
Using Proposition \ref{prop:discmu_to_extmu} we infer from Theorem \ref{th:randomconductance_non_unique} the existence of two translation invariant extended 
gradient Gibbs measures $\tilde{\mu}_0$ and $\tilde{\mu}_1$ constructed from $\discmu_{p_{\sd}}^0\neq\discmu_{p_{\sd}}^1$.
Their $\eta$-marginals $\mu_0$ and $\mu_1$ are not equal since then
the $\kappa$-marginals $\discmu_1$ and $\discmu_2$ would agree. They both have zero tilt by Proposition
\ref{prop:discmu_to_extmu} and the definition of $\extmu$ shows that 
$\extmu$ is translation invariant if $\discmu$ is translation invariant.
\end{proof}
\begin{remark}
A proof similar to Lemma 3.2 in \cite{MR2778801} shows that ergodicity of $\discmu_1$ and $\discmu_2$ implies that $\mu_0$ and $\mu_1$ are itself ergodic.
The only difference is that $\eta$ given $\kappa$ is not independent (which $\kappa$ given $\eta$ is).
Instead one has to rely on the decay of correlations for Gaussian fields stated in Appendix \ref{app:1}.
\end{remark}
\begin{theorem}\label{th:Dobrushin_discrete}
For $d\geq 4$ there is $q_0>1$ such that for $p\in [0,1]$ and $q\in [1,q_0)$
the Gibbs measure for the random conductance model is unique.
Similarly, for $d\geq 4$ and $q\geq 1$ there is  a $p_0=p_0(q,d)>0$ such that the Gibbs measure is unique 
for $p\in [0,p_0)\cup (1-p_0,1]$. 
\end{theorem}
\begin{proof}
We are going to apply Dobrushin's criterion (see, e.g., \cite[Theorem 8.7]{MR2807681}. The necessary estimate is basically a refined version of the proof of Lemma \ref{le:lattice_2_edges}. 
Fix two edges $f,g\in \Edge(\Z^d)$. Recall the notation 
$\lambda^{\pm\pm}=\lambda^{\pm\pm}_{fg}$ and $\lambda^\pm=\lambda^\pm_f$ introduced above Theorem \ref{th:stoch_dom}. We will write $\discgamma_f=\discgamma_{\{f\}}$ in the following.
Note that \eqref{eq:characterization_specification} and
$\discgamma_{f}(\lambda^{+},\lambda)+\discgamma_{f}(\lambda^{-},\lambda)=1$ 
imply that
\begin{align}
\discgamma_{f}(\lambda^{+},\lambda)=
\frac{\discgamma_{f}(\lambda^{+},\lambda)}{\discgamma_{f}(\lambda^{+},\lambda)+\discgamma_{f}(\lambda^{-},\lambda)}=\frac{p}{p+(1-p)\sqrt{1+(q-1)\mathbb{Q}_{\lambda^-}(f\in \t)}}
\end{align}
where $\Q_{\lambda^-}$ denotes the weighted spanning forest measure on $\Z^d$ with conductances $\lambda^-$.
 We need to bound the entries of the Dobrushin interdependence matrix given by
\begin{align}
\begin{split}
&C_{fg}=\sup_{\lambda\in \{1,q\}^{\Edge(\Z^d)}}|\discgamma_{f}(\lambda^{++}_{fg},\lambda^{++}_{fg})-\discgamma_{f}(\lambda^{+-}_{fg}, \lambda^{+-}_{fg})|
\\
&=\sup_{\lambda\in \{1,q\}^{\Edge(\Z^d)}}
\left|\frac{p}{p+(1-p)\sqrt{1+(q-1)\mathbb{Q}_{\lambda^{-+}}(f\in \t)}}-
\frac{p}{p+(1-p)\sqrt{1+(q-1)\mathbb{Q}_{\lambda^{--}}(f\in \t)}}\right|.
\end{split}\end{align}
Since the derivative of the map $x\mapsto p/(p+(1-p)\sqrt{x})$ is bounded by $p(1-p)$ for $x\geq 1$ we conclude that
\begin{align}
\sup_{\lambda\in \{1,q\}^{\Edge(\Z^d)}}|\discgamma_{f}(\lambda^{++},\lambda^{++})-\discgamma_{f}(\lambda^{+-}, \lambda^{+-})|
\leq p(1-p)(q-1)\left|\mathbb{Q}_{\lambda^{-+}}(f\in \t)-\mathbb{Q}_{\lambda^{--}}(f\in \t)\right|.
\end{align}
To simplify the notation we assume $\lambda=\lambda^{--}$.
We can express $\mathbb{Q}_{\lambda^{-+}}(f\in \t)$ through the measure $\mathbb{Q}_{\lambda^{--}}=\mathbb{Q}_\lambda$
as follows
\begin{align}
\mathbb{Q}_{\lambda^{-+}}(f\in \t)
=\frac{\mathbb{Q}_{\lambda}(f\in \t, g\notin \t)+q\mathbb{Q}_{\lambda}(f\in \t, g\in \t)}{
q\mathbb{Q}_{\lambda}(g\in \t)+\mathbb{Q}_{\lambda}(g\notin \t)}.
\end{align}
A sequence of manipulations then shows that
\begin{align}
\mathbb{Q}_{\lambda^{-+}}(f\in \t)-\mathbb{Q}_{\lambda^{--}}(f\in \t)
=\frac{(q-1)\big(\mathbb{Q}_{\lambda}(f\in \t,g\in \t)-\mathbb{Q}_{\lambda}(f\in \t)\mathbb{Q}_{\lambda}(g\in \t)\big)}{q\mathbb{Q}_{\lambda}(g\in \t)+\mathbb{Q}_{\lambda}(g\notin \t)}.
\end{align}
The numerator can be rewritten using the transfer-current Theorem for two edges (see \cite[Page 10]{MR1825141} and equation below 4.3 in \cite{MR3616205}) 
\begin{align}
\Q_{\lambda}(f\in \t, g\in \t)-\Q_{\lambda}(f\in \t)\Q_{\lambda}(g\in \t)=-I^{\lambda}_f(g)I^{\lambda}_g(f).
\end{align}
where $I_f^\kappa(g)$ denotes the current through $g$ in a resistor network with conductances $\kappa$ when 1 unit of current is inserted (respectively removed) at the ends of $f$ (using a fixed orientation of the edges here, e.g., lexicographic).
All together we have shown that
\begin{align}\label{eq:dob1}
C_{fg}\leq \sup_{\kappa\in \{1,q\}^{\Edge(\Z^d)}} p(1-p)(q-1)^2 I^\kappa_f(g)I^\kappa_{g}(f).
\end{align}
Using electrical network theory we can express for $f=(x,x+e_i)$ and $g=(y,y+e_j)$
\begin{align}\label{eq:dob2}
I^\kappa_f(g)=\kappa_g\big(\delta_{y+e_i}-\delta_y,(\Delta_\kappa)^{-1}(\delta_{x+e_i}-\delta_x)\big)_{\Z^d}=\kappa_g\nabla_{x,i}\nabla_{y,j} G_\kappa(y,x)
\end{align} 
where $G_\kappa$ denotes the inverse of the operator $\Delta_\kappa$ which exists in 
dimension $d\geq 3$ and whose derivative exists in dimension $d\geq 2$.
Combining the bound  \eqref{eq:Greens_Hoelder2} in Lemma \ref{le:GreensBound}, \eqref{eq:dob1}, and \eqref{eq:dob2} we conclude for $f\in \Edge(\Z^d)$ that
\begin{align}
\sum_{g\in \Edge(\Z^d)} C_{fg}\leq C(q,d)p(1-p)(q-1)^2
\sum_{x\in \Z^d} (1+|x|)^{2(2-d-2\alpha)}.
\end{align}
In dimension $d\geq 4$ the sum is finite.
Now, for fixed $q$, the sum becomes smaller 1 for $p$ sufficiently close to 0 or 1. Therefore there is
$p_0=p_0(q,d)$ such that  the Gibbs measure is unique for $p\in [0,p_0)\cup(1-p_0,1]$.
On the other hand, the constant $C(q,d)$ from Lemma \ref{le:elliptic_Nash_Moser} is decreasing in $q$.
Therefore we can estimate uniformly for $p\in [0,1]$  and for $q\leq 2$
\begin{align}
\sum_{g\in \Edge(\Z^d)} C_{fg}\leq \frac{C(2,d)}{4}(q-1)^2
\sum_{x\in \Z^d} (1+|x|)^{2(2-d-2\alpha)}.
\end{align}
 Hence the Dobrushin criterion is satisfied for $q$ sufficiently close to 1 and all $p\in [0,1]$.
\end{proof}
\begin{remark}
\begin{enumerate}
\item
Note that the gradient-gradient correlations in gradient models at best only decay critically with $|x|^{-d}$ (which is the decay rate for the discrete Gaussian free field). 
In particular, the sum of the covariance $\sum_{g\in \Edge(\Z^d)} \mathop{Cov}(\eta_f,\eta_g)$ diverges in this type of model.
We use crucially in the previous theorem that the decay of correlations is better for the discrete model: They decay with the square of the gradient-gradient correlations.
\item
The averaged (annealed) second order derivative of the Greens functions decays
with the optimal decay rate $|x|^{-d}$ as shown in  \cite{MR2198017}. For the application of the Dobrushin criterion we, however need deterministic bounds which are weaker.
\item To extend the uniqueness result for $q$ close to 1 to dimensions $d=3$ and $d=2$ one would need estimates for the optimal Hölder exponent $\alpha$ depending on the ellipticity contrast of discrete elliptic operators.
 Here the ellipticity contrast can be bounded by  $q$. There do not seem to be any results in this direction in the discrete setting. In the continuum setting the problem is open for $d\geq 3$, but has been solved for $d=2$ in \cite{MR0361422}. In this case  $\alpha\to 1$ 
as the ellipticity contrast converges to 1. A similar result in the discrete setting would imply uniqueness of the Gibbs measure for small $q$ in dimension 2.
\end{enumerate}
\end{remark}
Note that we can again lift the uniqueness result for the Gibbs measure of the random conductance model to a uniqueness result for the ergodic gradient Gibbs measures with zero tilt.
\begin{proof}[Proof of Theorem \ref{th:Dobrushin}]
The proof follows from the uniqueness of the discrete Gibbs measure proven in Theorem \ref{th:Dobrushin_discrete} in the same way 
as the proof of Theorem \ref{th:main_unique} which can be found above Remark 
\ref{rem:conv_pressure}.
\end{proof} 

\paragraph{Open questions} Let us end this section by stating one further result and two conjectures regarding the phase transitions of this model.
They are most easily expressed in terms of percolation properties of the model even though 
the interpretation as open and closed bonds is somehow misleading in this context.
We write $x\leftrightarrow y$ for $x,y\in \Z^d$ and $\kappa$ if there is a path of $q$-bonds in $\kappa$
connecting $x$ and $y$ and similarly for sets.
Observe that the results of \cite{MR3898174} can be applied to the model introduced here and we obtain the existence of a  sharp phase transition.
\begin{theorem}\label{th:sharpness}
For every $q$ the model undergoes a sharp phase transition in $p$, i.e., there is $p_c(q,d)$ such that
the following two properties hold. On the one hand there is a constant $c_1>0$ such that for $p>p_c$ sufficiently close to $p_c$  
\begin{align}\label{eq:sharpness_lower}
\discmu^1(0\leftrightarrow \infty)\geq c_1(p-p_c).
\end{align}
On the other hand, for $p<p_c$  there is a constant $c_p$ such that
\begin{align}\label{eq:sharpness_upper}
\discmu^1_n(0\leftrightarrow \partial\Lambda_n)\leq e^{-c_pn}.
\end{align}
\end{theorem}
\begin{proof}
The proof of Theorem 1.2 in \cite{MR3898174} for the random cluster model applies to this model.
Indeed, it only relies on $\mu_{n,p}^1$ being strongly positively associated and a certain relation 
for the $p$ derivative of events stated in Theorem 3.12 in \cite{MR2243761} 
which is still true since the $p$-dependence is the same as for the random cluster model.
\end{proof}
\begin{remark}
For $d=2$ the self dual point defined in \eqref{eq:self_dual_equation} and the critical point agree: $p_c=p_{\sd}$.
This can be seen based on Theorem 1.5 and the arguments used in the proof of Theorem 1.4 
in \cite{MR3898174}  for the random cluster model .
\end{remark}
In the random cluster model the most interesting phenomena happen for $p=p_c$ and the subcritical and supercritical phase are much simpler to understand (in particular in $d=2$).
Due to the differences explained in Remark  \ref{rem:sim_random-cluster} those
questions seem to be harder for our random conductance model.
Nevertheless we conjecture the 
following stronger version of Theorem \ref{th:countable_non_uniqueness}
and Theorem \ref{th:Dobrushin_discrete}
 \begin{conjecture}\label{con:uniqueness}
 For $p\neq p_c$ there is a unique Gibbs measure.
\end{conjecture}
Note that the sharpness result Theorem~\ref{th:sharpness}
shows that the probability of subcritical $q$-clusters to be large is exponentially small. Nevertheless it is not clear how this can be used to show uniqueness of the Gibbs measure in our setting.

The behaviour at $p_c$ is also very interesting.
A phase transition is called continuous 
if  $\mu_{p_c}^1(0\leftrightarrow \infty)=0$ and otherwise it is discontinuous.
For the random cluster model in dimensions $d=2$ the phase transition is continuous for $q\leq 4$ and otherwise discontinuous.
Moreover, the uniqueness of the Gibbs measure
 at $p_c$ is equivalent to a continuous phase transition. 
We do not know whether the same is true for the random conductance model considered here. But we expect the general picture to be true also for the random conductance and we make this precise in a  second conjecture.
 \begin{conjecture}\label{con:cont_discont}
 There is a $q_0=q_0(d)$ such that for $q> q_0$ there is non-uniqueness of Gibbs measures
 $\discmu_{p_c,q}^1\neq \discmu_{p_c,q}^0$ at the critical point while for $q<q_0$
 the Gibbs measures agree, i.e., $\discmu_{p_c,q}^1=\discmu_{p_c,q}^0$.
 \end{conjecture}
 A partial result in the direction of this conjecture is Theorem \ref{th:randomconductance_non_unique}
 that states non-uniqueness for large $q$ in dimension $d=2$  and Theorem \ref{th:Dobrushin_discrete} that shows
 uniqueness for $q$ close to 1 and $d\geq 4$.

\section{
Duality and coexistence of Gibbs measures}\label{sec:phase_transition}

In this section we are going to prove that 
$\mu^0_{p_{\sd}}\neq\mu^1_{p_{\sd}}$ for large $q$ which implies the non-uniqueness of gradient Gibbs measures  stated in Theorem \ref{th:main_non_unique}. This is a new proof for the result in  \cite{MR2322690}. 
They consider 
conductances $q_1$, $q_2$ with $q_1q_2=1$ which makes the presentation slightly more symmetric.

In contrast to their work we do not rely on reflection positivity but instead we exploit the
planar duality that is already used in \cite{MR2322690} to find the location of the phase transition. Therefore it is not possible to extend the argument given here to $d\geq 3$ while the proof using reflection positivity is in principle independent of the dimension (note that the spin wave calculations in \cite{MR2322690} can be simplified substantially and generalised to $d\geq 3$ using the Kirchhoff formula cf.\ \cite[Section 5.7]{myphd}). 
In addition to planar duality we rely on the properties proved in Section \ref{sec:percolation}, 
in particular on the Kirchhoff formula. 
Similar arguments were developed in the context of the random cluster model and we refer to
\cite[Section 6 and 7]{MR2243761}.

We proceed now by stating the duality property in our setting.
For a planar graph $G=(V,E)$ we denote its dual graph by
$G^\ast=(V^\ast,E^\ast)$. The dual graph has the faces of $G$ as vertices and the vertices of $G$ as 
faces and each edge has a corresponding dual edge.
For a formal definition of the dual of a graph and the necessary background we refer
to the literature, e.g., \cite{MR1813436}.

For any configuration $\kappa:E\to \{1,q\}$
we define its dual configuration $\kappa^\ast\in \{1, q\}^{E^\ast}$ by 
$\kappa^\ast_{e^\ast}=1+q-\kappa_e$ where $e^\ast\in E^\ast$ denotes the dual edge of an edge $e\in E$.
More generally we denote for $E_1\subset E$ by $E_1^\ast=\{e^\ast\,:e\in E_1\}$ the dual edges of the
edges $E_1$.
We also introduce the notation
 $E_1^{\d}=\{e^\ast \in E^\ast: e\notin E_1\}=(E_1^\crm)^\ast$ for $E_1\subset E$  for the dual set of an edge subset.
Note that  $E_1$ is acyclic if and only if $E_1^{\d}$ is spanning, i.e., every two points $x^\ast, y^\ast\in V^\ast$ are connected by a path in $E_1^{\d}$. 
In particular,  $\t\subset E$ is a spanning tree in $G$ if and only if $\t^{\d}$ is a spanning tree in $G^\ast$ and the map $\t\mapsto \t^d$ is an involution and in particular bijective from $\ST(G)$ to $\ST(G^\ast)$.

Recall that $h(\kappa,\t)=| \{e\in \t: \kappa_e=q\} |$ denotes the number of $q$-bonds in the set $t\subset \Edge(G)$ of $\kappa$
and the similar definition of $s(\kappa,\t)$ for the number of soft $1$-bonds in $\t$.
The definitions imply that
\begin{alignat}{2}\label{eq:duality_h_s}
h(\kappa) &= s(\kappa^\ast),
&s(\kappa)&=h (\k^\ast),
\\ \label{eq:duality_h_s_s}
h(\k, \t)&= s(\k^\ast)-s(\k^\ast,\t^{\d}),\qquad
&\qquad s(\k,\t)&= h(\k^\ast)-h(\k^\ast,\t^{\d}).
\end{alignat}
The last two identities follow from the observation that
$s(\k^\ast,\t^{\d})=h(\kappa, E\setminus \t)$ and similarly for $s$ and $h$ interchanged.
We calculate the distribution of $\kappa^\ast$ if $\kappa$ is distributed according to $\P^{G,p}$
\begin{align}\label{eq:proba_duality}
\begin{split}
\P(\k^\ast)=\P(\k)
\propto \!\frac{p^{h(\k)}(1-p)^{s(\k)}}{
\sqrt{\sum_{\t \in \ST(G)} q^{h(\k,\t)}}}
&=\!
\frac{p^{s(\k^\ast)}(1-p)^{h(\k^\ast)}}{
\sqrt{\sum_{\t^{\d}\in \ST(G^\ast)} q^{s(\k^\ast)-s(\k^\ast,\t^{\d})}}}
=\! \frac{ \left(\frac{p}{\sqrt{q}}\right)^{s(\k^\ast)}
(1-p)^{h(\k^\ast)}}
{\sqrt{\sum_{\t^{\d} \in \ST(G^\ast)} q^{h(\k^\ast,\t^{\d})-|\t^{\d}|}}}.
\end{split}
\end{align} 
This implies that if
$\kappa$ is distributed according to $\P^{G,p}$ 
the dual configuration $\kappa^\ast$ 
is distributed according to $\P^{G^\ast, p^\ast}$ where $q^\ast=q$ and 
\begin{align}\label{eq:p_ast}
\frac{p^\ast}{1-p^\ast}=\frac{(1-p)}{{p}/{\sqrt{q}}}.
\end{align}
Note that the self dual point $p_{\sd}$ defined by $p_{\sd}^\ast=p_{\sd}$ is given by the solution of
\begin{align}\label{eq:self_dual_point}
 \frac{p^4}{(1-p)^4}=q.
\end{align}
%
We will now restrict our attention to $\Z^2$. 
Let us mention that detailed proofs of the topological statements we use can be found in  \cite{MR692943}.
 
We can identify the dual of the graph 
$(\Z^2,\Edge(\Z^2))$, which will be denoted by $((\Z^2)^\ast,\Edge(\Z^2)^\ast)$, with $\Z^2$ shifted by the vector $w=(\tfrac12,\tfrac12)$. 
We also consider the set of directed bonds $\vec{\Edge}(\Z^2)$ and $\vec{\Edge}(\Z^2)^\ast$.
For a directed bond $\vec{e}=(x,y)\in\Edge(\Z^2)$ we define its dual bond as the directed bond $\vec{e}^\ast=(\tfrac12(x+y+(x-y)^\perp),\tfrac12 (x+y+(y-x)^\perp)$
where $\perp$ denotes counter-clockwise rotation by $90^\circ$, i.e., the linear map that
satisfies $e_1^\perp=e_2$, $e_2^\perp=-e_1$. In other words, the dual of a directed bond $\vec{e}$ is the bond whose orientation is rotated by $90^\circ$ counter-clockwise and crosses $\vec{e}$.

Every point $x\in \Z^2$ determines a plaquette with corners $z_1,z_2,z_3,z_4\in (\Z^2)^\ast$
where $z_i$ are the four nearest neighbours of $x$ in $(\Z^2)^\ast$ and
the plaquette has faces
$e_1^\ast,e_2^\ast, e_3^\ast, e_4^\ast\in \Edge(\Z^2)^\ast$ where $e_i^\ast$ are the dual bonds of the
four bonds $e_i$ that are incident to $x$. Vice versa every point $z\in (\Z^2)^\ast$ determines a plaquette in $\Z^2$.
We write $\Plaq(\Z^2)$ for the set of plaquettes of $\Z^2$.

 For a bond $e=\{x,y\}$ we define the shifted dual bond $e+w=\{x+w,y+w\}$. Similarly,  we define
$E+w=\{e+w\in \Edge(\Z^2)^\ast\,: \,e\in E\}$ for a set $E\subset \Edge(\Z^2)$.  
For a subgraph $G\subset \Z^2$ we denote
by $\Plaq(G)=\{P\in \Plaq(\Z^2): \text{all faces of $P$ are in $\Edge(G)$}\}$ the plaquettes of $G$. 
A subgraph $G\subset \Z^2$ is called simply connected if the union of all vertices $v\in \Vertex(G)$,
all edges $\{x,y\}\in \Edge(G)$ which are identified with the line segment from $x$ to $y$ in $\R^2$ and
all plaquettes $\Plaq(G)$ is a simply connected subset of $\R^2$.

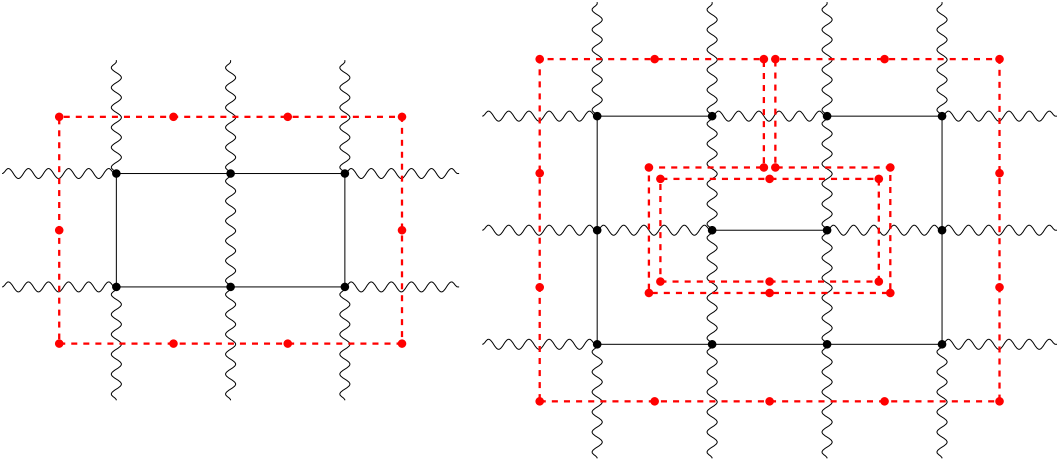
\begin{figure}
\begin{minipage}{0.4\textwidth}
\resizebox{\textwidth}{!}{
\begin{tikzpicture}
\foreach \x in {1,2,3}
{
\draw[wiggly] (2*\x,0) -- (2*\x,2);
\draw[wiggly] (2*\x,4) -- (2*\x,6);
}
\foreach \x in {1,2}
{
\draw[wiggly] (0,2*\x) --(2,2*\x);
\draw[wiggly] (6,2*\x) --(8,2*\x);
}
\draw[wiggly] (4,2) --(4,4);
\foreach \x in {1,2}
{
\draw (2*\x,2) --(2*\x+2,2);
\draw (2*\x,4) --(2*\x+2,4);
}
\draw (2,2) --(2,4);
\draw (6,2) --(6,4);
\foreach \x in {1,2,3}
\foreach \y in {1,2}
{
\draw (2*\x,2*\y) node[circle,fill,inner sep=1.5pt]{};
}
\draw[red, dashed, line width=0.04cm] (1,1)--(3,1)--(5,1)--(7,1)--(7,5)--(1,5)--cycle;
\foreach \x in {1,2,3,4}
{
\draw (2*\x-1,1) node[red,circle,fill,inner sep=1.5pt]{};
\draw (2*\x-1,5) node[red,circle,fill,inner sep=1.5pt]{};
}

\draw (1,3) node[red,circle,fill,inner sep=1.5pt]{};
\draw (7,3) node[red,circle,fill,inner sep=1.5pt]{};
\end{tikzpicture}
}
\end{minipage}
\begin{minipage}{0.5\textwidth}
\resizebox{\textwidth}{!}{
\begin{tikzpicture}
\foreach \x in {2,4,6,8}
\foreach \y in {0,6}
{
\draw[wiggly] (\x,\y)--(\x,\y+2);
}
\foreach \x in {4,6}
\foreach \y in {2,4}
{
\draw[wiggly] (\x,\y)--(\x,\y+2);
}
\draw[wiggly] (2,4)--(4,4);
\draw[wiggly] (6,4)--(8,4);
\draw[wiggly] (4,6)--(6,6);

\foreach \x in {0,8}
\foreach \y in {2,4,6}
{
\draw[wiggly] (\x,\y)--(\x+2,\y);
}
\draw (4,6)--(2,6)--(2,2)--(8,2)--(8,6)--(6,6);
\foreach \x in {2,4,6,8}
\foreach \y in {2,4,6}
{
\draw (\x,\y) node[circle,fill,inner sep=1.5pt]{};
}
\draw (4,4)--(6,4);

\draw[red, dashed, line width=0.04cm] (1,1)--(9,1)--(9,7)--(5.1,7)--(5.1,5.1)--(7.1,5.1)--(7.1,2.9)--(2.9,2.9)--(2.9,5.1)--(4.9,5.1)--(4.9,7)--(1,7)--cycle;
\draw[red, dashed, line width=0.04cm] (3.1,3.1)--(6.9,3.1)--(6.9,4.9)--(3.1,4.9)--cycle;
\draw  (1,1) node[red, circle,fill,inner sep=1.5pt]{}(9,1)node[red, circle,fill,inner sep=1.5pt]{}(9,7)node[red, circle,fill,inner sep=1.5pt]{}(5.1,7)node[red, circle,fill,inner sep=1.5pt]{}(5.1,5.1)node[red, circle,fill,inner sep=1.5pt]{}(7.1,5.1)node[red, circle,fill,inner sep=1.5pt]{}(7.1,2.9)node[red, circle,fill,inner sep=1.5pt]{}(2.9,2.9)node[red, circle,fill,inner sep=1.5pt]{}(2.9,5.1)node[red, circle,fill,inner sep=1.5pt]{}(4.9,5.1)node[red, circle,fill,inner sep=1.5pt]{}(4.9,7)node[red, circle,fill,inner sep=1.5pt]{}(1,7) node[red, circle,fill,inner sep=1.5pt]{};
\draw (3.1,3.1)node[red, circle,fill,inner sep=1.5pt]{}(6.9,3.1)node[red, circle,fill,inner sep=1.5pt]{}(6.9,4.9)node[red, circle,fill,inner sep=1.5pt]{}(3.1,4.9)node[red, circle,fill,inner sep=1.5pt]{}
(5,2.9)node[red, circle,fill,inner sep=1.5pt]{}(5,3.1)node[red, circle,fill,inner sep=1.5pt]{}(5,4.9)node[red, circle,fill,inner sep=1.5pt]{};
\foreach \x in {1,9}
\foreach \y in {3,5}
{
\draw (\x,\y) node[red,circle,fill,inner sep=1.5pt]{};
}
\foreach \x in {3,7}
\foreach \y in {1,7}
{
\draw (\x,\y) node[red,circle,fill,inner sep=1.5pt]{};
}
\draw (5,1) node[red,circle,fill,inner sep=1.5pt]{};
\end{tikzpicture}
}
\end{minipage}
\caption{Examples of $q$-contours. In the second example there are two nested $q$-contours.
Curly bonds indicate soft bonds with $\kappa_e=1$ and straight bonds indicate hard bonds with $\kappa_e=q$. In red the dual bonds of the contours are shown. The horizontal curly bond in the top middle connects
two point in $\Int \,\gamma$.}\label{fig1}
\end{figure}
An important tool in the analysis of planar models from statistical mechanics is the use of contours. Let us provide a notion of contours that is useful for our purposes. Our definition is slightly
more complicated than the definition of contours for the random cluster model.
We consider closed paths $\gamma=(x_1^\ast,\ldots,x_n^\ast,x_1^\ast)$ with $x_i^\ast\in (\Z^2)^\ast$  (not necessarily all distinct)
along pairwise distinct directed dual bonds $\vec{b}_1^\ast=(x_1^\ast,x_2^\ast),\ldots, \vec{b}_n^\ast=(x_n^\ast,x_1^\ast)$.
We denote the vertices in the contour by $\Vertex(\gamma)^\ast=\{ x_i^\ast\, :\, 1\leq i\leq n\}$
and the bonds by $\vec{\Edge}(\gamma)^\ast=\{\vec{b}_i^\ast\,:\, 1\leq i\leq n\}$. Similarly 
we write 
$\vec{\Edge}(\gamma)=\{\vec{b}_i\,:\, 1\leq i\leq n\}$ for the corresponding primal bonds.
We also consider the underlying sets of undirected bonds $\Edge(\gamma)$ and $\Edge(\gamma)^\ast$.
Finally, we denote the heads and tails of $\vec{b}_i$ by $y_i$ and $z_i$, i.e., $\vec{b}_i=(z_i,y_i)$.
\begin{define}\label{def:contour}
A contour $\gamma$ is a closed path in the dual lattice without self-crossings in
the sense that
there is a bounded connected component $\Int(\gamma)$ of the graph $(\Z^2,\Edge(\Z^2)\setminus \Edge(\gamma))$ 
such that $\partial (\Int(\gamma))=\{z_i\,:\, 1\leq i\leq n\}$.
We denote the union of the remaining connected components by $\Ext(\gamma)$ and 
we define the length $|\gamma|$ of the contour as the number of (directed) bonds it contains, i.e.,  $|\gamma|=|\vec{\Edge}(\gamma)|=n$.
\end{define}
Note that $\Ext(\gamma)$ is not necessarily connected and that $\{x,y\}\in \Edge(\gamma)$ if $x\in \Int(\gamma)$ and $y\in \Ext(\gamma)$ (see Figure \ref{fig1}). 

Contours are a suitable notion to define interfaces between hard and soft bonds.
\begin{define}\label{def:q_contour}
A contour $\gamma$ is a $q$-contour for $\kappa$ if  the following two conditions hold.
First, the  primal bonds $b\in \Edge(\gamma)$ are soft, i.e., $\kappa_{b}=1$.
Moreover, for every plaquette  with center 
$x^\ast \in \Vertex(\gamma)^\ast$  all its faces $b$ such that $b\in \Edge(\Int(\gamma))$  are hard, i.e., satisfy $\kappa_b=q$.
\end{define}
Our goal is to show that $q$-contours are unlikely for large values of $q$ and $p\leq p_{\sd}$.
We now fix a contour $\gamma$ and introduce some useful notation and  helpful observations for the proof of the following theorem. We use the shorthand $G_i=\Int(\gamma)$ and  $E_i=\Edge(\Int(\gamma))$.
We observe that $G_i$ is simply connected because $\gamma$ is connected and without self-crossings.
Therefore the faces of $G_i$ consist of plaquettes in $\Z^2$  and one infinite face. 
We also consider the graph $G$ with edges $E=E_i\cup \Edge(\gamma)$ and endpoints of edges as vertices. Let $\bar{1}\in \{1,q\}^{E}$ denote the configuration given by $\bar{1}_e=1$ for all $e\in E$. 
 We write $G^w=G/\partial G=G/ \Ext(\gamma)$ for the graph $G$ with wired boundary conditions.
Moreover we introduce the graph $H^\ast$ with edges $E_i^\ast$ and their endpoints as vertices.
We claim that $H^\ast/\partial H^\ast$ agrees with the graph theoretic dual of $G_i$.
To show this we need to prove that we identify all vertices that lie in the same face of $G_i$.
First we note that every point in $\overcirc{H}^\ast=H^\ast\setminus \partial H^\ast$ 
determines a plaquette in $\mathbf{P}(G_i)$ and this is a bijection.
Then it remains to show that all vertices in $\partial H^\ast$ lie in the infinite face of $G_i$.
This follows from the observation
\begin{align}\label{eq:obs1}
\partial H^\ast = \Vertex(\gamma)^\ast\cap \Vertex(H^\ast).
\end{align}
To show the observation we note that if $x^\ast\in \partial H^\ast$ then 
there are edges $e^\ast_1 \notin \Edge(H^\ast)$ and $e^\ast_2\in \Edge(H^\ast)$ incident to $x^\ast$.
This implies that there is a face $e'=\{z_1,z_2\}$
of the plaquette with center $x^\ast$ such that $z_1\in \Vertex(G_i)$ but $e'\notin \Edge(G_i)$.
Then $e\in \Edge(\gamma)$ and therefore $x^\ast\in \Vertex(\gamma)^\ast$ because $x^\ast$ is an endpoint of $e^\ast\in \Edge(\gamma)^\ast$. This ends the proof of the inclusion '$\subset$'.
Now we note that if $x^\ast\in  \Vertex (H^\ast)\cap \Vertex(\gamma)^\ast$ 
there is an edge $e^\ast\in \Edge(\gamma)^\ast$ incident to $x^\ast$ which
is not contained in $\Edge(H^\ast)$ and therefore $x^\ast\in \partial H^\ast$.
 
 Finally we remark that if $\gamma$ is a $q$-contour for $\kappa$  then
\begin{align}\label{eq:observation_partialH}
\kappa^\ast_{e^\ast}=1+q-\kappa_e=1 \qquad \text{if $e^\ast\in \Edge(H^\ast)$ is incident to $ \partial H^\ast.$} 
\end{align}  
  Indeed, we argued above that if $e^\ast\in \Edge(H^\ast)$ is incident to $ \partial H^\ast$ then 
  $x^\ast\in \Vertex(\gamma)^\ast$. Thus $e\in E_i$ is a face of the plaquette
  with center $x^\ast$ so that the definition of $q$-contours implies that $\kappa_{e}=q$.  
 
\begin{theorem}\label{th:contour}
Let $\gamma$ be a contour. The probability that $\gamma$ is
a $q$-contour under the measure $\P^{G^w, E_i,\bar{1}}$
for $p=p_{\sd}$ is bounded by
\begin{align}\label{eq:contour_probability}
\P^{G^w,E_i,\bar{1}}(\text{$\gamma$ is a $q$-contour})\leq \left(\frac{4}{q^{\frac18}}\right)^{|\gamma|} q^{\frac{1}{2}}.
\end{align} 
\end{theorem}
\begin{remark}
The general idea of the proof is the same as when proving similar estimates for the Ising model.
One tries to find a map from configurations where the contour is present to configurations where this is not the case and then estimates the corresponding probabilities.
The more similar argument for the random cluster model can be found, e.g., in Theorem 6.35 in \cite{MR2243761}. For an  illustrated version see \cite{duminil2017lectures}.
\end{remark}
\begin{proof}
 
\begin{figure}
\begin{minipage}{0.48\textwidth}
\centering
\resizebox{0.6\textwidth}{!}{

\begin{tikzpicture}

\draw[dashed] (1,1)--(7,1)--(7,7)--(3,7)--(3,5)--(1,5)--cycle;

\draw[wiggly] (2,0)--(2,2);

\draw[wiggly] (4,0)--(4,2);
\draw[wiggly] (6,0)--(6,2);
\draw[wiggly] (0,2)--(2,2);
\draw[wiggly] (6,2)--(8,2);
\draw[wiggly] (0,4)--(2,4);
\draw[wiggly] (2,6)--(4,6);
\draw[wiggly] (6,4)--(8,4);
\draw[wiggly] (6,6)--(8,6);
\draw[wiggly] (4,6)--(4,8);
\draw[wiggly] (6,6)--(6,8);
\draw[wiggly] (2,4)--(2,6);

\draw[wiggly] (4,2)--(4,4);
\draw[wiggly] (4,4)--(6,4);

\draw (2,2 )--(6,2);
\draw (2,2 )--(2,4);
\draw (2,4 )--(4,4)--(4,6)--(6,6)--(6,2);

\draw[red] (3,3)--(5,3)--(5,5);
\draw[red,wiggly] (3,3)--(3,5)--(5,5)--(5,7) (5,5)--(7,5) (1,3)--(3,3) (3,1)--(3,3);
\draw[red,wiggly] (5,3)--(7,3);
\draw[red,wiggly] (5,1)--(5,3);
\draw  (5,5) node[red,circle,fill,inner sep=2.5pt]{};
\draw  (3,5) node[red,circle,fill,inner sep=2.5pt]{};
\draw  (3,3) node[red,circle,fill,inner sep=2.5pt]{};
\draw  (5,3) node[red,circle,fill,inner sep=2.5pt]{};

\draw  (2,2) node[circle,fill,inner sep=2.5pt]{};
\draw  (2,4) node[circle,fill,inner sep=2.5pt]{};
\draw  (6,2) node[circle,fill,inner sep=2.5pt]{};
\draw  (4,2) node[circle,fill,inner sep=2.5pt]{};
\draw  (4,4) node[circle,fill,inner sep=2.5pt]{};
\draw  (4,6) node[circle,fill,inner sep=2.5pt]{};
\draw  (6,4) node[circle,fill,inner sep=2.5pt]{};
\draw  (6,6) node[circle,fill,inner sep=2.5pt]{};

\end{tikzpicture}

}
\end{minipage}
\begin{minipage}{0.48\textwidth}
\centering
\resizebox{0.6\textwidth}{!}{

\begin{tikzpicture}

\draw[wiggly,blue] (2,4)--(4,4)--(4,2) (6,2)--(6,4)--(8,4) (4,4)--(4,6)--(6,6) (6,8)--(6,6)--(8,6);
\draw[blue]  (6,6)--(6,4)--(4,4);

\draw[wiggly] (4,0)--(4,2);
\draw[wiggly] (6,0)--(6,2);
\draw[wiggly] (0,2)--(2,2);
\draw[wiggly] (6,2)--(8,2);
\draw[wiggly] (0,4)--(2,4);
\draw[wiggly] (2,6)--(4,6);
\draw[wiggly] (4,6)--(4,8);
\draw[wiggly] (2,4)--(2,6);
\draw[wiggly] (2,0)--(2,2);

\draw[wiggly] (6,2)--(2,2 )--(2,4);

\draw  (2,2) node[circle,fill,inner sep=2.5pt]{};
\draw  (2,4) node[circle,fill,inner sep=2.5pt]{};
\draw  (6,2) node[circle,fill,inner sep=2.5pt]{};
\draw  (4,2) node[circle,fill,inner sep=2.5pt]{};
\draw  (4,4) node[blue,circle,fill,inner sep=2.5pt]{};
\draw  (4,6) node[circle,fill,inner sep=2.5pt]{};
\draw  (6,4) node[blue,circle,fill,inner sep=2.5pt]{};
\draw  (6,6) node[blue,circle,fill,inner sep=2.5pt]{};

\end{tikzpicture}

}
\end{minipage}
\caption{(left) The dashed line indicates a $q$-contour for the depicted configuration. In red the dual configuration on $E^\ast$  for the edges $E$ is shown.
(right) The configuration $\kappa^\#$ for $\kappa$ depicted on the left. The blue edges are the 
shifted dual edges forming the set $\tilde{E}$. }
\label{fig2}
\end{figure}
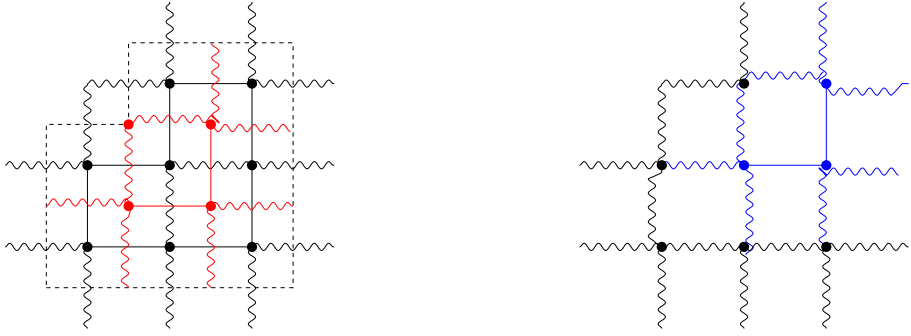
 We denote the set of all $\kappa\in \{1,q\}^{E}$ such that $\gamma$ is a $q$-contour for $\kappa$
by $\Omega_\gamma$.

\textbf{Step 1.} We define a map $\Phi:\Omega_\gamma\to \{1,q\}^{E}$ with $\Phi(\kappa)=\kappa^\#$ as follows.
Recall the definition of the dual configuration $\kappa^\ast$ on $E^\ast\subset \Edge(\Z^2)^\ast$ 
 and define for $e\in E$
\begin{align}
\kappa^{\#}_e=\begin{cases}
\kappa^\ast_{e-w} \quad&\text{if $e-w\in {E_i}^\ast$,}
\\
1\quad &\text{otherwise}.
\end{cases} 
\end{align}
We claim that 
\begin{align}\label{eq:kappa_sharp_bc1}
\kappa^{\#}_e=1 \qquad \text{if $e\in E\setminus E_i$}.
\end{align}
By definition of $\kappa^{\#}$, we only need to consider the case $e-w\in {E_i}^\ast=\Edge(H^\ast)$. 
We will show a slightly more general statement.
Let us introduce the set 
$\tilde{E}=\Edge(H^\ast)+w=E_i^\ast+w\subset E$ and the graph $\tilde G $ consisting of the edges $\tilde E$ and their endpoints as vertices.   See Figure \ref{fig2} for an illustration of this construction.
We remark that $\tilde{G}$ agrees with $H^\ast$ shifted by $w$, which we denote by $\tilde{G}=H^\ast+w$.
Equation  \eqref{eq:observation_partialH} implies
that 
\begin{align}\label{eq:kappa_sharp_bc2}
\kappa^{\#}_e=\kappa^\ast_{e-w}=1\qquad \text{for $e\in \tilde{G}$ incident to $\partial \tilde{G}$}
\end{align}
because then $e-w\in \Edge(H^\ast)$ is incident to $\partial H^\ast$.
It remains to show that all edges  $e\in E\cap \tilde{E}\setminus E_i$
are incident to $\partial\tilde{G}$. 
From $e\in E\setminus E_i$ we conclude that $e\in \Edge(\gamma)$.
The edge $e-w$ has a common endpoint with $e^\ast\in \Edge(\gamma)^\ast$ and is therefore incident to $\Vertex(\gamma)^\ast$ in this case.
Using the observation \eqref{eq:obs1} this implies that $e-w\in \Edge(H^\ast)$ is incident to $\partial H^\ast$.

 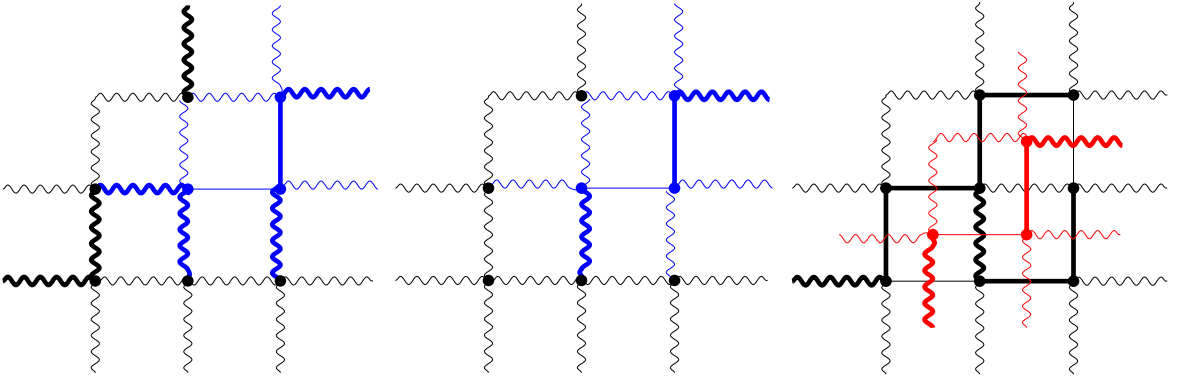
\begin{figure}
\begin{minipage}{0.33\textwidth}
\resizebox{\textwidth}{!}{

\begin{tikzpicture}

\draw[wiggly,blue] (4,6)--(6,6)  (6,4)--(8,4) (4,4)--(4,6)  (6,8)--(6,6) ;

\draw[wiggly,blue, line width=\thick](2,4)--(4,4) (6,2)--(6,4) (6,6)--(8,6) (4,4)--(4,2) ;

\draw[blue]  (6,4)--(4,4);
\draw[blue, line width=\thick]  (6,6)--(6,4);

\draw[wiggly] (6,0)--(6,2);
\draw[wiggly] (6,2)--(8,2);
\draw[wiggly] (0,4)--(2,4);
\draw[wiggly] (2,6)--(4,6);
\draw[wiggly] (2,4)--(2,6);
\draw[wiggly] (2,0)--(2,2);
\draw[wiggly] (4,0)--(4,2);

\draw[wiggly, line width=\thick] (0,2)--(2,2);
\draw[wiggly, line width=\thick] (4,6)--(4,8);

\draw[wiggly] (6,2)--(2,2 );

\draw[wiggly, line width=\thick] (2,2)--(2,4);

\draw  (2,2) node[circle,fill,inner sep=2.5pt]{};
\draw  (2,4) node[circle,fill,inner sep=2.5pt]{};
\draw  (6,2) node[circle,fill,inner sep=2.5pt]{};
\draw  (4,2) node[circle,fill,inner sep=2.5pt]{};
\draw  (4,4) node[blue,circle,fill,inner sep=2.5pt]{};
\draw  (4,6) node[circle,fill,inner sep=2.5pt]{};
\draw  (6,4) node[blue,circle,fill,inner sep=2.5pt]{};
\draw  (6,6) node[blue,circle,fill,inner sep=2.5pt]{};

\end{tikzpicture}

}
\end{minipage}
\begin{minipage}{0.33\textwidth}
\resizebox{\textwidth}{!}{

\begin{tikzpicture}

\draw[wiggly,blue] (4,6)--(6,6)  (6,4)--(8,4) (6,2)--(6,4) (4,4)--(4,6)  (6,8)--(6,6) (2,4)--(4,4) ;

\draw[wiggly,blue, line width=\thick] (6,6)--(8,6) (4,4)--(4,2) ;

\draw[blue]  (6,4)--(4,4);
\draw[blue, line width=\thick]  (6,6)--(6,4);

\draw[wiggly] (6,0)--(6,2);
\draw[wiggly] (6,2)--(8,2);
\draw[wiggly] (0,4)--(2,4);
\draw[wiggly] (2,6)--(4,6);
\draw[wiggly] (2,4)--(2,6);
\draw[wiggly] (2,0)--(2,2);
\draw[wiggly] (4,0)--(4,2);

\draw[wiggly] (0,2)--(2,2);
\draw[wiggly] (4,6)--(4,8);

\draw[wiggly] (6,2)--(2,2 );

\draw[wiggly] (2,2)--(2,4);

\draw  (2,2) node[circle,fill,inner sep=2.5pt]{};
\draw  (2,4) node[circle,fill,inner sep=2.5pt]{};
\draw  (6,2) node[circle,fill,inner sep=2.5pt]{};
\draw  (4,2) node[circle,fill,inner sep=2.5pt]{};
\draw  (4,4) node[blue,circle,fill,inner sep=2.5pt]{};
\draw  (4,6) node[circle,fill,inner sep=2.5pt]{};
\draw  (6,4) node[blue,circle,fill,inner sep=2.5pt]{};
\draw  (6,6) node[blue,circle,fill,inner sep=2.5pt]{};

\end{tikzpicture}

}
\end{minipage}
\begin{minipage}{0.33\textwidth}
\resizebox{\textwidth}{!}{
\begin{tikzpicture}

\draw[wiggly] (2,0)--(2,2);

\draw[wiggly] (4,0)--(4,2);
\draw[wiggly] (6,0)--(6,2);
\draw[wiggly] (6,2)--(8,2);
\draw[wiggly] (0,4)--(2,4);
\draw[wiggly] (2,6)--(4,6);
\draw[wiggly] (6,4)--(8,4);
\draw[wiggly] (6,6)--(8,6);
\draw[wiggly] (4,6)--(4,8);
\draw[wiggly] (6,6)--(6,8);
\draw[wiggly] (2,4)--(2,6);

\draw[wiggly, line width=\thick] (0,2)--(2,2);

\draw[wiggly, line width=\thick] (4,2)--(4,4);
\draw[wiggly] (4,4)--(6,4);

\draw[line width=\thick] (4,2 )--(6,2) (2,2)--(2,4);
\draw (2,2 )--(4,2);

\draw (6,6)--(6,4);
\draw[line width=\thick] (2,4 )--(4,4)--(4,6)--(6,6) (6,4)--(6,2);

\draw[red] (3,3)--(5,3);
\draw[red, line width=\thick] (5,3)--(5,5);

\draw[red,wiggly] (3,3)--(3,5) (3,5)--(5,5) (5,5)--(5,7) (1,3)--(3,3);
\draw[red,wiggly] (5,3)--(7,3);
\draw[red,wiggly] (5,1)--(5,3);

\draw[red, wiggly, line width=\thick]		(5,5)--(7,5)  (3,1)--(3,3);

\draw  (5,5) node[red,circle,fill,inner sep=2.5pt]{};
\draw  (3,3) node[red,circle,fill,inner sep=2.5pt]{};
\draw  (5,3) node[red,circle,fill,inner sep=2.5pt]{};

\draw  (2,2) node[circle,fill,inner sep=2.5pt]{};
\draw  (2,4) node[circle,fill,inner sep=2.5pt]{};
\draw  (6,2) node[circle,fill,inner sep=2.5pt]{};
\draw  (4,2) node[circle,fill,inner sep=2.5pt]{};
\draw  (4,4) node[circle,fill,inner sep=2.5pt]{};
\draw  (4,6) node[circle,fill,inner sep=2.5pt]{};
\draw  (6,4) node[circle,fill,inner sep=2.5pt]{};
\draw  (6,6) node[circle,fill,inner sep=2.5pt]{};

\end{tikzpicture}

}
\end{minipage}
\caption{(left) An example of a wired tree $t^\#$ for $\kappa^\#$.
(center) The subtree $\tilde t$.
(right) The shifted tree $\tilde t-w$ and the dual tree 
$\Psi(t^\#)$.
}\label{fig3}
\end{figure}
 
 Our  goal is to compare the probabilities of $\P^{G^w,E_i,\bar{1}}(\kappa_{E_i})$  and 
$\P^{G^w,E_i,\bar{1}}(\kappa^{\#}_{E_i})$.
To achieve this we use a strategy similar to the proof of Lemma \ref{le:monotonicity_in_p}.

\textbf{Step 2.}
We define a map $\Psi:\ST(G^w)\to \ST(G^w)$ with $\Psi(\t^{\#})= \t$ in the following steps
\begin{enumerate}
\item We choose deterministically a subset $\tilde \t\subset \t^{\#}{\restriction}_{\tilde{E}}$
such that $\tilde \t$ is a spanning tree on $\tilde{G}/\partial\tilde G$
and  all edges in $\t^\#{\restriction}_{\tilde{E}}\setminus \tilde \t$ are incident to $\partial\tilde{G}$.
\item We set $\Psi(\t^\#){\restriction}_{E_i}=\{e\in E_i\, :\, e^\ast\notin \tilde \t-w\}= (\tilde{\t}-w)^d$
(as a subset of $E_i^\ast$).
\item We consider a fixed $b\in \Edge(\gamma)$ that is incident
to $\Int(\gamma)$ and $\Ext(\gamma)$ and we define  $\t=\Psi(\t^\#)=\Psi(\t^\#){\restriction}_{E_i}\cup b$
\end{enumerate}
 See Figure \ref{fig3} for an illustration of the construction.
 
We have to show that this construction is possible, in particular that $t\in \ST(G^w)$.
We start with the first step. 
The relation $\tilde{G}\subset G$ implies
\begin{align}\label{eq:boundary_incl}
\partial G\cap \tilde{G}\subset \partial\tilde{G}.
\end{align}
Hence $\tilde{G}/\partial\tilde{G}$ agrees with $(G/\partial G)/({\tilde{G}}^{\crm}\cup \partial\tilde{G})$ up to self loops. This implies that  $\t^{\#}{\restriction}_{\tilde{E}}$ is spanning in $\tilde G/\partial \tilde G$ if $\t^{\#}\in \ST(G^w)$.
We consider the subset $\t'\subset \t^{\#}{\restriction}_{\tilde{E}}$
consisting of all edges $e\in  \t^{\#}{\restriction}_{\tilde{E}}$ that are not incident to $\partial \tilde G$.
 The set $\t'$ contains no cycles because  $\t^{\#}\in \ST(G^w)$ and
no edge in $\t'$ is incident to $\partial G$ by \eqref{eq:boundary_incl}.
Therefore we can select a spanning tree $\tilde \t$ in $\tilde{G}/\partial{\tilde{G}}$ with $\t'\subset \tilde \t\subset
\t^{\#}{\restriction}_{\tilde{E}}$ deterministically, e.g., using Kruskal's algorithm.

We now argue that the second and third step yield a spanning tree in $G^w$.
Clearly it is sufficient to show that $\Psi(\t^\#){\restriction}_{E_i}\in \ST(G_i)$.
 We note that the relation between $\tilde{G}$ and $H^\ast$ implies
 that $\tilde{\t}-w$ is a spanning tree on $H^\ast/\partial H^\ast$. 
 As shown before the theorem $H^\ast/\partial H^\ast$ agrees with the dual 
 of $G_i$ and thus $(\tilde{\t}-w)^d \in \ST(G_i)$.

 \textbf{Step 3.}
The next step is to consider $\kappa^\#=\Phi(\kappa)$ and $\t=\Psi(\t^\#)$ and compare the weights 
 $w(\kappa^\#,\t^\#)$ and $w(\kappa,\t)$.
First we argue that 
\begin{align}
w(\kappa^\#,\t^\#)=w(\kappa^\#,\tilde{\t}).
\end{align}
Since $\tilde{\t}\subset \t^\#$ it is sufficient to show that $\t^\#\setminus \tilde \t$ contains
only edges $e$ such that $\kappa^\#_e=1$.  
 Indeed, let $e$ be an edge in $\t^\#\setminus \tilde \t$.
 For $e\notin \tilde E$ we have $\kappa^\#_e=1$ by definition.
 Let us now consider 
 \begin{align}
 e\in \tilde E\cap (\t^\#\setminus \tilde \t).
 \end{align}
 By construction of $\tilde \t$ the edge  $e$ is incident to a vertex $v\in \partial \tilde{G}$.
 This implies that $e-w\in \Edge(H^\ast)$ is incident to $v-w\in \partial H^\ast\subset \Vertex(\gamma)^\ast$. Using \eqref{eq:observation_partialH} we conclude that
 \begin{align}
 \kappa^{\#}_e=\kappa^\ast_{e-w}=1.
 \end{align}
 

For the trees $\tilde{\t}$ and $\Psi( \t^\#){\restriction}_E$ we can apply 
the usual duality relations stated before.
Using \eqref{eq:duality_h_s_s} and as before $\kappa^\#=\Phi(\kappa)$ and $\t=\Psi(\t^\#)$ we obtain
\begin{align}\label{eq:h_and_s_for_dual_trees}
h(\kappa^\#, \t^\#)=h(\kappa^\#, \tilde \t)
=h(\kappa^\ast, \tilde \t-w)=s(\kappa, E_i)-s(\kappa,E_i\cap \t).
\end{align}
We compute
\begin{align}\label{eq:weight_compare_prov}
\frac{w(\kappa^\#, \t^\#)}{w(\kappa,\t)}
=\frac{q^{h(\kappa^\#,\t^\#)}}{q^{h(\kappa,\t)}}
=q^{s(\kappa,E_i)-s(\kappa,\t\cap{E_i})-h(\kappa,\t\cap E_i)}
=q^{s(\kappa,E_i)-|\t\cap E_i|}
=q^{s(\kappa,E_i)-|\Vertex(G_i)|+1}.
\end{align}
In the last step we used that $\t\cap E_i$ is a free spanning tree on $G_i$
and therefore has $|\Vertex(G_i)|-1$ edges.

\textbf{Step 4.} We bound the number of preimages of a tree $\t$ under $\Psi$.
Note that $\Psi$ factorizes into two maps $\t^\#\to \tilde \t\to \t$. The second map is injective since we only pass to the dual tree which is an injective map and we add one additional edge. 
For the first map we observe that we only delete edges $e$ incident to $\partial \tilde G$.
However, for $x\in \partial\tilde G$ the point $x-w\in \partial H^\ast$ is contained in the contour
by  \eqref{eq:obs1}.
Therefore there are at most $4|\gamma|$ such edges.
We conclude that
\begin{align}\label{eq:number_preimages_tree}
|\{\t^\#\in \ST(G^w) : \Psi(\t^\#)=\t\}| \leq 2^{4|\gamma|}
\end{align} 
for every $\t\in \ST(G^w)$.
The displays \eqref{eq:weight_compare_prov} and \eqref{eq:number_preimages_tree}
imply
\begin{align}\label{eq:comparison_sum_spanning}
\sum_{\t^\#\in \ST(G^w)}\hspace{-0.2cm}w(\kappa^\#,\t^\#)
\leq \!\!\!
\sum_{\t^\#\in \ST(G^w)}\hspace{-0.2cm} w(\kappa, \Psi(\t^\#)) q^{ s(\kappa,E_i)-|\Vertex(G_i)|+1}
\leq q^{ s(\kappa,E_i)-|\Vertex(G_i)|+1} 2^{4|\gamma|}\!\!
\!\sum_{\t\in \ST(G^w)}\hspace{-0.2cm} w(\kappa, \t).
\end{align}

\textbf{Step 5.} We can now estimate the probabilities of the patterns $\kappa$ and $\kappa^\#=\Psi(\kappa)$ under $\P^{G^w,E_i,\bar{1}}$ using \eqref{eq:comparison_sum_spanning} 
and $\kappa^\#_e=1=\bar{1}_e$ for $e\in E\setminus E_i$
\begin{align}
\begin{split}\label{eq:probability_pattern}
\frac{\P^{G^w,E_i,\bar{1}}_{p_{\sd}}(\kappa_{E_i})}{\P^{G^w,E_i,\bar{1}}_{p_{\sd}}(\kappa^\#_{E_i})}
&=
\frac{\P^{G^w}_{p_{\sd}}(\kappa)}{\P^{G^w}_{p_{\sd}}(\kappa^\#)}
=\frac{p_{\sd}^{h(\kappa,E_i)}(1-p_{\sd})^{s(\kappa,E_i)}}
{p_{\sd}^{h(\kappa^\#,E_i)}(1-p_{\sd})^{s(\kappa^\#,E_i)}}
\sqrt{\frac{\sum_{\t^\#\in \ST(G^w)}w(\kappa^\#,\t^\#)}{\sum_{\t\in \ST(G^w)}w(\kappa,\t)}}
\\
&\leq 
\frac{\left(\frac{p_{\sd}}{1-p_{\sd}}\right)^{h(\kappa,E_i)}}
{\left(\frac{p_{\sd}}{1-p_{\sd}}\right)^{h(\kappa^\#,E_i)}}
q^{\frac{s(\kappa,E_i)-|\Vertex(G_i)|+1}{2}}2^{2|\gamma|}
=2^{2|\gamma|}q^{\frac{1}{4}(h(\kappa,E_i)-h(\kappa^\#,E_i))}q^{\frac{s(\kappa,E_i)-|\Vertex(G)|+1}{2}}
\end{split}
\end{align}
where we used equation \eqref{eq:self_dual_point} of $p_{\sd}$ in the last step.
The definition of $\kappa^\#$ implies that
$h(\kappa^\#,E_i)=h(\kappa^\#,\tilde{E})=s(\kappa,E_i)$ and we get 
\begin{align}\label{eq:probability_pattern2}
\frac{\P^{G^w,E_i,\bar{1}}(\kappa_{E_i})}{\P^{G^w,E_i,\bar{1}}(\kappa^\#_{E_i})}
\leq 2^{2|\gamma|}q^{\frac{1}{4}(h(\kappa,E_i)+s(\kappa,E_i)-2|\Vertex(G_i)|+2)}
=2^{2|\gamma|}q^{\frac{1}{4}(|E_i|-2|\Vertex(G_i)|+2)}
\end{align}
Now we observe that $4|\Vertex(G_i)|-2|\Edge(G_i)|=|\vec{\Edge}(\gamma)\}| =|\gamma|$. We end up with the estimate
\begin{align}\label{eq:probability_pattern3}
\frac{\P^{G^w,E_i,\bar{1}}(\kappa_{E_i})}{\P^{G^w,E_i,\bar{1}}(\kappa^\#_{E_i})}
\leq 2^{2|\gamma|}q^{-\frac{1}{8}|\gamma|} q^\frac12
= \left(4q^{-\frac18}\right)^{|\gamma|} q^{\frac12}.
\end{align}

\textbf{Conclusion.}
Note that the map $\Phi$ is injective, hence
\begin{align}\label{eq:probability_pattern4}
\P^{G^w,E_i,\bar{1}}(\text{$\gamma$ is a $q$-contour})
\leq 
\frac{\sum_{\kappa \in \Omega_\gamma} \P^{G^w ,E_i,\bar{1}}(\kappa_{E_i})}
{\sum_{\kappa \in \Omega_\gamma} \P^{G^w,E_i,\bar{1}}(\Phi(\kappa)_{E_i})}
\leq \left(4q^{-\frac18}\right)^{|\gamma|} q^{\frac12}.
\end{align}
\end{proof}
Using correlation inequalities we can derive the following stronger version of the previous theorem.
For a simply connected subgraph $H\subset \Z^2$ we say that $\gamma$ is contained in $H$ if
all faces of plaquettes with center $x^\ast$ for $x^\ast\in \Vertex(\gamma)^\ast$ are contained in $\Edge(H)$.
\begin{corollary}\label{co:contours}
For any $p\leq p_{\sd}$, any simply connected subgraph $H\subset \Z^2$, and a contour $\gamma$ that
is contained in $H$ the probability that $\gamma$ is a $q$-contour can be estimated by
\begin{align}\label{eq:contour_general}
\P^H_p\Big(\text{$\gamma$ is a $q$-contour}\Big)\leq  \left(4q^{-\frac18}\right)^{|\gamma|} q^{\frac12}.
\end{align}  
\end{corollary}
\begin{proof}
We estimate
\begin{align}\label{eq:contour_general2}
\begin{split}
\P^H_p\Big(\text{$\gamma$ is a $q$-contour}\Big)
&
=\P^H_p\Big(\text{$\gamma$ is a $q$-contour}, \text{$\kappa_{b}=1$ for $b\in \Edge(\gamma)$ }\Big)
\\
&
\leq 
\P^H_p\Big(\text{$\gamma$ is a $q$-contour}\mid \text{$\kappa_{b}=1$ for $b\in \Edge(\gamma)$ }\Big)
\\
&
=\P^{H,\Edge(H)\setminus \Edge(\gamma), \bar{1}}_p\Big(\text{$\gamma$ is $q$-contour}\Big).
\end{split}
\end{align}

For the measure $\P^{H,\Edge(H)\setminus  \Vertex(\gamma),\bar{1}}$
 the bonds crossing the contour  are fixed to the correct value. Hence the event that $\gamma$ is a $q$-contour for $\kappa$ is  increasing under this event, such that the stochastic domination results
proved in Lemma \ref{le:mon_p} and Corollary \ref{co:boundary_conditions}
imply that 
\begin{align}\label{eq:contour_general3}
\P^{H,\Edge(H)\setminus \Vertex(\gamma),\bar{1}}_p(\text{$\gamma$ is a $q$-contour})
\leq 
\P^{G^w,\Edge(G),\bar{1}}_{p_{\sd}}(\text{$\gamma$ is $q$-contour})
\end{align}
where $G$ denotes the graph corresponding to $\gamma$ as introduced above Theorem \ref{th:contour}.
 Theorem \ref{th:contour} implies the claim.
\end{proof}

We can now give a new proof for the coexistence result
stated in Theorem \ref{th:main_non_unique}.

\begin{proof}[Proof of Theorem \ref{th:randomconductance_non_unique}]
First we note that the duality between free and wired boundary conditions in finite volume implies that
$\mu^0_{p_{\sd}}$ and $\mu^1_{p_{\sd}}$ are dual to each other in the sense
that if $\kappa\sim \mu^0_{p_{\sd}}$ then $\kappa^\ast\sim \mu^1_{p_{\sd}}$ (on $(\Z^2)^\ast)$).
The proof is the same as for the random cluster model, see, e.g., \cite[Chapter 6]{MR2243761}.
Hence, it is sufficient to show that $\mu^0_{p_{\sd}}(\kappa_e=q)<1/2$ because then we can conclude that
\begin{align}
\discmu_{p_{\sd}}^1(\kappa_e=q)= \discmu_{p_{\sd}}^0(\kappa_e=1) > 1/2
\end{align}
whence $\discmu_{p_{\sd}}^1\neq \discmu_{p_{\sd}}^0$.

Note that if $\kappa_e=q$ and there is any contour $\gamma$ such that $e\in \Edge(\Int(\gamma))$ and
$\kappa_b=1$ for $b\in \Edge(\gamma)$ then there is a $q$-contour surrounding $e$.
We can thus estimate for $e\in E_n$
\begin{align}\label{eq:Peierls_estimate_prep}
\P^{\Lambda_{n+1},E_n,\bar{1}}(\kappa_e=q)\leq
\P^{\Lambda_{n+1},E_n,\bar{1}}(\text{there is a $q$-contour around $e$})
\end{align}
where as before $\bar{1}_e=1$ for all $e$.
The shortest contour $\gamma$ that surrounds the edge $e$
has length 6 so the bound in Corollary \ref{co:contours}
implies that 
$\P^{\Lambda_{n+1},E_n,\bar{1}}(\text{$\gamma$ is a $q$-contour})\leq C/q^{\frac14}$
for any $\gamma$ surrounding $e$.
Using Corollary \ref{co:boundary_conditions} 
we can compare boundary conditions to obtain the relation $\discmu_n^0\precsim \P^{\Lambda_{n+1},E_n,\bar{1}}$.
This and a standard Peierls argument imply for $q$ sufficiently large 
\begin{align}\label{eq:Peierls_estimate}
\discmu_{n,p_{\sd}}^0(\kappa_e=q)\leq \P^{\Lambda_{n+1},E_n,\bar{1}}(\kappa_e=q)
\leq \frac{C}{q^\frac14}\leq\frac14.
\end{align}
Taking the limit $n\to \infty$ we obtain
$\discmu_{p_{\sd}}^0(\kappa_e=q)\leq \tfrac{1}{4}$.
\end{proof}

\appendix
\section{Proofs of Proposition \ref{prop:discmu_Gibbs} and Proposition \ref{prop:discmu_to_extmu}}
\label{sec:proofs_of_propositions}
In this section we pay the last remaining debt of proving two propositions from Section \ref{sec:model}.
\begin{proof}[Proof of Proposition \ref{prop:discmu_Gibbs}]
For $\lambda\in \{1,q\}^{\Edge(\Z^d)}$ 
and $E\subset \Edge(\Z^d)$ finite we define the cylinder event
\begin{align}
\mathbf{A}(\lambda_E)=\{\kappa\in \{1,q\}^{\Edge(\Z^d)}: \kappa_E=\lambda_E\}\in \discsigma_E.
\end{align}
With a slight abuse of notation we drop the pullback from the notation when we consider the set $\pi_2^{-1}(\mathbf{A}(\lambda_E))\subset
\R^{\Edge(\Z^d)}\times \{1,q\}^{\Edge(\Z^d)}$.
Since all local cylinder events in $\discsigma$ can be written as a union of events of the form  
$\mathbf{A}(\lambda_{E_L})$ it is by Remark~\ref{remark:Gibbs_specification}
sufficient to show 
\begin{align}\label{eq:discmu_Gibbs1}
\discmu(\mathbf{A}(\lambda_{E_L}))=\discmu \discgamma_{E_n}(\mathbf{A}(\lambda_{E_L}))
\end{align}
for all $L,n\geq 0$ and all $\lambda\in \{1,q\}^{\Edge(\Z^d)}$. 
Using the quasilocality of $\discgamma$ stated in Corollary~\ref{co:quasilocal}
and Remark 4.21 in \cite{MR2807681}   it is  sufficient to consider $L=n$ and we will do this in the following.
We are going to show the claim in a series of steps.

\textbf{Step 1.}
We investigate the distribution of the $\kappa$-marginal conditioned on 
$\omega{_{E_N^{\crm}}}$.

Since $\extmu$ is a gradient Gibbs measure we know by \eqref{eq:specification_gamma_tilde} that for $\omega\in\R^{\Edge(\Z^d)}_g$
and $\kappa \in \{1,q\}^{\Edge(\Lambda)}$
\begin{align}\label{eq:discmu_Gibbs2}
\extmu(\mathbf{A}(\kappa_{\Edge(\Lambda)}) \mid \gradsigma_{\Edge(\Lambda)^{\crm}})(\omega) 
=\frac{1}{Z} \int 
p^{h(\kappa)} (1-p)^{s(\kappa)} \prod_{e\in \Edge(\Lambda)} e^{-\kappa_e \eta_e^2}\;\nu_\Lambda^{\omega_{\Edge(\Lambda)^{\crm}}} (\d\eta)
=\frac{Z(\kappa,\omega)}{Z}
\end{align}
where $Z$ is the normalisation and
\begin{align}\label{eq:discmu_Gibbs3}
Z(\kappa,\omega)=\int p^{h(\kappa)} (1-p)^{s(\kappa)} \prod_{e\in \Edge(\Lambda)} e^{-\frac12\kappa_e \eta_e^2}\;\nu_\Lambda^{\omega_{\Edge(\Lambda)^{\crm}}} (\d\eta)
\end{align}
 denotes the partition function corresponding to the configuration $\kappa$. 
 Let $\p\in \R^{\Z^d}$ be the configuration such that $\nabla \p=\omega$ and $\p(0)=0$.
We denote by  $\chi_\kappa$  the corrector of $\kappa$,
i.e., the solution of $\nabla^\ast\kappa\nabla \chi_\kappa=0$ with boundary values $\p_{\Lambda^{\crm}}$.
A shift of the integration variables and  Gaussian calculus implies (see also \eqref{eq:motivation_rc_model}) 
 \begin{align}\label{eq:discmu_Gibbs4}
 Z(\kappa,\omega)=Z(\kappa, \bar{0})e^{-\frac12(\nabla\chi_\kappa,\kappa\nabla \chi_\kappa)_{\Edge(\Lambda)}}
 = e^{-\frac12(\nabla\chi_\kappa,\kappa\nabla \chi_\kappa)_{\Edge(\Lambda)}}\frac{p^{h(\kappa)}(1-p)^{s(\kappa)}}{\sqrt{\det 2\pi (\tilde{\Delta}_\kappa^{\Lambda^w})^{-1}}}
\end{align}  
where $\bar{0}$ is the configuration with vanishing gradients, i.e., $\bar{0}_e=0$ for $e\in \Edge(\Z^d)$.
The necessary calculation to obtain \eqref{eq:discmu_Gibbs4} basically agrees with the calculation
that shows that the discrete Gaussian free field can be decomposed in a zero boundary discrete Gaussian free field and a harmonic extension.
We now restrict our attention to $\Lambda=\Lambda_N=[-N,N]^d\cap \Z^d$ for $N\in \mathbb{N}$.
We introduce the law of the $\kappa$-marginal for wired non-constant boundary conditions for
$\kappa\in \{1,q\}^{E_N}$ by
\begin{align}\label{eq:discmu_Gibbs5}
\discmu^{1,\omega}_N(\kappa)=\frac{Z(\kappa,\omega)}{Z}.
\end{align}
Note that $\discmu^{1,\bar{0}}_N=\discmu^{1}_N$ where $\discmu^{1}_N$ was defined in \eqref{eq:def_Gibbs_Lambda_n}.

\textbf{Step 2.} In this step we are going to show that there is $N_0\in \mathbb{N}$ depending on $n$ such that for $N\geq N_0$ and uniformly in $\lambda\in \{1,q\}^{\Edge(\Z^d)}$
\begin{align}\label{eq:result_step2}
\bigg|\discmu\Big(\mathbf{A}(\lambda_{E_n})\mid\mathbf{A}(\lambda_{E_N \setminus E_n})\Big)-\discmu^1_N
\Big(\mathbf{A}(\lambda_{E_n})\mid\mathbf{A}(\lambda_{E_N \setminus E_n})\Big)\bigg|\leq 4\varepsilon. ,
\end{align}
i.e., the boundary effect is negligible.
We start by showing that typically the difference between the corrector energies 
for configurations  $\kappa$ and $\tilde\kappa$ that only differ in $E_n$ will be small. This will allow us 
to estimate the difference between $\discmu^1_N$ and $\discmu^{1,\omega}_N$ conditioned to agree close to the boundary.

Recall that we consider the case that $\Lambda=\Lambda_N$ is a box.
The Nash-Moser estimate stated in Lemma \ref{le:elliptic_Nash_Moser} combined with the maximum principle
for the equation $\nabla^\ast\kappa\nabla \chi_\kappa=0$ 
imply  for $b\in E_n$ and some $\alpha=\alpha(q)>0$
\begin{align}\label{eq:use_of_NashMoser}
|\nabla\chi_\kappa(b)|\leq \frac{C\left( \max_{x\in \partial \Lambda_N}\p(x)-\min_{y\in \partial \Lambda_N}\p(y)\right)}{|N-n|^\alpha}.
\end{align}
We introduce the event $\bs{M}(N)=\{\omega:  \max_{x\in \partial \Lambda_N}\p(x)-\min_{y\in \partial \Lambda_N}\p(y)\leq (\ln N)^3\}$.
Consider 
configurations $\kappa, \tilde\kappa\in \{1,q\}^{\Edge(\Z^d)}$ such
that $\kappa_e=\tilde\kappa_e$ for $e\notin E_n$. Using the fact that 
the corrector is the minimizer of the quadratic form $(\nabla \chi_\kappa,\kappa\nabla \chi_\kappa)_{E_N}$
with given boundary condition we can estimate
\begin{align}\label{eq:discmu_Gibbs6}
(\nabla\chi_\kappa,\kappa\nabla\chi_\kappa)_{E_N}\leq 
(\nabla\chi_{\tilde\kappa},\kappa\nabla\chi_{\tilde\kappa})_{E_N}\leq
(\nabla\chi_{\tilde\kappa},\tilde\kappa\nabla\chi_{\tilde\kappa})_{E_N}
+|E_n|q\sup_{b\in E_n} |\nabla\chi_{\tilde\kappa}|^2.
\end{align}
From \eqref{eq:use_of_NashMoser} we infer that for $N\geq 2n$ and  $\p\in \mathbf{M}(N)$ 
\begin{align}\label{eq:discmu_Gibbs7}
|(\nabla \chi_\kappa,\kappa\nabla\chi_\kappa)_{E_N}
-(\nabla \chi_{\tilde\kappa},\tilde\kappa\nabla\chi_{\tilde\kappa})_{E_N}|
\leq C|E_n|q \frac{(\ln N)^3}{N^\alpha}. 
\end{align}
By choosing $N_1\geq 2n$ sufficiently large  we can
ensure that for $N\geq N_1$, $\p\in \mathbf{M}(N)$,  and uniformly in   $\kappa, \tilde\kappa$ as before
\begin{align}\label{eq:discmu_Gibbs8}
1-\varepsilon\leq e^{\frac12(\nabla \chi_\kappa,\kappa\nabla\chi_\kappa)_{E_N}
-\frac12(\nabla \chi_{\tilde\kappa},\tilde\kappa\nabla\chi_{\tilde\kappa})_{E_N}}\leq 1+\varepsilon.
\end{align}
Using this in \eqref{eq:discmu_Gibbs4}  we conclude that for $N\geq N_1\vee 2n$,  $\omega\in \mathbf{M}(N)$,
 $\varepsilon<1/3$, and $\lambda\in \{1,q\}^{\Edge(\Z^d)}$
\begin{align}\label{eq:bar_mu_omega_vs_bar_mu}
\bigg|\discmu^{1,\omega}_N\Big(\mathbf{A}(\lambda_{E_n})\mid \mathbf{A}(\lambda_{E_N \setminus E_n})\Big)
-\bar{\mu}^1_N\Big(\mathbf{A}(\lambda_{E_n})\mid\mathbf{A}(\lambda_{E_N \setminus E_n})\Big)\bigg|\leq
\frac{2\varepsilon}{1-\varepsilon}\leq  3\varepsilon.
\end{align}
This implies 
\begin{align}\label{eq:bar_mu_omega_vs_bar_mu2}
\left|\extmu\Big(\mathbf{A}(\lambda_{E_n})\mid \mathbf{A}(\lambda_{E_N-E_n})\cap \mathbf{M}(N)\Big)
-\bar{\mu}^1_N\Big(\mathbf{A}(\lambda_{E_n})\mid\mathbf{A}(\lambda_{E_N \setminus E_n})\Big)\right|\leq  3\varepsilon.
\end{align}
From Lemma \ref{le:boundaryGaussian} below and Proposition \ref{prop:phi_given_kappa}
  we infer that for an extended gradient Gibbs measure $\tilde\mu$ associated to an ergodic zero tilt Gibbs measure $\mu$ and any $\lambda\in \{1,q\}^{\Edge(\Z^d)}$
  \begin{align}
  \tilde\mu\Big(M(N)^{\crm} \mid\mathbf{A}(\lambda_{E_N \setminus E_n})\Big)\leq \frac{C}{\ln(N)}\leq \varepsilon
  \end{align}
  for all $N\geq N_2$ and $N_2$ sufficiently large. 
We conclude that for $N\geq N_0\coloneqq N_1\vee N_2\vee 2n$
\begin{align}
\begin{split}\label{eq:bar_mu_omega_vs_bar_mu3}
&\bigg|\discmu\Big(\mathbf{A}(\lambda_{E_n})\mid\mathbf{A}(\lambda_{E_N \setminus E_n})\Big)-\discmu^1_N
\Big(\mathbf{A}(\lambda_{E_n})\mid\mathbf{A}(\lambda_{E_N \setminus E_n})\Big)\bigg|\\
&\quad
\leq 
\tilde{\mu}\Big(\mathbf{M}(N)\mid\mathbf{A}(\lambda_{E_N \setminus E_n})\Big)\;
\bigg|\tilde{\mu}\Big(\mathbf{A}(\lambda_{E_n})\mid\mathbf{A}(\lambda_{E_N\setminus E_n}), \mathbf{M}(N)\Big)-\discmu^1_N
\Big(\mathbf{A}(\lambda_{E_n})\mid\mathbf{A}(\lambda_{E_N \setminus E_n}\Big) \bigg|
\\
&\quad\;+
\tilde{\mu}\Big(\mathbf{M}(N)^{\crm}\mid\mathbf{A}(\lambda_{E_N \setminus E_n})\Big)\;
\bigg|\tilde{\mu}\Big(\mathbf{A}(\lambda_{E_n})\mid\mathbf{A}(\lambda_{E_N\setminus E_n}), \mathbf{M}(N)^{\crm}\Big)-\discmu^1_N
\Big(\mathbf{A}(\lambda_{E_n})\mid\mathbf{A}(\lambda_{E_N \setminus E_n}\Big)\bigg|
\\
&\quad\leq 3\varepsilon+\varepsilon=4\varepsilon.
\end{split}
\end{align}

\textbf{Step 3.} Using the previous results we can now finish the proof.
We rewrite 
\begin{align}
\begin{split}
\discmu\Big(\mathbf{A}(\lambda_{E_n})\Big)
&=\sum_{{\sigma'\in \{1,q\}^{E_N\setminus{E_n}}}}
\discmu\Big(\mathbf{A}(\sigma'_{E_N\setminus E_n})\Big)\discmu\Big(\mathbf{A}(\lambda_{E_n})\mid
\mathbf{A}(\sigma'_{E_N\setminus E_n})\Big)
\\
&=
\sum_{\sigma'\in \{1,q\}^{E_N\setminus E_n}}
\sum_{\sigma\in \{1,q\}^{E_N}}\discmu\Big(\mathbf{A}(\sigma'_{E_N\setminus E_n})\cap  \mathbf{A}(\sigma_{E_N})\Big)\discmu\Big(\mathbf{A}(\lambda_{E_n})\mid
\mathbf{A}(\sigma'_{E_N\setminus E_n})\Big)
\\
&=
\sum_{\sigma\in \{1,q\}^{E_N}}
\discmu\Big(\mathbf{A}(\sigma_{E_N})\Big)\discmu\Big(\mathbf{A}(\lambda_{E_n})\mid
\mathbf{A}(\sigma_{E_N\setminus E_n})\Big).
\end{split}
\end{align}
The  identity above and the fact that $\discgamma_{E_n}$ is proper imply  adding and subtracting the same term 
\begin{align}
\begin{split}\label{eq:discmu_Gibbs12}
&\bigg|\discmu\Big(\mathbf{A}(\lambda_{E_n})\Big)- \discmu\discgamma_{E_n}\Big(\mathbf{A}(\lambda_{E_n})\Big)
\bigg|
\\
&\; =
\bigg|\hspace{-0.cm}
\sum_{\sigma\in \{1,q\}^{E_N}}
\hspace{-0.2cm}
\discmu\Big(\mathbf{A}(\sigma_{E_N})\Big)\discmu\Big(\mathbf{A}(\lambda_{E_n})\mid
\mathbf{A}(\sigma_{E_N\setminus E_n})\Big)
-
\int \discmu(\d\kappa) \mathbb{1}_{\mathbf{A}(\sigma_{E_N})}(\kappa)\; \discgamma_{E_n}\Big(\mathbf{A}(\lambda_{E_n}),\kappa\Big)
\bigg|
\\
&\;\leq 
\bigg|\hspace{-0.cm}
\sum_{\sigma\in \{1,q\}^{E_N}}
\discmu\Big(\mathbf{A}(\sigma_{E_N})\Big)\discmu\Big(\mathbf{A}(\lambda_{E_n})\mid
\mathbf{A}(\sigma_{E_N\setminus E_n})\Big)
-\discmu\Big(\mathbf{A}(\sigma_{E_N})\Big)
\discgamma^{\Lambda_N^w}_{E_n}\Big(\mathbf{A}(\lambda_{E_L}), \sigma_{E_N}\Big)\bigg|
\\ 
&\;\quad +\bigg|
\sum_{\sigma\in \{1,q\}^{E_N}}
\discmu\Big(\mathbf{A}(\sigma_{E_N})\Big)
\discgamma^{\Lambda_N^w}_{E_n}\Big(\mathbf{A}(\lambda_{E_n}), \sigma_{E_N}\Big)-
\int \discmu(\d\kappa) \mathbb{1}_{\mathbf{A}(\sigma_{E_N})}(\kappa)\; \discgamma_{E_n}\Big(\mathbf{A}(\lambda_{E_n}),\kappa\Big)\bigg|.
\end{split}
\end{align}
We continue to estimate the right hand side of this expression.
We start with the first term. Since $\discmu^1_{\Lambda_N}$ is a finite volume Gibbs measure (see \eqref{eq:discmu_Gibbs}) we have for $\mathbf{A}\in \discsigma_{E_N}$
\begin{align}\label{eq:discmu_Gibbs11}
\discmu^1_{\Lambda_N}\Big(\mathbf{A}\mid\mathbf{A}(\sigma_{E_N\setminus E_n})\Big)=
\discmu_{\Lambda_N}^1\Big(\mathbf{A}\mid \discsigma_{E_n^\crm}\Big)(\sigma_{E_N})=
\discgamma^{\Lambda_N^w}_{E_n}(\mathbf{A},\sigma_{E_N}).
\end{align}
Using this and the bound \eqref{eq:result_step2} we obtain for $N\geq N_0$
\begin{align}
\begin{split}\label{eq:discmu_Gibbs41}
\bigg|\hspace{-0.cm}
&\sum_{\sigma\in \{1,q\}^{E_N}}
\discmu\Big(\mathbf{A}(\sigma_{E_N})\Big)\discmu\Big(\mathbf{A}(\lambda_{E_n})\mid
\mathbf{A}(\sigma_{E_N\setminus E_n})\Big)
-\discmu\Big(\mathbf{A}(\sigma_{E_N})\Big)
\discgamma^{\Lambda_N^w}_{E_n}\Big(\mathbf{A}(\lambda_{E_n}), \sigma_{E_N}\Big)\bigg|
\\
&\qquad\leq 
\sum_{\sigma\in \{1,q\}^{E_N}}
\discmu\Big(\mathbf{A}(\sigma_{E_N})\Big)
\bigg|\discmu\Big(\mathbf{A}(\lambda_{E_L})\mid
\mathbf{A}(\sigma_{E_N\setminus E_n})\Big)
-
\discmu^1_{\Lambda_N}\Big(\mathbf{A}(\lambda_{E_N})\mid\mathbf{A}(\sigma_{E_N\setminus E_n})\Big)\bigg|
\\
&\qquad\leq
4\varepsilon \sum_{\sigma\in \{1,q\}^{E_N}}
\discmu\Big(\mathbf{A}(\sigma_{E_N})\Big)\leq 4\varepsilon.
\end{split}
\end{align}

We now address the second term on the right hand side of \eqref{eq:discmu_Gibbs12}.
By Lemma \ref{le:uniform_specification} there is $N_3$ such that for $N\geq N_3$ and any $\lambda,\sigma\in \{1,q\}^{\Edge(\Z^d)}$
\begin{align}\label{eq:discmu_Gibbs10}
|\discgamma_{E_n}(\sigma,\lambda)-\discgamma^{\Lambda_N^w}_{E_n}(\sigma,\lambda)|<\varepsilon.
\end{align}
This implies for $N\geq N_3$
\begin{align}
\begin{split}\label{eq:discmu_Gibbs42}
&\bigg|
\sum_{\sigma\in \{1,q\}^{E_N}}
\discmu\Big(\mathbf{A}(\sigma_{E_N})\Big)
\discgamma^{\Lambda_N^w}_{E_n}\Big(\mathbf{A}(\lambda_{E_n}), \kappa\Big)-
\int \discmu(\d\kappa) \mathbb{1}_{\mathbf{A}(\sigma_{E_N})}(\kappa)\; \discgamma_{E_n}\Big(\mathbf{A}(\lambda_{E_n}),\kappa\Big)\bigg|
\\
&\qquad
\leq
\sum_{\sigma\in \{1,q\}^{E_N}}
\int \discmu(\d\kappa) \mathbb{1}_{\mathbf{A}(\sigma_{E_N})}(\kappa)\;
\left|\discgamma^{\Lambda_N^w}_{E_n}\Big(\mathbf{A}(\lambda_{E_n}), \sigma_{E_N}\Big)-
 \discgamma_{E_n}\Big(\mathbf{A}(\lambda_{E_n}),\kappa\Big)\right|
 \\
&\qquad \leq
  \varepsilon
  \sum_{\sigma\in \{1,q\}^{E_N}} \discmu\Big(\mathbf{A}(\sigma_{E_N})\Big)
\leq \varepsilon.
\end{split}
\end{align}
Using \eqref{eq:discmu_Gibbs12}, \eqref{eq:discmu_Gibbs41}, and \eqref{eq:discmu_Gibbs42}
we conclude that for any $\varepsilon>0$
\begin{align}
\bigg|\discmu\Big(\mathbf{A}(\lambda_{E_n})\Big)- \discmu\discgamma_{E_n}\Big(\mathbf{A}(\lambda_{E_n})\Big)
\bigg|\leq 5\varepsilon.
\end{align}
This ends the proof.
\end{proof}

The following simple Lemma was used in the proof of Proposition \ref{prop:discmu_Gibbs}.
\begin{lemma}\label{le:boundaryGaussian}
Let $\lambda\in \{1,q\}^{\Edge(\Z^d)}$ and denote by $\p^\lambda$ the centred Gaussian field
on $\Z^d$ with $\p(0)=0$ and covariance $\Delta_\lambda^{-1}$. Then $\p^\lambda$ satisfies
\begin{align}
\P\Big(\max_{x\in \partial\Lambda_N}\p^\lambda(x)-\min_{y\in \partial\Lambda_N} \p^\lambda(y)\geq (\ln N)^3\Big)\leq C(\ln N)^{-1}.
\end{align}
\end{lemma}
\begin{proof}
We use the notation $\bar{1}\in \{1,q\}^{\Edge(\Z^d)}$ for the configuration given by $\bar{1}_e=1$ for $e\in \Edge(\Z^d)$.
The Brascamp-Lieb inequality (see \cite[Theorem 5.1]{MR0450480}) implies for the centred
Gaussian fields $\p^\lambda$ and $\p^{\bar{1}}$ that 
\begin{align}
\E\left( (\p^\lambda(x)-\p^\lambda(0))^2\right)\leq \E\left( (\p^{\bar{1}}(x)-\p^{\bar{1}}(0))^2\right)\leq 
\begin{cases}
C\ln(|x|) \quad &\text{for $d=2$}\\
C \quad &\text{for $d\geq 3$}.
\end{cases}
\end{align}
It is well known that for a centred Gaussian random vector $X\in \Pcal(\R^m)$
with $\mathbb{E}(X_i^2)\leq \sigma^2$ 
the expectation of the maximum is bounded by
\begin{align}
\mathbb{E}(\max_i X_i)\leq \sigma\sqrt{2\ln m}.
\end{align}
We use this for the Gaussian field  $\p^\lambda$ and conclude that
\begin{align}
\E\left(\max_{x\in \partial \Lambda_N} \p^\lambda(x)-\p^\lambda(0)\right)\leq 
\begin{cases}
C\ln(N)^2 \quad \text{for $d=2$}\\
C \ln(N^{d-1})\quad \text{for $d\geq 3$}.
\end{cases}
\end{align}
A simple Markov bound implies that there is $C=C(d)>0$ such that
\begin{align}
\P\left(\max_{x\in \partial\Lambda_N}\p^\lambda(x)-\min_{y\in \partial\Lambda_N} \p^\lambda(y)\geq (\ln N)^3\right)\leq C(\ln N)^{-1}.
\end{align}
\end{proof}

It remains to provide a proof of Proposition \ref{prop:discmu_to_extmu}.
We will only sketch the argument.
\begin{proof}[Proof of Proposition \ref{prop:discmu_to_extmu}]
First we remark that the law of $(\kappa, \nabla\p^\kappa)$ is a Borel-measure on $\{1,q\}^{\Edge(\Z^d)}\times \R_g^{\Edge(\Z^d)}$.
This follows  from Carathéodory's extension theorem and the observation that for a local
event $\mathbf{A}\in \gradsigma_E$ with $E\subset \Edge(\Z^d)$ finite
the function $\kappa\mapsto \mu_{\p^\kappa}(\mathbf{A})$ is continuous (this can be shown using Lemma \ref{le:uniform_Laplace_inverse}). 
By Remark~\ref{remark:Gibbs_specification} it is sufficient to prove that $\extmu\extgamma_{\Lambda_n}=\extmu$ for all $n$.
To prove this we use an approximation procedure. We fix $n$ and define for $N>n$ a measure $\extmu_N$ on $\R^{\Edge(\Z^d)}_g\times \{1,q\}^{\Edge(\Z^d)}$
as follows. The $\kappa$-marginal of $\extmu_N$ is given by 
 $\discmu_N=\discmu\discgamma_{E_n}^{\Lambda_N^w}$
 where as before we
 extended $\discgamma_{E_n}^{\Lambda_N^w}$ to a proper probability kernel
 on $\{1,q\}^{\Edge(\Z^d)}$.
 For given $\kappa$, let $\p^\kappa$ be the centred Gaussian field with zero boundary data outside of $\accentset{\circ}\Lambda_N$
and covariance $(\tilde{\Delta}_{\kappa_{E_N}}^{\Lambda_N^w})^{-1}$
where $\tilde{\Delta}_{\kappa_{E_N}}^{\Lambda_N^w}$ was defined in Section~\ref{sec:model_intro}.
The measure $\extmu_N$ is the joint law of $(\kappa,\p^\kappa)$ where $\kappa$ has law $\discmu_N$.
We claim that for $N>n$
\begin{align}\label{eq:Gibbs_mu_N}
\extmu_N \extgamma_{\Lambda_n}=\extmu_N.
\end{align}
We prove this by showing the statement for the measures $\extmu_N\big(\cdot|\mathbf{A}(\lambda_{E_N\setminus E_n})\big)$ for every configuration $\lambda\in \{1,q\}^{\Edge(\Z^d)}$.
To shorten the notation we write $\extmu_N^{\lambda}=\extmu_N\big(\cdot|\mathbf{A}(\lambda_{E_N\setminus E_n})\big)$.
By definition of $\extmu_N$ 
the $\p$-field conditioned on $\kappa$ has density  $\exp(-\frac12 (\p,\tilde{\Delta}_{\kappa}^{\Lambda_N^w}\p))/\sqrt{\det 2\pi (\tilde{\Delta}_\kappa^{\Lambda_N^w}})^{-1}\,\d\p_{\accentset{\circ}\Lambda_N}$ where $\d\p_\Lambda=\prod_{x\in \Lambda}\d \p_x$ denotes the Lebesgue measure.
This implies  for $\mathbf{B}\in \mc{B}(\R^{\Lambda_N})$ and $\sigma\in \{1,q\}^{\Edge(\Z^d)}$ such that
$\sigma_{E_N\setminus E_n}=\lambda_{E_N\setminus E_n}$
\begin{align}\label{eq:extmu_claim2}
\begin{split}
\extmu_N^\lambda\Big(\p\in \mathbf{B}, \kappa\in \mathbf{A}(\sigma_{E_N})\Big)
=
\extmu_N^\lambda\Big(\mathbf{A}(\sigma_{E_N})\Big)
\int_{\mathbf{B}}\frac{\exp(-\frac12 (\p,\tilde\Delta_{\sigma}^{\Lambda_N^w}\p))}{\sqrt{\det 2\pi (\tilde\Delta_\sigma^{\Lambda_N^w})^{-1}}}\,\d\p_{\accentset{\circ}\Lambda_N}
\end{split}
\end{align}
We use the definition of  $\extmu_N$ and the fact that specifications are proper to rewrite
\begin{align}\label{eq:extmu_claim3}
\begin{split}
\extmu_N^\lambda\Big(\mathbf{A}(\sigma_{E_N})\Big)
&=
\frac{\discmu\discgamma_{E_n}^{\Lambda_N^w}\Big(\mathbf{A}(\sigma_{E_N})\cap \mathbf{A}(\lambda_{E_N\setminus E_n})\Big)}{\discmu\discgamma_{E_n}^{\Lambda_N^w}\Big(\mathbf{A}(\lambda_{E_N\setminus E_n})\Big)}
=\frac{\discmu\Big(\mathbb{1}_{\mathbf{A}(\lambda_{E_N\setminus E_n})}(\kappa)\discgamma_{E_n}^{\Lambda_N^w}(\mathbf{A}(\sigma_{E_N},\kappa)\Big)}{\discmu\Big(\mathbf{A}(\lambda_{E_N\setminus E_n})\Big)}
\\
&=\discgamma_{E_n}^{\Lambda_N^w}(\sigma_{E_N},\lambda_{E_N})
=\mathbb{1}_{\sigma_{E_N\setminus E_n}=\lambda_{E_N\setminus E_n}} \frac{1}{Z_\lambda} \frac{p^{h(\sigma,E_N)}(1-p)^{s(\sigma,E_N)}}{\sqrt{\det \Delta_\sigma^{\Lambda_N^w}}}.
\end{split}
\end{align}
Note that 
\begin{align}
p^{h(\sigma,\{e\})}(1-p)^{s(\sigma,\{e\})}e^{-\frac{\sigma_e\eta_e^2}{2}}=\int_{\{\sigma_e\}}
\rho(\d\kappa_e)\, e^{-\frac{\kappa_e\eta_e^2}{2}}.
\end{align}
The last three displays, a summation by parts, and \eqref{eq:relation_Laplacians} lead us to
\begin{align}
\begin{split}\label{eq:extmu_claim4}
\extmu_N^\lambda(\p\in \mathbf{B}, \kappa\in \mathbf{A}(\sigma_{E_N}))
&=\frac{1}{(2\pi)^{|\Lambda_N|}|\Lambda_N^w|Z_\lambda} 
\int_{\mathbf{B}}
p^{h(\sigma,E_N)}(1-p)^{s(\sigma,E_N)}\prod_{e\in E_N} e^{-\frac{\sigma_e\eta_e^2}{2}}\,\d\p_{\overcirc\Lambda_N}
\\
&=\frac{1}{Z'_\lambda} 
\int_{\mathbf{B}}\d\p_{\overcirc\Lambda_N}
\int_{\mathbf{A}(\sigma_{E_N})} \prod_{e\in E_N} \rho(\d\kappa_e) 
\,e^{-\frac{\kappa_e\eta_e^2}{2}}.
\end{split}
\end{align}
Combining this with the definition \eqref{eq:specification_gamma_tilde} we conclude that for $\sigma'\in \{1,q\}^{\Edge(\Z^d)}$ such that $\sigma'_{E_N\setminus E_n}=\lambda_{E_N\setminus E_n}$
and $\omega\in \R_g^{\Edge(\Z^d)}$ such that $\omega_{E_N^\crm}=0$
\begin{align}
\begin{split}
\extmu^\lambda_N\Big(\mathbf{B}\times \mathbf{A}(\sigma_{E_N}) \mid &\extsigma_{E_n^c}\Big)
((\omega,\sigma'))
\\
&=
\frac{1}{Z_{\lambda,\omega}} \int_{\mathbf{B}}\,\nu_{\Lambda_n}^{\omega_{E_n^{\crm}}}(\d \eta) \int_{\mathbf{A}(\sigma_{E_N})}
\prod_{e\in E_n} \rho(\d\kappa_e) \prod_{e\notin E_n} \delta_{\sigma'_e}(\d \kappa_e) 
 \prod_{e\in E_N} e^{-\frac{\kappa_e\eta_e^2}{2}}
 \\
&=\extgamma_{\Lambda_n}\Big(\mathbf{B}\times \mathbf{A}(\sigma_{E_N}),(\omega,\sigma')\Big).
\end{split}
\end{align}
This implies $\extmu_N^\lambda\extgamma_{\Lambda_n}=\extmu_N^\lambda$ and \eqref{eq:Gibbs_mu_N} follows directly.

It remains to pass to the limit in equation \eqref{eq:Gibbs_mu_N}, i.e., we show that the right hand side converges in the topology of local convergence to $\extmu$ and the left hand side to $\extmu\extgamma_{\Lambda_n}$ thus finishing the proof.
We only sketch the argument. 
Since $\extgamma(A,\cdot)$ is a measurable, local, and bounded function if $A$ is a local event  it is sufficient
to show that $\extmu_N$ converges to $\extmu$ locally in total variation, that is for every $\Lambda\subset\subset \Z^d$
\begin{align}
\lim_{N\to \infty}\sup_{A\in \extsigma_{\Edge(\Lambda)}} |\extmu_N(A)-\extmu(A)|=0.
\end{align}
Where we used the $\sigma$-algebra $\extsigma_E$ on $\R^{\Edge(\Z^d)}_g\times \{1,q\}^{\Edge(\Z^d)}$
defined in Section \ref{sec:model_intro} as the product of the pullbacks of $\gradsigma_E$ and $\discsigma_E$.
We first consider the $\kappa$-marginals of 
$\tilde\mu_N$ and $\tilde\mu$. They  are given by 
$\discmu_N=\discmu\discgamma_{E_n}^{\Lambda_N^w}$ and $\discmu=\discmu\discgamma_{E_n}$ where we use that 
$\discmu$ is a Gibbs measure.
We can estimate the total
variation of those two measures by
\begin{align}
\begin{split}
\lVert \discmu\discgamma_{E_n}^{\Lambda_N^w}-\discmu\discgamma_{E_n}\rVert_{\mathrm{TV}}
&\leq \sup_{\kappa\in \{1,q\}^{\Edge(\Z^d)}}\sup_{A\subset \{1,q\}^{\Edge(\Z^d)}} 
|\discgamma_{E_n}^{\Lambda_N^w}(A,\kappa)-\discgamma_{E_n}(A,\kappa)|
\\
&\leq 2^{|E_n|} \sup_{\sigma,\kappa\in \{1,q\}^{\Edge(\Z^d)}}|\discgamma_{E_n}^{\Lambda_N^w}(\sigma,\kappa)-\discgamma_{E_n}(\sigma,\kappa)|.
\end{split}
\end{align}
 In the second step we used that the specifications are proper 
thus we can assume  $A\subset \mathbf{A}(\kappa_{E_n^{\crm}})$
and use that $|\mathbf{A}(\kappa_{E_n^{\crm}})|\leq 2^{|E_n|}$.
Using Lemma \ref{le:uniform_specification} we conclude
\begin{align}
\begin{split}\label{eq:TV1}
\lim_{N\to\infty}\lVert \discmu_N-\discmu\rVert_{\mathrm{TV}}
=
\lim_{N\to\infty}\lVert \discmu\discgamma_{E_n}^{\Lambda_N^w}-\discmu\discgamma_{E_n}\rVert_{\mathrm{TV}}
=0.
\end{split}
\end{align}
We address the $\eta$-marginals of the measures $\extmu$ and $\extmu_N$. 
We write $\extmu(\cdot\mid \kappa), \extmu_N(\cdot\mid \kappa)\in \mc{P}(\R_g^{\Edge(\Z^d)})$ 
for the conditional distribution of the $\eta$-field for a given $\kappa\in \{1,q\}^{\Edge(\Z^d)}$.
From the construction this is well defined for every $\kappa$.
We define the centred Gaussian field
$\p^\kappa$ by  $\p^\kappa(0)=0$ and its covariance $(\Delta_\kappa)^{-1}$ and the centred fields
$\p_N^\kappa$ pinned to 0 outside of $\accentset{\circ}\Lambda_N$ with covariance $(\tilde{\Delta}_\kappa^{\Lambda_N^w})^{-1}$ and we denote their gradients by $\eta^\kappa=\nabla\p^\kappa$ and $\eta_N^\kappa=\nabla\p_N^\kappa$.
Note that by definition of $\extmu$ and $\extmu_N$ the law of $\eta^\kappa$ and $\eta_N^\kappa$ coincides with 
$\extmu(\cdot\mid\kappa)$ and $\extmu_N(\cdot \mid \kappa)$. 
Fix an integer $L$. 
We introduce the Gaussian vectors 
$X^\kappa=(\p^\kappa(x)-\p^\kappa(0))_{x\in \Lambda_L}$ and $X_N^\kappa=(\p^\kappa_N(x)-\p_N^\kappa(0))_{x\in \Lambda_L}$.
Note that given $X^\kappa$, $X_N^\kappa$ the gradient field $\eta^\kappa{\restriction_{\Edge(\Lambda_{L})}}$ respectively
$\eta_N^\kappa{\restriction_{\Edge(\Lambda_{L})}}$ can be expressed as a function of $X^\kappa$ and $X^\kappa_N$ respectively.
This implies that 
\begin{align}
\sup_{B\in \gradsigma_{\Edge(\Lambda_{L})}} \big|\extmu_N( B\mid\kappa)-\extmu(B \mid\kappa)\big|\leq \lVert X^\kappa-X_N^\kappa\rVert_{\mathrm{TV}}.
\end{align}
Theorem 1.1 in \cite{devroye2018total} states that the total variation 
distance between two centred Gaussian vectors $Z_1$, $Z_2$ with covariance matrices
$\Sigma_1$ and $\Sigma_2$  can be bounded by $\tfrac{3}{2}|\Sigma_1^{-1}\Sigma_2-\mathbb{1}|_F$
where $|\cdot|_F$ denotes the Frobenius norm.
Using this  theorem and the uniform convergence of the covariance of $\eta_N^\kappa$ to the covariance of $\eta^\kappa$ stated in Lemma \ref{le:uniform_Laplace_inverse} we conclude that 
\begin{align}\label{eq:TV2}
\lim_{N\to\infty}\sup_{B\in \gradsigma_{\Edge(\Lambda_{L})}} \big| \extmu_N(B\mid\kappa)-\extmu(B\mid\kappa)\big|\leq \lim_{N\to \infty}\lVert X^\kappa-X_N^\kappa\rVert_{\mathrm{TV}}=0.
\end{align}
We denote for a set 
$A\in \extsigma$
 and $\kappa\in \{1,q\}^{\Edge(\Z^d)}$
by $A_\kappa$ the intersection of $A$ and the line through $\kappa$, i.e., $A_\kappa=\{\eta\in \R^{\Edge(\Z^d)}\,:\, (\eta,\kappa)\in A\}$.
Using disintegration, \eqref{eq:TV1}, \eqref{eq:TV2}, and the dominated convergence theorem  we estimate
\begin{align}
\begin{split}
&\lim_{N\to \infty}\sup_{A\in \extsigma_{\Edge(\Lambda_{L})}} |\extmu_N(A)-\extmu(A)|
\\
&\quad=
\lim_{N\to \infty}\sup_{A\in \extsigma_{\Edge(\Lambda_{L})}} 
\left|\int\discmu(\d\kappa)\, \extmu( A_\kappa\mid \kappa)
-\int \discmu_N(\d\kappa)\, \extmu_N(A_\kappa\mid \kappa)
\right|
\\
&\quad\leq
\lim_{N\to \infty}\sup_{A\in \extsigma_{\Edge(\Lambda_{L})}} 
\int\discmu(\d\kappa)\, \Big|\extmu( A_\kappa\mid \kappa)-\extmu_N( A_\kappa\mid \kappa)\Big|
\\
&\qquad+\lim_{N\to \infty}\sup_{A\in \extsigma_{\Edge(\Lambda_{L})}} \left|\int \discmu(\d\kappa)\, \extmu_N( A_\kappa\mid \kappa)
-\int \discmu_N(\d\kappa)\, \extmu_N( A_\kappa \mid \kappa)\right|
\\
&\quad \leq\lim_{N\to \infty}
\int\discmu(\d\kappa)\, \lVert X^\kappa-X^\kappa_N\rVert_{\mathrm{TV}}
+\lim_{N\to \infty}\lVert \discmu-\discmu_N\rVert_{\mathrm{TV}}=0.
\end{split}
\end{align}
We conclude that for any local event $A$
\begin{align}
\extmu(A)=\lim_{N\to\infty} \extmu_N(A)=\lim_{N\to \infty} \extmu_N\extgamma_{\Lambda_n}(A)
=\extmu\extgamma_{\Lambda_n}(A).
\end{align}
\end{proof}
\section{Estimates for discrete elliptic equations}\label{app:1}
In this appendix we collect some regularity estimates for discrete elliptic equations.
We consider as before uniformly elliptic $\kappa:\Edge(\Z^d)\to \R_+$ with $0<c_-\leq \kappa_e\leq c_+<\infty$ for
all $e\in \Edge(\Z^d)$. We denote corresponding set of conductances by $M(c_-,c_+)=[c_-,c_+]^{\Edge(\Z^d)}$.

Next we state a discrete version of the well known Nash-Moser estimates for scalar elliptic 
partial differential equations with $L^\infty$ coefficients.
\begin{lemma}\label{le:elliptic_Nash_Moser}
Let $0<c_-<c_+<\infty$, $\Lambda\subset \Z^d$, and $\kappa\in M(c_-,c_+)$. 
Let $u:\Lambda\to \R$ be a solution of
\begin{align}
-\nabla^\ast \kappa\nabla u=0\quad \text{in $\overcirc{\Lambda}$}
\end{align} 
Then there are constants $\alpha=\alpha(c_-,c_+,d)$ and 
$C=C(c_-,c_+,d)$ such that the following estimate holds for $x,y\in \Lambda$  
\begin{align}
|u(x)-u(y)|\leq C\lVert u\rVert_{L^\infty(\Lambda)} \left(\frac{|x-y|}{d(x,\partial\Lambda) \wedge d(y,\partial\Lambda)} \right)^\alpha.
\end{align}
\end{lemma}
\begin{proof}
This is Proposition~6.2 in \cite{MR1425544}.
\end{proof}

Moreover, we state some consequences for the Green's function of uniformly elliptic 
operators in divergence form.
 We define the Green's function $G_\kappa:\Z^d\times \Z^d\to \R$ as the inverse
of $\Delta_\kappa$, i.e., $G_\kappa$ satisfies for $d\geq 3$
\begin{align}
 \Delta_\kappa G_\kappa(\cdot, y)=\delta_y,\quad \lim_{x\to \infty}G_\kappa(x,y)= 0. 
\end{align}
It is well known that such a Green's function does not exist in dimension 2, however the
derivative $\nabla_{x,i}G_\kappa$ does exist in dimension 2, in particular one can make sense of  expressions as $G_\kappa(x_1,y)-G_\kappa(x_2,y)$.  Formally one can define $\nabla G_\kappa$  by adding a mass $m^2$ to the Laplace operator, i.e., consider the Green's function of $\Delta_\kappa+m^2$ and then send $m^2\to 0$. 
The following estimates hold 
for the Green's function.
\begin{lemma}\label{le:GreensBound}
For any $d\geq 3$ and $\kappa\in M(c_-,c_+)$ the estimate 
\begin{align}\label{eq:Greens_direct}
0\leq G_\kappa(x,y)\leq \frac{C}{|x-y|^{d-2}} 
\end{align}
holds where the constant $C$ depends only on $c_-$, $c_+$, and $d$.
Moreover there exist $\alpha>0$ depending on $c_+/c_-$,and $d$ and $C$ depending on $c_-$, $c_+$, and $d$
such that for $d\geq 2$
\begin{align}\label{eq:Greens_Hoelder}
|\nabla_x G_\kappa(x,y)|&\leq \frac{C}{|x-y|^{d-2+\alpha}},
\\
\label{eq:Greens_Hoelder2}
|\nabla_x\nabla_y G_\kappa(x,y)|&\leq \frac{C}{|x-y|^{d-2+2\alpha}}.
\end{align}
\end{lemma}
\begin{proof} 
These  estimates are well known. 
Estimates for the corresponding  parabolic Green's function are
called Nash-Aronson estimates and they can be found, e.g., in Proposition B.3 in 
\cite{MR1872740}. Integrating the bound for the parabolic Green's function implies \eqref{eq:Greens_direct}.
The estimates
\eqref{eq:Greens_Hoelder} and \eqref{eq:Greens_Hoelder2} follow for $d> 2$ from 
\eqref{eq:Greens_direct}  and Lemma~\ref{le:elliptic_Nash_Moser}.
For $d=2$ one can bound the oscillation of the Green's function using Nash-Aronson estimates and the parabolic Nash-Moser estimate. 
In particular as shown, e.g., in \cite[Chapter 8]{MR3932093} there is a constant $C=C(c_-,c_+)$ such that for all $r>0$ 
\begin{align}\label{eq:bound_osc}
\sup_{x,y\in B_{2r}(0)\setminus B_r(0)}|G_\kappa(x,0)-G_{\kappa}(y,0)|\leq C.
\end{align}   
Lemma~\ref{le:elliptic_Nash_Moser} then implies \eqref{eq:Greens_Hoelder} and \eqref{eq:Greens_Hoelder2}.
\end{proof}

The previous results allow us to bound the difference between the Green's function in a set with Dirichlet boundary conditions and the Green's function on whole space.
We define the Green's function $G^{\Lambda^w}_\kappa:\Lambda\times \Lambda\to \R$
with Dirichlet boundary values in finite volume by
\begin{align}
\begin{split}
\Delta_\kappa G^{\Lambda^w}_\kappa(\cdot, y)&=\delta_y \quad \text{in $\overcirc{\Lambda}$},\\
  G^{\Lambda^w}_\kappa(x, y)&=0\quad \text{for $x\in \partial\Lambda$}.
  \end{split}
\end{align}
For clarity we write $G^{\Z^d}_\kappa=G_\kappa$ in the following.
\begin{lemma}\label{le:uniform_Laplace_inverse}
Let $0<c_-<c_+<\infty$ and $R>0$, then 
\begin{align}\label{eq:uniformGreens}
\lim_{n\to \infty} \sup_{\substack{x,y \in B_R(0)\\ 1\leq i,j\leq d}}\sup_{\kappa\in M(c_-,c_+)} 
\left|\left(\delta_{x+e_i}-\delta_x,G^{\Lambda_n^w}_\kappa(\delta_{y+e_j}-\delta_y)\right)
-\left(\delta_{x+e_i}-\delta_x,G^{\Z^d}_\kappa(\delta_{y+e_j}-\delta_y)\right)
\right|=0.
\end{align}  
\end{lemma}
\begin{remark}
Note that the two scalar products can be equivalently written as $\nabla_{x,i}\nabla_{y,j}G_\kappa^{\Lambda_n^w }(x,y)$ and 
$\nabla_{x,i}\nabla_{y,j}G_\kappa^{\Z^d }(x,y)$.
This expression agrees with the gradient correlations of a Gaussian field:
\begin{align}
\mathbb{E}_{(\Delta_\kappa^{\Lambda_n^w })^{-1}}\left(\eta_{x,x+e_i}\eta_{y,y+e_j}\right)
=\nabla_{x,i}\nabla_{y,j}G_\kappa^{\Lambda_n^w }(x,y).
\end{align}
A similar equation holds when $\Lambda_n^w$ is replaced by $\Z^d$. Thus the lemma implies local uniform convergence of the covariance matrix of those two gradient Gaussian fields.
\end{remark}
\begin{proof}
In $d>2$ the difference of the Green's functions can be expressed through the corrector function $\p_{\kappa,n,y}:\Lambda_n\to \R$ that is defined by
\begin{align}
G^{\Lambda_n}_\kappa(\cdot,y)=G^{\Z^d}_\kappa(\cdot,y)-\p_{\kappa,n,y}(\cdot).
\end{align}
Using the definition of the Green's function we obtain that the corrector satisfies
\begin{align}
\begin{split}
\Delta_\kappa\p_{\kappa,n,y}&=0 \quad \text{ in $\overcirc{\Lambda}_n$}\\
\p_{\kappa,n,y}(x)&=G^{\Z^d}_\kappa(x,y)\quad \text{for  $x\in\partial\Lambda_n$}.
\end{split}
\end{align}
The estimate \eqref{eq:Greens_direct} in Lemma \ref{le:GreensBound} 
now implies
\begin{align}\label{eq:sufficient}
|\p_{k,n,y}(z)|\leq \frac{C}{|\mathrm{dist}(y,\partial\Lambda_n)|^{d-2}}
\end{align}
for $z\in \partial \Lambda_n$. By the maximum principle for $\Delta_\kappa$ the bound extends to all $z\in \Lambda_n$. 
 The claim  then follows from
\begin{align}
\begin{split}\label{eq:rewrite_corrector}
&\left(\delta_{x+e_i}-\delta_x,G^{\Z^d}_\kappa(\delta_{y+e_j}-\delta_y)\right)-\left(\delta_{x+e_i}-\delta_x,G^{\Lambda_n^w}_\kappa(\delta_{y+e_j}-\delta_y)\right)\\
&\quad =\p_{\kappa,n,y}(x)+\p_{\kappa,n,y+e_j}(x+e_i)-\p_{\kappa,n,y}(x+e_i)-\p_{\kappa,n,y+e_j}(x)
\\
&\quad =\nabla_i\p_{\kappa,n,y+e_j}(x)-\nabla_i\p_{\kappa,n,y}(x).
\end{split}
\end{align} 
   The extension to dimension $d=2$ is again slightly technical. 
   We can define 
\begin{align}
\p_{\kappa,n,y}(\cdot)=\Big(G^{\Z^d}_\kappa(\cdot,y)-G^{\Z^d}_\kappa(y,y)\Big)-G^{\Lambda_n}_\kappa(\cdot,y).
\end{align} 
 The corrector satisfies
\begin{align}
\begin{split}
\Delta_\kappa\p_{\kappa,n,y}&=0 \quad \text{ in $\overcirc{\Lambda}_n$}\\
\p_{\kappa,n,y}(x)&=G^{\Z^d}_\kappa(x,y)-G^{\Z^d}_\kappa(y,y)\quad \text{for  $x\in\partial\Lambda_n$}.
\end{split}
\end{align}  
Using \eqref{eq:bound_osc} we can bound  for $y\in B_{n/2}(0)$  
   \begin{align}
   \max_{x\in \partial \Lambda_n} \p_{k,n,y}(x)-
   \min_{x\in \partial\Lambda_n} \p_{k,n,y}(x)\leq C.
   \end{align}
 Lemma~\ref{le:elliptic_Nash_Moser}  implies $\nabla\p_{\kappa,n,y}(x)\leq Cn^{-\alpha}$ for $x\in \Lambda_{n/2}$ and we can conclude using  \eqref{eq:rewrite_corrector}.
   \end{proof}

\paragraph{Acknowledgements}
	This work was supported by the  CRC 1060
	{\it The mathematics of emergent effects} and by the Hausdorff Center for Mathematics (GZ
2047/1, Projekt-ID 390685813) through the Bonn International Graduate School of
Mathematics.
	The author would like to thank the Isaac Newton Institute for Mathematical Sciences, Cambridge, for support and hospitality during the programme {\it Scaling limits, rough paths, quantum field theory} where work on this paper was undertaken. This work was supported by EPSRC grant no EP/K032208/1.

\begin{thebibliography}{10}

\bibitem{adams2016strict}
Stefan Adams, Roman Koteck{\'y}, and Stefan M{\"u}ller, \emph{Strict convexity
  of the surface tension for non-convex potentials}, arXiv preprint
  arXiv:1606.09541 (2016).

\bibitem{MR3932093}
Scott Armstrong, Tuomo Kuusi, and Jean-Christophe Mourrat, \emph{Quantitative
  stochastic homogenization and large-scale regularity}, Grundlehren der
  Mathematischen Wissenschaften [Fundamental Principles of Mathematical
  Sciences], vol. 352, Springer, Cham, 2019. \MR{3932093}

\bibitem{MR1825141}
Itai Benjamini, Russell Lyons, Yuval Peres, and Oded Schramm, \emph{Uniform
  spanning forests}, Ann. Probab. \textbf{29} (2001), no.~1, 1--65.
  \MR{1825141}

\bibitem{MR1700749}
Patrick Billingsley, \emph{Convergence of probability measures}, second ed.,
  Wiley Series in Probability and Statistics: Probability and Statistics, John
  Wiley \& Sons, Inc., New York, 1999, A Wiley-Interscience Publication.
  \MR{1700749}

\bibitem{MR2322690}
Marek Biskup and Roman Koteck{\'y}, \emph{{Phase coexistence of gradient
  {G}ibbs states}}, Probab. Theory Related Fields \textbf{139} (2007), no.~1-2,
  1--39. \MR{2322690}

\bibitem{MR2778801}
Marek Biskup and Herbert Spohn, \emph{{Scaling limit for a class of gradient
  fields with nonconvex potentials}}, Ann. Probab. \textbf{39} (2011), no.~1,
  224--251. \MR{2778801}

\bibitem{MR0450480}
Herm~Jan Brascamp and Elliott~H. Lieb, \emph{On extensions of the
  {B}runn-{M}inkowski and {P}r\'{e}kopa-{L}eindler theorems, including
  inequalities for log concave functions, and with an application to the
  diffusion equation}, J. Functional Analysis \textbf{22} (1976), no.~4,
  366--389. \MR{0450480}

\bibitem{MR2905798}
David~C. Brydges and Thomas Spencer, \emph{Fluctuation estimates for
  sub-quadratic gradient field actions}, J. Math. Phys. \textbf{53} (2012),
  no.~9, 095216, 5. \MR{2905798}

\bibitem{myphd}
Simon Buchholz, \emph{Renormalisation in discrete elasticity}, Dissertation,
  University of Bonn, 2019.

\bibitem{MR2976565}
Codina Cotar and Jean-Dominique Deuschel, \emph{{Decay of covariances,
  uniqueness of ergodic component and scaling limit for a class of
  {$\nabla\phi$} systems with non-convex potential}}, Ann. Inst. Henri
  Poincar{\'e} Probab. Stat. \textbf{48} (2012), no.~3, 819--853. \MR{2976565}

\bibitem{MR2470934}
Codina Cotar, Jean-Dominique Deuschel, and Stefan M{\"u}ller, \emph{{Strict
  convexity of the free energy for a class of non-convex gradient models}},
  Comm. Math. Phys. \textbf{286} (2009), no.~1, 359--376. \MR{2470934}

\bibitem{MR1425544}
Thierry Delmotte, \emph{In\'{e}galit\'{e} de {H}arnack elliptique sur les
  graphes}, Colloq. Math. \textbf{72} (1997), no.~1, 19--37. \MR{1425544}

\bibitem{MR2198017}
Thierry Delmotte and Jean-Dominique Deuschel, \emph{On estimating the
  derivatives of symmetric diffusions in stationary random environment, with
  applications to {$\nabla\phi$} interface model}, Probab. Theory Related
  Fields \textbf{133} (2005), no.~3, 358--390. \MR{2198017}

\bibitem{MR3913274}
Jean-Dominique Deuschel, Takao Nishikawa, and Yvon Vignaud, \emph{Hydrodynamic
  limit for the {G}inzburg-{L}andau {$\nabla\varphi$} interface model with
  non-convex potential}, Stochastic Process. Appl. \textbf{129} (2019), no.~3,
  924--953. \MR{3913274}

\bibitem{devroye2018total}
Luc Devroye, Abbas Mehrabian, and Tommy Reddad, \emph{The total variation
  distance between high-dimensional gaussians}, arXiv preprint arXiv:1810.08693
  (2018).

\bibitem{duminil2017lectures}
Hugo Duminil-Copin, \emph{Lectures on the ising and potts models on the
  hypercubic lattice}, arXiv preprint arXiv:1707.00520 (2017).

\bibitem{MR3898174}
Hugo Duminil-Copin, Aran Raoufi, and Vincent Tassion, \emph{Sharp phase
  transition for the random-cluster and {P}otts models via decision trees},
  Ann. of Math. (2) \textbf{189} (2019), no.~1, 75--99. \MR{3898174}

\bibitem{MR2228384}
Tadahisa Funaki, \emph{{Stochastic interface models}}, {Lectures on probability
  theory and statistics}, {Lecture Notes in Math.}, vol. 1869, Springer,
  Berlin, 2005, pp.~103--274. \MR{2228384}

\bibitem{MR1463032}
Tadahisa Funaki and Herbert Spohn, \emph{{Motion by mean curvature from the
  {G}inzburg-{L}andau {$\nabla \phi$} interface model}}, Comm. Math. Phys.
  \textbf{185} (1997), no.~1, 1--36. \MR{1463032}

\bibitem{MR2807681}
Hans-Otto Georgii, \emph{{Gibbs measures and phase transitions}}, second ed.,
  {de Gruyter Studies in Mathematics}, vol.~9, Walter de Gruyter \& Co.,
  Berlin, 2011. \MR{2807681}

\bibitem{MR1872740}
Giambattista Giacomin, Stefano Olla, and Herbert Spohn, \emph{{Equilibrium
  fluctuations for {$\nabla\phi$} interface model}}, Ann. Probab. \textbf{29}
  (2001), no.~3, 1138--1172. \MR{1872740}

\bibitem{MR2243761}
Geoffrey Grimmett, \emph{The random-cluster model}, Grundlehren der
  Mathematischen Wissen{\-}schaften [Fundamental Principles of Mathematical
  Sciences], vol. 333, Springer-Verlag, Berlin, 2006. \MR{2243761}

\bibitem{MR0094682}
Alexander Grothendieck, \emph{R\'{e}sum\'{e} de la th\'{e}orie m\'{e}trique des
  produits tensoriels topologiques}, Bol. Soc. Mat. S\~{a}o Paulo \textbf{8}
  (1953), 1--79. \MR{0094682}

\bibitem{hilger2016scaling}
Susanne Hilger, \emph{Scaling limit and convergence of smoothed covariance for
  gradient models with non-convex potential}, arXiv preprint arXiv:1603.04703
  (2016).

\bibitem{MR0341552}
Richard Holley, \emph{Remarks on the {${\rm FKG}$} inequalities}, Comm. Math.
  Phys. \textbf{36} (1974), 227--231. \MR{0341552}

\bibitem{MR692943}
Harry Kesten, \emph{Percolation theory for mathematicians}, Progress in
  Probability and Statistics, vol.~2, Birkh\"{a}user, Boston, Mass., 1982.
  \MR{692943}

\bibitem{MR3616205}
Russell Lyons and Yuval Peres, \emph{Probability on trees and networks},
  Cambridge Series in Statistical and Probabilistic Mathematics, vol.~42,
  Cambridge University Press, New York, 2016. \MR{3616205}

\bibitem{MR1461951}
Ali Naddaf and Thomas Spencer, \emph{{On homogenization and scaling limit of
  some gradient perturbations of a massless free field}}, Comm. Math. Phys.
  \textbf{183} (1997), no.~1, 55--84. \MR{1461951}

\bibitem{MR0361422}
Livio~C. Piccinini and Sergio Spagnolo, \emph{On the {H}\"{o}lder continuity of
  solutions of second order elliptic equations in two variables}, Ann. Scuola
  Norm. Sup. Pisa (3) \textbf{26} (1972), 391--402. \MR{0361422}

\bibitem{MR2251117}
Scott Sheffield, \emph{{Random surfaces}}, Ast{\'e}risque (2005), no.~304,
  vi+175. \MR{2251117}

\bibitem{MR177430}
Volker Strassen, \emph{The existence of probability measures with given
  marginals}, Ann. Math. Statist. \textbf{36} (1965), 423--439. \MR{177430}

\bibitem{MR1813436}
William~T. Tutte, \emph{Graph theory}, Encyclopedia of Mathematics and its
  Applications, vol.~21, Cambridge University Press, Cambridge, 2001, With a
  foreword by Crispin St. J. A. Nash-Williams, Reprint of the 1984 original.
  \MR{1813436}

\bibitem{MR3982951}
Zichun Ye, \emph{Models of gradient type with sub-quadratic actions}, J. Math.
  Phys. \textbf{60} (2019), no.~7, 073304, 26. \MR{3982951}

\bibitem{MR1766352}
Milo\v{s} Zahradn\'{i}k, \emph{Contour methods and {P}irogov-{S}inai theory for
  continuous spin lattice models}, On {D}obrushin's way. {F}rom probability
  theory to statistical physics, Amer. Math. Soc. Transl. Ser. 2, vol. 198,
  Amer. Math. Soc., Providence, RI, 2000, pp.~197--220. \MR{1766352}

\end{thebibliography}
\providecommand{\bysame}{\leavevmode\hbox to3em{\hrulefill}\thinspace}
\providecommand{\MR}{\relax\ifhmode\unskip\space\fi MR }
\providecommand{\MRhref}[2]{%
  \href{http://www.ams.org/mathscinet-getitem?mr=#1}{#2}
}
\providecommand{\href}[2]{#2}

\end{document}